\documentclass[12pt,reqno]{amsart}
\usepackage{amsmath,amsfonts,amsthm,amssymb,color,tikz}
\usepackage[width=6.5in, left=1in, right=1in, height=8.8in, top=1.1in,bottom=1.1in]{geometry}
\usepackage{hyperref}
\usepackage[T1]{fontenc}
\usepackage{eufrak}
\usepackage{eufrak}
\usepackage{mathrsfs}
\usepackage{euscript}
\usepackage{pdfsync}

\definecolor{dg}{rgb}{0, 0.5, 0}

\newcommand{\be}{\mathbf{E}}
\newcommand{\bp}{\mathbf{P}}

\newcommand{\id}{\mbox{Id}}

\newcommand{\iot}{\int_{0}^{t}}

\newcommand{\var}{\textbf{Var}}
\newcommand{\1}{{\bf 1}}

\newcommand{\W}{\mathscr W}

\newcommand{\lp}{\left(}
\newcommand{\rp}{\right)}
\newcommand{\lc}{\left[}
\newcommand{\rc}{\right]}
\newcommand{\lcl}{\left\{}
\newcommand{\rcl}{\right\}}
\newcommand{\lln}{\left|}
\newcommand{\rrn}{\right|}
\newcommand{\lla}{\left\langle}
\newcommand{\rra}{\right\rangle}

\newcommand{\al}{\alpha}
\newcommand{\del}{\delta}

\newcommand{\ep}{\varepsilon}

\newcommand{\ga}{\gamma}
\newcommand{\gga}{\Gamma}
\newcommand{\ka}{\kappa}
\newcommand{\la}{\lambda}
\newcommand{\laa}{\Lambda}

\newcommand{\si}{\sigma}
\newcommand{\te}{\theta}

\newcommand{\vp}{\varphi}
\def \eref#1{\hbox{(\ref{#1})}}

\newcommand{\beq}{\begin{equation}}
\newcommand{\eeq}{\end{equation}}
\newcommand{\bea}{\begin{eqnarray}}
\newcommand{\eea}{\end{eqnarray}}
\newcommand{\beas}{\begin{eqnarray*}}
\newcommand{\eeas}{\end{eqnarray*}}

\newcommand{\ca}{{\mathcal A}}
\newcommand{\cb}{{\mathcal B}}
\newcommand{\cac}{{\mathcal C}}
\newcommand{\cd}{{\mathcal D}}

\newcommand{\cf}{{\mathcal F}}

\newcommand{\ch}{{\mathcal H}}

\newcommand{\cs}{{\mathcal S}}

\newcommand{\cw}{{\mathcal W}}

\newcommand{\R}{{\mathbb R}}

\newtheorem{theorem}{Theorem}[section]

\newtheorem{definition}[theorem]{Definition}

\newtheorem{hypothesis}[theorem]{Hypothesis}
\newtheorem{lemma}[theorem]{Lemma}
\newtheorem{notation}[theorem]{Notation}

\newtheorem{proposition}[theorem]{Proposition}
\theoremstyle{remark}
\newtheorem{remark}[theorem]{Remark}
\newtheorem{example}[theorem]{Example}

\let\Section=\section
\def\section{\setcounter{equation}{0}\Section}

\title[Stochastic heat equation with multiplicative colored noise]
{Stochastic heat equations   with general  multiplicative \\
Gaussian noises: H\"older continuity and intermittency}

\author[Y. Hu, J. Huang, D. Nualart, S. Tindel]{Yaozhong Hu \and Jingyu Huang \and David Nualart \and Samy Tindel}

\address{Yaozhong Hu, Jingyu Huang and David Nualart: Department of Mathematics, University of Kansas, 405 Snow Hall, Lawrence, Kansas, USA.} 
\email{nualart@math.ku.edu, hu@math.ku.edu, jhuang@math.ku.edu}

\address{Samy Tindel: Institut {\'E}lie Cartan,
Universit\'e de Lorraine, B.P. 239,
54506 Vand{\oe}u\-vre-l{\`e}s-Nancy, France.}
\email{samy.tindel@univ-lorraine.fr}

\thanks{This project has been carried out while S. Tindel was on sabbatical at the University of Kansas. He wishes to express his gratitude to this institution for its warm hospitality.}



\begin{document}
\begin{abstract}
This paper studies  the stochastic heat equation with
multiplicative
  noises: $\frac {\partial u }{\partial t} =\frac  12 \Delta
u  + u \dot{W}$, where $\dot W$ is a   mean zero Gaussian noise and
$u \dot{W}$ is interpreted  both in the sense of Skorohod and
Stratonovich. The existence and uniqueness of the solution are
studied for noises with     general time and spatial covariance
structure. Feynman-Kac formulas for the solutions and for the
moments of the solutions are obtained under   general and different
conditions. These formulas are applied to obtain the H\"older
continuity of the solutions.  They are also applied to obtain
 the  intermittency bounds for the moments of the solutions.
\end{abstract}

\maketitle

\tableofcontents

\section{Introduction}
In this paper we are interested in the stochastic heat equation in
$\mathbb{R}^d$ driven by a general multiplicative centered Gaussian
noise. This equation can be written as
\begin{equation}  \label{e1}
\frac {\partial u }{\partial t} =\frac  12 \Delta u  + u \dot{W},  \quad t>0, x\in \mathbb{R}^d,
\end{equation}
with initial condition $u_{0,x} =u_0(x)$, where $u_0$ is a continuous and bounded function.
In this equation, the notation $\dot{W}$ stands for the partial derivative
$\frac {\partial^{d+1}W}{ \partial t \partial x_1 \cdots \partial x_d} $ (or $\frac {\partial^{d}W}{  \partial x_1 \cdots \partial x_d} $ when the noise does not depend on time),
where $W$ is a random field  formally defined in the next section.
We assume that $\dot{W}$ has a covariance of the form
\[
\be \lc\dot{W}_{t,x} \dot{W}_{s,y}\rc=\gamma(s-t) \, \laa(x-y),
\]
where $\gamma$ and $\laa$ are general nonnegative and nonnegative
definite (generalized) functions satisfying some integrability
conditions. The product appearing in the above equation (\ref{e1})
can be interpreted as an ordinary product of the solution $u_{t,x}$
times the noise
 $\dot{W}_{t,x}$ (which is a distribution). In this case the 
 evolution form of the equation will
 involve
 a  Stratonovich integral (or  pathwise Young integral).  The product
 in  (\ref{e1}) can also be
also interpreted as a Wick product (defined in the next
 section) and in this case the solution satisfies an evolution equation
formulated by using the Skorohod
 integral. We shall consider both of these formulations.
 
\smallskip
 
There has been a widespread interest in the model \eqref{e1} in the recent past, with several motivations for its study:

\smallskip

\noindent $\bullet$
It is one of the basic stochastic partial differential equations (PDEs) one might wish to solve, either by extending It\^o's theory  \cite{Dalang,PZ} or by pathwise techniques \cite{CFO,GT}. These developments are also related to Zakai's equation from filtering theory.

\smallskip

\noindent $\bullet$
It appears naturally in homogenization problems for PDEs driven by highly oscillating stationary random fields (see \cite{GB,HPP,IPP} and references therein). Notice that in this case limit theorems are often obtained through a Feynman-Kac representation of the solution to the heat equation.

\smallskip

\noindent $\bullet$
Equation \eqref{e1} is also related to the KPZ growth model through the Cole-Hopf's transform. In this context, definitions of the equation by means of renormalization and rough paths techniques  have been recently investigated in \cite{GIP,Ha}.

\smallskip

\noindent $\bullet$
There is a strong connexion between equation \eqref{e1} and the partition function of directed and undirected continuum polymers. This link has been exploited in \cite{La,RT} and is particularly present in \cite{ABQ}, where basic properties of an equation of type \eqref{e1} are translated into corresponding properties of the polymer.

\smallskip

\noindent $\bullet$
The multiplicative stochastic heat equation exhibits concentration properties of its energy. This interesting phenomenon is referred to as \emph{intermittency} for the process $u$ solution to~\eqref{e1} (see e.g \cite{CFK,CFJK,CJK,FK,KLMS}), and as a \emph{localization} property for the polymer measure \cite{CH}. The intermittency property for our model is one of the main result of the current paper, and will be developed later in the introduction.

\smallskip

\noindent $\bullet$
Finally, the large time behavior of equation \eqref{e1} also provides some information on the random operator $Lu =\Delta u + \dot W u$. A sample of the related Lyapunov exponent literature is given by \cite{CM,Mo}.

\smallskip

\noindent
Being so ubiquitous, the model \eqref{e1} has thus obviously been the object of numerous studies.

\smallskip 

Indeed, when the noise $\dot W$ is white in time and colored in space, that is, when
$\gamma$ is the Dirac delta function $\delta_0(x)$, there is a huge
literature devoted to our linear stochastic heat equation.  Notice that in this case the stochastic integral involving $\dot W$
is interpreted in an extended It\^o sense.  Starting with the seminal paper
by Dalang \cite{Da}, these equations, even with more general
nonlinearities (namely $u \dot{W}$ in \eref{e1} is replaced by
$\sigma(u) \dot{W}$ for a general nonlinear function $\sigma$), have
received a lot of attention. In this context, the existence and uniqueness of a
solution is guaranteed by the integrability condition
\begin{equation} \label{mu}
\int_{\mathbb{R}^d}  \frac {\mu(d\xi)} {1+|\xi|^2} <\infty,
\end{equation}
where $\mu$ is the Fourier transform of $\laa$. This condition is sharp, in the sense that it is also
necessary  in the case of an additive noise.

\smallskip

Recently, there also has been a growing interest in studying equation \eref{e1} when the
noise is colored in time.  Unlike the case where the noise is white
in time,   one can  no longer make use of the martingale
structure of the noise, and just making sense of the equation offers
new challenges. Recent progresses for some specific
Gaussian noises  include \cite{BT, HN, Song Jian} by means of stochastic analysis methods, and \cite{CFO,GT,DGT} using rough paths arguments.

\smallskip

As mentioned above, we shall focus in this article on intermittency properties for the stochastic heat equation \eref{e1}.  There exist several ways to express this phenomenon, heuristically meaning that the process $u$ concentrates into a few very high peaks. However, all the definitions involve two functions $\{a(t); t\ge 0\}$ and $\{\ell(k); \, k\ge 2\}$ such that $\ell(k)\in
(0, \infty)$ and: 
\begin{equation}\label{eq:def-lyapunov-exponent}
\ell(k):=\limsup_{t\rightarrow \infty} \frac{1}{a(t)}\log\left(\be
\left[|u_{t,x}|^k\right] \right),
\end{equation}
where we assume that the limit above is independent of $x$. In this case, we call $a(t)$ the upper Lyapunov rate  and $\ell(k)$ the upper Lyapunov exponent. The process
$u$ is then called {\it weakly intermittent}  if
\[
\ell(2)>0\,,\quad \hbox{and}\quad \ell(k)<\infty\quad \forall  \
k\ge 2\,.
\]
The computation of the exact value of Lyapunov exponents is difficult in general.  A related property (which  corresponds to the intuitive notion of   intermittency) requires that  for any $k_1>k_2$ the moment of order $k_{1}$ is significantly greater than the moment of order $k_{2}$, or otherwise stated:

\begin{equation}\label{eq:fraction-moments-u-k}
\limsup_{t\rightarrow \infty}\frac{\be ^{1/k_1}
\left[|u_{t,x}|^{k_1}\right]}{\be ^{1/k_2}
\left[|u_{t,x}|^{k_2}\right]}=\infty\,.
\end{equation}
Most of the studies concerning this challenging property involve a white noise in time, and we refer to \cite{BC, CM, FK} for an account on the topic. The recent paper \cite{BalanConus} tackles the problem for a fractional noise in time, with some special (though important) examples of spatial covariance structures, within the landmark of Skorohod equations. In this case the results are confined to weak intermittency, with an upper bound on $L^{k}$ moments obtained  invoking hypercontractivity arguments and lower bounds computed 
only for the $L^2$ norm.

\smallskip

With all those preliminary considerations in mind, the current paper proposes to study existence-uniqueness results, Feynman-Kac representations, chaos expansions and intermittency results for a very wide class of Gaussian noises $\dot W$ (including in particular those considered in \cite{BalanConus, Dalang}), for both Skorohod and Stratonovich type equations \eref{e1}.
In particular  we obtain some lower bounds for $\ell(k)$ defined by \eqref{eq:def-lyapunov-exponent} for all $k\ge 2$,   which are sharp in the sense that they have the same exponential order as the upper
bounds. 

\smallskip

More specifically, here is a brief description of the results obtained in the current paper:

\begin{itemize}
\item[(i)] In the Skorohod case, the mild solution has a formal Wiener chaos expansion,
which converges in  $L^2(\Omega)$ provided $\gamma $ is locally
integrable and the spectral measure $\mu$ of the spatial  covariance
satisfies condition \eref{mu}. Moreover, the solution is unique.
This result (proved in Theorem \ref{thmSk1}) is based on Fourier
analysis techniques, and covers the particular examples of the Riesz
kernel and the Bessel kernel considered by Balan and Tudor in
\cite{BT}.  Our results also encompass the case of the fractional covariance
$\laa(x)=\prod_{i=1}^d H_i (2H_i-1) |x_i|^{2H_i-2}$, where $H_i
>\frac 12$ and condition \eref{mu} is satisfied if and only if
$\sum_{i=1}^d H_i >d-1$. This particular structure has been examined in \cite{HN}.


 \item[(ii)] Under these general hypothesis to  ensure the existence and uniqueness
of the solution of Skorohod type one cannot expect to have a
Feynman-Kac formula for the solution, but one can establish
Feynman-Kac-type formulas for the moments of the solution. The
formulas we obtain (see \eref{momSk1}), generalize those obtained for
the   Riesz or the Bessel kernels in \cite{BT,HN}.

\item[(iii)] Under more restrictive integrability assumptions  on $\gamma$ and $\mu$ (see Hypothesis \ref{hyp:mu2})
we derive a  Feynman-Kac formula for the solution $u$ to \eref{e1} in the Stratonovich sense. An immediate application of the Feynman-Kac formula is the H\"older continuity of the solution.

\item[(iv)] In the Stratonovich case, we give a notion of solution
using two different methodologies. One is based on the Stratonovich
integral defined as the limit in probability of the integrals with
respect to a regularization of the noise, and another one uses a
pathwise approach, weighted Besov spaces and a Young integral approach. 
We show that the two notions coincide and some existence-uniqueness results which are (to the best of our knowledge) the first link between pathwise and Malliavin calculus solutions to equation~\eref{e1}.

\item[(v)] Under some further restrictions (see hypothesis at the
beginning of Section 6), we obtain some sharp lower bounds for the
moments of the solution. Namely, we can find explicit numbers
$\kappa_1$ and $\kappa_2$ and constants $c_{j},C_{j}$ for $j=1,2$ such that
\[
C_{1}  \exp\left(c_{1} t^{\kappa_1}k^{\kappa_2}\right)\le \be
\left[|u_{t,x}|^{k}\right]\le 
C_{2} \exp\left(c_{2}
t^{\kappa_1}k^{\kappa_2}\right)\
\]
for all $t\ge 0$, $x\in \R^d$ and $k\ge 2$.
\end{itemize}
As it might be clear from the description above, our central object for the study of \eref{e1} is  
 the Feynman-Kac formula for the solution $u$ or for its moments, which
is a very interesting result in its own right. A substantial part of the article is devoted to establish these formulae with optimal conditions on the covariances $\gamma$ and $\Lambda$, including critical cases. Notice that we also handle the case of noises which only depend on the space variable. This situation is usually treated separately in the paper,  due to its particular physical relevance.

\smallskip


\smallskip

Here is the organization of the paper. In Section \ref{sec:preliminaries}, we briefly set
up some preliminary material on the Gaussian noises that we are
dealing with. We also recall some material from
Malliavin calculus. Section~\ref{sec:Skorohod} is devoted to the stochastic heat
equation of Skorohod type. Existence and uniqueness of the mild
solutions are obtained, and  Feynman-Kac formula for the moments of the
solution is established.  Section \ref{sec:eq-strato} focuses on the Feynman-Kac formula related to equation \eqref{e1} and studies the regularity of the process $u^{F}$ defined in that way under some conditions on $\gamma$ and $\Lambda$. In section \ref{sec:pathwise-solution} we first prove that the process $u^{F}$ can really be seen as a solution to the stochastic heat equation interpreted in a mild sense related to Malliavin calculus. However, uniqueness is missing in this general context.   Under some slightly more restrictive
conditions on the noises, we then study the existence and uniqueness of
the mild solution to equation~\eqref{e1} using Young integration techniques. Finally, Section~\ref{sec:moment-estimates} is
concerned with the bounds for the moments and related intermittency results.

\paragraph{\textbf{Notations.}}
In the remainder of the article, all generic constants will be denoted by $c,C$, and their value
may vary from different occurrences. We  denote by $p_{t}(x)$ the
$d$-dimensional heat kernel $p_{t}(x)=(2\pi
t)^{-d/2}e^{-|x|^2/2t} $, for any $t > 0$, $x \in \R^d$.

\section{Preliminaries}\label{sec:preliminaries}
This section is devoted to a further description  of the structure
of our noise $W$. We consider first the time dependent case and
later the time independent case. We will also provide some basic
elements of Malliavin calculus used in the paper.

\subsection{Time dependent noise}  \label{sec2.1}
Let us start by introducing some basic notions on Fourier transforms
of functions: the space of real valued infinitely differentiable
functions with compact support is denoted by $\mathcal{D} (
\mathbb{R}^d)$ or $\mathcal{D}$. The space of Schwartz functions is
denoted by $\mathcal{S} ( \mathbb{R}^d)$ or $\mathcal{S}$. Its dual,
the space of tempered distributions, is $\mathcal{S}' (
\mathbb{R}^d)$ or $\mathcal{S}'$. If $u$ is a vector  of tempered
distributions from  $\mathbb{R}^d$ to $\mathbb{R}^n$,  then we write
$u \in \mathcal{S}' (\mathbb{R}^d, \mathbb{R}^n)$. The Fourier
transform is defined with the normalization
\[ \mathcal{F}u ( \xi)  = \int_{\mathbb{R}^d} e^{- \imath \langle
   \xi, x \rangle} u ( x) d x, \]
so that the inverse Fourier transform is given by $\mathcal{F}^{- 1} u ( \xi)
= ( 2 \pi)^{- d} \mathcal{F}u ( - \xi)$.

\smallskip
Similarly to \cite{Da}, on a complete probability space
$(\Omega,\cf,\bp)$ we consider a Gaussian noise $W$ encoded by a
centered Gaussian family $\{W(\vp) ; \, \vp\in
\mathcal{D}([0,\infty)\times \R^{d})\}$, whose covariance structure
is given by
\begin{equation}\label{cov1}
\be\lc W(\vp) \, W(\psi) \rc
= \int_{\R_{+}^{2}\times\R^{2d}}
\varphi(s,x)\psi(t,y)\gamma(s-t)\laa(x-y)dxdydsdt,
\end{equation}
where $\gamma: \R \rightarrow \R_+$ and $\laa: \R^d \rightarrow
\R_+$ are  non-negative definite
functions and the Fourier transform $\cf\laa=\mu$ is a tempered
measure, that is, there is an integer $m \geq 1$ such that
$\int_{\R^d}(1+|\xi|^2)^{-m}\mu(d\xi)< \infty$.

\smallskip

 Let $\mathcal{H}$  be the completion of
$\mathcal{D}([0,\infty)\times\R^d)$
endowed with the inner product
\begin{eqnarray}\label{innprod1}
\langle \varphi , \psi \rangle_{\mathcal{H}}&=&
\int_{\R_{+}^{2}\times\R^{2d}}
\varphi(s,x)\psi(t,y)\gamma(s-t)\laa(x-y) \, dxdydsdt\\ \notag
&=&\int_{\R_{+}^{2}\times\R^{d}}  \cf \varphi(s,\xi) \overline{ \cf \psi(t,\xi)}\gamma(s-t) \mu(d\xi) \, dsdt,
\end{eqnarray}
where $\cf \varphi$ refers to the Fourier transform with respect to the space variable only.
 The mapping $\varphi \rightarrow W(\varphi)$ defined in $\mathcal{D}([0,\infty)\times\R^d)$  extends to a linear isometry between
$\mathcal{H}$  and the Gaussian space
spanned by $W$. We will denote this isometry by
\begin{equation*}
W(\phi)=\int_0^{\infty}\int_{\R^d}\phi(t,x)W(dt,dx)
\end{equation*}
for $\phi \in \mathcal{H}$.
Notice that if $\phi$ and $\psi$ are in
$\mathcal{H}$, then
$\be \lc W(\phi)W(\psi)\rc =\langle\phi,\psi\rangle_{\mathcal{H}}$. Furthermore, $\mathcal{H}$  contains
the class of measurable functions $\phi$ on $\R_+\times
\R^d$  such that
\begin{equation}\label{abs1}
\int_{\R^2_+  \times\R^{2d}} |\phi(s,x)\phi(t,y)| \, \gamma(s-t)\laa(x-y) \, dxdydsdt <
\infty\,.
\end{equation}

\smallskip

We shall make a standard assumption on the spectral measure $\mu$, which will prevail until the end of the paper.
\begin{hypothesis}\label{hyp:mu}
The measure $\mu$ satisfies the following integrability condition:
\begin{equation} \label{mu1}
\int_{\R^{d}}\frac{\mu(d\xi)}{1+|\xi|^{2}} <\infty.
\end{equation}
\end{hypothesis}

 Let us now recall some of the main examples of stationary covariances, which will be our guiding examples in the remainder of the paper.

\begin{example}\label{ex:riesz-noise}
One of the most popular spatial  covariances is given by the
so-called Riesz kernel, for which $\laa(x)=|x|^{-\eta}$ and
$\mu(d\xi)=c_{\eta,d}|\xi|^{-(d-\eta)} \, d\xi$. We refer to this kind of noise as a spatial
$\eta$-Riesz noise. In this case, Hypothesis \ref{hyp:mu} is
satisfied whenever $0<\eta<2$, which will be our standing assumption.
   \end{example}

\begin{example}\label{ex:white-noise}
We shall also handle the space white noise case, namely
$\laa=\delta_{0}$ (notice that in this case $\laa$ is not a function
 but a measure)  and $\mu=\text{Lebesgue}$. This noise satisfies
Hypothesis~\ref{hyp:mu} only in dimension 1.
\end{example}

\begin{example}\label{ex:bessel-noise}
The spatial  covariance given by the so-called Bessel kernel is
defined by
\[
\laa(x)=\int_0^{\infty}w^{\frac{\eta-d}{2}}e^{-w}e^{-\frac{|x|^2}{4w}}dw\,.
\]
In this case
   $\mu(d\xi)=c_{\eta,d} (1+|\xi|^2)^{-\frac{\eta}{2}} d\xi$ and Hypothesis \ref{hyp:mu} is satisfied  if $\eta> d-2$.
\end{example}

\begin{example}\label{ex:fractional-noise in time}
An example of time covariance $\gamma$ that has received a lot of
attention is the case of a one-dimensional Riesz kernel, which
corresponds to the fractional Brownian motion. Suppose that
$\gamma(t)= H(2H-1) |t|^{2H-2}$ with $\frac{1}{2}< H < 1$ and  $W$
is a noise with this time covariance and a spatial  covariance
$\laa$. For any $t\ge 0$ and any $\varphi \in C^\infty_c (\R^d)$,
the function $\mathbf{1}_{[0,t]} \varphi$ belongs to the space
$\mathcal{H}$,  and we can define  $W_t(\varphi):=
W(\mathbf{1}_{[0,t]} \varphi)$.  Then,  for any fixed $  \varphi \in
\mathcal{D} (\R^d)$,  the stochastic process $\{c_{\varphi} ^{-1/2}W_t(\varphi); t\ge 0\}$ is a fractional Brownian motion with
Hurst parameter $H$, where
\[
c_\varphi= \int_{\R^d} |\cf \varphi(\xi) |^2 \mu(d\xi).
\]
That is $\be \lc W_t(\varphi)W_s(\varphi) \rc= R_H(s,t) c_\varphi$,  where 
for each $H\in (0,1)$ we have:
\[
R_H(s,t)=\frac{1}{2}\lp |s|^{2H}+|t|^{2H}-|s-t|^{2H}\rp.
\]
\end{example}

\begin{example}\label{ex:fractional-noise in space}
In the same way, the spatial  fractional covariance is given by
 $\laa(x)= \prod_{i=1}^d   H_i$ $(2H_{i}-2)|x_i|^{2H_i-2}$, where $\frac{1}{2}< H_i < 1$ for
$i=1,\dots, d$.   The Fourier transform of $\laa$ is $\mu(d\xi)=C_H
\prod_{i=1}^d |\xi_i|^{1-2H_i}d\xi$, where $C_H$ is a constant
depending on the parameters $H_i$. Then an easy calculation shows
that when $\sum_{i=1}^d H_i > d-1$,  Hypothesis \ref{hyp:mu} holds.
\end{example}

If $W$ is a noise with fractional space and time covariances, with
Hurst parameters $H_0$ in time, and  $H_1, \dots, H_d$ in space,
then we can write formally $W(\varphi)$ as the distributional integral $\int_{\R_+ \times \R^d}
\varphi(t,x) \dot{W}_H(t,x) dt dx$, where $\dot{W}_H(t,x)$ is the
formal partial derivative $\frac{\partial^{d+1}W_H}{\partial t
\partial x_1 \cdots \partial x_d}(t,x)$ and $W_H$ is centered
Gaussian random field which is a fractional Brownian motion on each
coordinate, that is,
\begin{equation*}
\be \lc W_H(s,x)W_H(t,y)\rc=R_{H_0}(s,t)\prod_{i=1}^d
R_{H_i}(x_i,y_i)\,, \quad s,t \geq 0, x, y \in \R^d\,.
\end{equation*}

\subsection{Time independent noise} \label{sec2.2}  In this case we consider a zero mean Gaussian family
$W=\{W(\varphi); \varphi \in \mathcal{D}(\R^d)\}$, defined in a complete probability space
$(\Omega,\mathcal{F},\bp)$, with covariance
\begin{equation}\label{cov2}
\be \lc W (\varphi) W (\psi)\rc=\int_{
\R^{2d}}\varphi(x)\psi(y) \laa (x-y) \, dxdy\,,
\end{equation}
where, as before, $\laa :
\R^d \rightarrow \R_+$  is a non-negative definite function whose Fourier transform $\mu$ is a tempered measure.
In this case  $\mathcal{H}$  is the completion of
$\mathcal{D}(\R^d)$
endowed with the inner product
\begin{equation}\label{innprod2}
\langle \varphi , \psi \rangle_{ \mathcal{H}} =\int_{\R^{2d}}
\varphi(x)\psi(y) \laa (x-y)dxdy= \int_{\R^d} \cf\varphi(\xi) \overline{\cf\psi(\xi)} \mu(d\xi).
\end{equation}
The mapping $\varphi \rightarrow W(\varphi)$ defined in $\mathcal{D}( \R^d)$  extends to a linear isometry between
$\mathcal{H}$  and the Gaussian space
spanned by $W$, denoted   by
\begin{equation*}
 W (\phi)=\int_{\R^d}\phi(x)W(dx)
\end{equation*}
for $\phi \in \mathcal{H}$.
If  $\phi$ and $\psi$ are in
$\mathcal{H}$, then
$\be \lc W(\phi)W(\psi)\rc =\langle\phi,\psi\rangle_{\mathcal{H}}$ and $\mathcal{H}$  contains
the class of measurable functions $\phi$ on  $\R^d$  such that
\begin{equation}\label{abs2}
\int_{\R^{2d}}|\phi(x)\phi(y)| \, \laa(x-y) \, dxdy < \infty\,.
\end{equation}

\subsection{Elements of Malliavin calculus}
Consider first the case of a time dependent noise.
We will denote by $D$ the derivative operator in the sense of
Malliavin calculus. That is, if $F$ is a smooth and cylindrical
random variable of the form
\begin{equation*}
F=f(W(\phi_1),\dots,W(\phi_n))\,,
\end{equation*}
with $\phi_i \in \mathcal{H}$, $f \in C^{\infty}_p (\R^n)$ (namely $f$ and all
its partial derivatives have polynomial growth), then $DF$ is the
$\mathcal{H}$-valued random variable defined by
\begin{equation*}
DF=\sum_{j=1}^n\frac{\partial f}{\partial
x_j}(W(\phi_1),\dots,W(\phi_n))\phi_j\,.
\end{equation*}
The operator $D$ is closable from $L^2(\Omega)$ into $L^2(\Omega;
\mathcal{H})$  and we define the Sobolev space $\mathbb{D}^{1,2}$ as
the closure of the space of smooth and cylindrical random variables
under the norm
\[
\|DF\|_{1,2}=\sqrt{\be[F^2]+\be[\|DF\|^2_{\mathcal{H}}]}\,.
\]
We denote by $\delta$ the adjoint of the derivative operator given
by the duality formula
\begin{equation}\label{dual}
\be \lc \delta (u)F \rc =\be \lc \langle DF,u
\rangle_{\mathcal{H}}\rc , 
\end{equation}
for any $F \in \mathbb{D}^{1,2}$ and any element $u \in L^2(\Omega;
\mathcal{H})$ in the domain of $\delta$. The operator $\delta$ is
also called the {\it Skorohod integral} because in the case of the
Brownian motion, it coincides with an extension of the It\^o
integral introduced by Skorohod. We refer to Nualart \cite{Nualart2}
for a detailed account of the Malliavin calculus with respect to a
Gaussian process. If $DF$ and $u$ are almost surely measurable
functions on $\R_+ \times \R^d$  verifying condition \eref{abs1},
then the duality formula~\eref{dual} can be written  using the
expression of the inner product in $\mathcal{H}$  given in
\eref{innprod1}
\begin{equation} \label{eq1}
\be \lc \delta(u)F \rc =
\be \lc \int_{\R^2_+\times\R^{2d}}D_{s,x}F \, u_{t,y} \, \gamma(s-t) \, \laa(x-y) \, dsdtdxdy \rc.
\end{equation}

\smallskip

Let us recall 3 other classical relations in stochastic analysis, which will be used in the
paper:

\smallskip

\noindent
\emph{(i) Divergence type formula.} For any $\phi \in \mathcal{H}$ and any random
variable $F$ in the Sobolev space $\mathbb{D}^{1,2}$, we have
\begin{equation}  \label{eq2}
FW(\phi)=\delta(F\phi)+\langle DF,
\phi\rangle_{\mathcal{H}}.
\end{equation}

\smallskip

\noindent
\emph{(ii) A duality relationship.}
Given a random variable $F\in \mathbb{D}^{2,2}$ and two elements $h,g \in \mathcal{H}$, the duality formula
(\ref{dual}) implies
\begin{equation} \label{k1}
\be \lc F \, W(h)W(g) \rc= \be \lc \langle D^2F, h\otimes g \rangle_{\mathcal{H}^{\otimes 2}} \rc
+ \be\lc F  \rc \, \langle h,g \rangle _{\mathcal{H}}.
\end{equation}

\smallskip

\noindent
\emph{(iii) Definition of the Wick product of a random and a Gaussian element.}
If $F\in \mathbb{D}^{1,2}$ and $h$ is an element of $%
\mathcal{H}$, then $Fh$ is Skorohod
integrable and, by definition, the
Wick product equals to the Skorohod integral of $Fh$
\begin{equation}
\delta (Fh)=F\diamond W(h).  \label{Wick}
\end{equation}%

\smallskip

When handling the stochastic heat equation in the Skorohod sense we will make use of chaos expansions, and we should give a small account on this notion.
For any integer $n\ge 0$ we denote by $\mathbf{H}_n$ the $n$th Wiener chaos of $W$. We recall that $\mathbf{H}_0$ is simply  $\R$ and for $n\ge 1$, $\mathbf {H}_n$ is the closed linear subspace of $L^2(\Omega)$ generated by the random variables $\{ H_n(W(h));h \in \mathcal{H}, \|h\|_{\mathcal{H}}=1 \}$, where $H_n$ is the $n$th Hermite polynomial.
For any $n\ge 1$, we denote by $\mathcal{H}^{\otimes n}$ (resp. $\mathcal{H}^{\odot n}$) the $n$th tensor product (resp. the $n$th  symmetric tensor product) of $\mathcal{H}$. Then, the mapping $I_n(h^{\otimes n})= H_n(W(h))$ can be extended to a linear isometry between    $\mathcal{H}^{\odot n}$ (equipped with the modified norm $\sqrt{n!}\| \cdot\|_{\mathcal{H}^{\otimes n}}$) and $\mathbf{H}_n$.

\smallskip

Consider now a random variable $F\in L^2(\Omega)$ measurable with respect to the $\sigma$-field  $\mathcal{F}^W$ generated by $W$. This random variable can be expressed as
\begin{equation}\label{eq:chaos-dcp}
F= \be \lc F\rc + \sum_{n=1} ^\infty I_n(f_n),
\end{equation}
where the series converges in $L^2(\Omega)$, and the elements $f_n \in \mathcal{H}^{\odot n}$, $n\ge 1$, are determined by $F$.  This identity is called the Wiener-chaos expansion of $F$.

\smallskip

The Skorohod integral (or divergence) of a random field $u$ can be
computed by  using the Wiener chaos expansion. More precisely,
suppose that $u=\{u_{t,x} ; (t,x) \in \R_+ \times\R^d\}$ is a random
field such that for each $(t,x)$, $u_{t,x}$ is an
$\mathcal{F}^W$-measurable and square integrable random  variable.
Then, for each $(t,x)$ we have a Wiener chaos expansion of the form
\begin{equation}  \label{exp1}
u_{t,x} = \be \lc u_{t,x} \rc + \sum_{n=1}^\infty I_n (f_n(\cdot,t,x)).
\end{equation}
Suppose also that
\[
\be \lc \int_0^\infty \int_0^\infty \int_{\R^{2d}}  |u_{t,x} \, u_{s,y}
| \, \gamma( s-t) \laa(x-y) \, dxdydsdt \rc <\infty.
\]
Then, we can interpret $u$ as a square  integrable
random function with values in $\mathcal{H}$ and the kernels $f_n$
in the expansion (\ref{exp1}) are functions in $\mathcal{H}
^{\otimes (n+1)}$ which are symmetric in the first $n$ variables. In
this situation, $u$ belongs to the domain of the divergence (that
is, $u$ is Skorohod integrable with respect to $W$) if and only if
the following series converges in $L^2(\Omega)$
\begin{equation}\label{eq:delta-u-chaos}
\delta(u)= \int_0 ^\infty \int_{\R^d}  u_{t,s} \, \delta W_{t,x} = W(\be[u]) + \sum_{n=1}^\infty I_{n+1} (\widetilde{f}_n(\cdot,t,x)),
\end{equation}
where $\widetilde{f}_n$ denotes the symmetrization of $f_n$ in all its $n+1$ variables.

\smallskip

The operators $D$ and $\delta$ can be introduced in a similar way in the time independent case.  If
 $DF$ and $u$ are almost surely measurable functions on $\mathbb{R}^d$ verifying condition~(\ref{abs2}), then formula
\eref{dual} can be written using the expression of the inner product
in $\mathcal{H}$  given in \eref{innprod2}:
\begin{equation}\label{eq:duality-time-indep}
\be \lc \delta(u)F\rc =\be \lc \int_{\R^{2d}}D_{x}F \, u(y) \, \laa (x-y) \, dxdy\rc\,.
\end{equation}

\section{Equation of Skorohod type}\label{sec:Skorohod}
The first part of this section is devoted to  the study of the
following $d$-dimensional stochastic heat equation  with the  time
dependent multiplicative Gaussian noise $W$ introduced in Section
\ref{sec2.1}, where the product is understood in the Wick sense (see
(\ref{Wick})):
\begin{equation}\label{eqSk1}
\frac{\partial u}{\partial t}=\frac{1}{2}\Delta u+u\diamond
\frac{\partial^{d+1}W}{\partial t \partial x_1 \cdots \partial
x_d}\, ,
\end{equation}
the initial condition being a continuous and bounded function
$u_0(x)$. This equation is formal and below  we provide a rigorous
definition of a mild solution   using the Skorohod integral. The
main objective of this section is to show that the  mild solution
exists and is unique in $L^2(\Omega)$, assuming that the spectral
measure $\mu$ satisfies Hypothesis \ref{hyp:mu}. This is proved by
showing that the formal Wiener chaos expansion which defines the
solution converges in $L^2(\Omega)$. In a second part of this
section we obtain a Feynman-Kac formula for the moments of the
solution. In the last part we will extend these results to the case
where the noise is time independent.

\subsection{Existence and uniqueness of a solution via chaos expansions}
\label{sec:ex-uniq-chaos}
Recall that we  denote by $p_{t}(x)$ the
$d$-dimensional heat kernel $p_{t}(x)=(2\pi
t)^{-d/2}e^{-|x|^2/2t} $, for any $t > 0$, $x \in \R^d$.
For each $t\ge 0$ let $\mathcal{F}_t$ be the $\sigma$-field generated by the random variables $W(\varphi)$, where
$\varphi$ has support in $[0,t ]\times \mathbb{R}^d$. We say that a random field $u_{t,x}$ is adapted if for each $(t,x)$ the random variable $u_{t,x}$ is $\mathcal{F}_t$-measurable.
We define the solution of equation  \eref{eqSk1} as follows.

\begin{definition}\label{def1}
An adapted   random field $u=\{u_{t,x}; t \geq 0, x \in
\mathbb{R}^d\}$ such that $\be [ u^2_{t,x}] < \infty$ for all $(t,x)$ is
a mild solution to equation \eref{eqSk1} with initial condition $u_0 \in C_b (\mathbb{R}^d)$, if for any $(t,x) \in [0,
\infty)\times \mathbb{R}^d$, the process $\{p_{t-s}(x-y)u_{s,y}{\bf
1}_{[0,t]}(s); s \geq 0, y \in \mathbb{R}^d\}$ is Skorohod
integrable, and the following equation holds
\begin{equation}\label{eq:sko-mild}
u_{t,x}=p_t u_0(x)+\int_0^t\int_{\mathbb{R}^d}p_{t-s}(x-y)u_{s,y} \, \delta
W_{s,y}.
\end{equation}
\end{definition}

Suppose now that $u=\{u_{t,x}; t\geq 0, x \in \R^d\}$ is a solution to equation \eref{eq:sko-mild}. Then according to \eref{eq:chaos-dcp}, for any fixed $(t,x)$ the random variable $u_{t,x}$ admits the following Wiener chaos expansion
\begin{equation}
u_{t,x}=\sum_{n=0}^{\infty}I_n(f_n(\cdot,t,x))\,,
\end{equation}
where for each $(t,x)$, $f_n(\cdot,t,x)$ is a symmetric element in
$\mathcal{H}^{\otimes n}$.
Thanks to \eqref{eq:delta-u-chaos} and using an iteration procedure, one can then find an
explicit formula for the kernels $f_n$ for $n \geq 1$
\begin{eqnarray*}
f_n(s_1,x_1,\dots,s_n,x_n,t,x)=\frac{1}{n!}p_{t-s_{\si(n)}}(x-x_{\si(n)})\cdots p_{s_{\si(2)}-s_{\si(1)}}(x_{\si(2)}-x_{\si(1)})
p_{s_{\si(1)}}u_0(x_{\si(1)})\,,
\end{eqnarray*}
where $\si$ denotes the permutation of $\{1,2,\dots,n\}$ such that $0<s_{\si(1)}<\cdots<s_{\si(n)}<t$
(see, for instance,  equation (4.4) in \cite{HN}, where this formula is established in the case of a noise which is white in space).
Then, to show the existence and uniqueness of the solution it suffices to show that for all $(t,x)$ we have
\begin{equation}\label{chaos}
\sum_{n=0}^{\infty}n!\|f_n(\cdot,t,x)\|^2_{\mathcal{H}^{\otimes n}}< \infty\,.
\end{equation}

\begin{theorem}\label{thmSk1}
Suppose that   $\mu$ satisfies Hypothesis   \ref{hyp:mu}
and $\gamma$ is locally integrable. Then relation (\ref{chaos}) holds
for each $(t,x)$. Consequently, equation \eref{eqSk1} admits a
unique mild solution in the sense of Definition \ref{def1}.
\end{theorem}

\begin{proof}
Fix $t>0$ and $x\in \mathbb{R}^d$. Set
$f_n(s,y,t,x)=f_n(s_1,y_1,\dots,s_n,y_n,t,x)$.    We have the
following expression
\begin{eqnarray*}
&&n!\|f_n(\cdot,t,x)\|^2_{\mathcal{H}^{\otimes n}}\\
&=& n!
\int_{[0,t]^{2n}}\int_{\R^{2nd}}f_n(s,y,t,x)f_n(r,z,t,x)\prod_{i=1}^n
\laa(y_i-z_i)\prod_{i=1}^n \gamma(s_i-r_i) \, dydzdsdr\\
&\leq& n!
\|u_0\|^2_{\infty} \int_{[0,t]^{2n}}\int_{\R^{2nd}} 
g_n(s,y,t,x)g_n(r,z,t,x)\prod_{i=1}^n
\laa(y_i-z_i)\prod_{i=1}^n \gamma(s_i-r_i) \, dydzdsdr
\end{eqnarray*}
where $dx=dx_1 \cdots dx_n$, the differentials $dy, ds$ and $ dr$ are defined
similarly and
\begin{equation}\label{eq:def-gn}
g_n(s,y,t,x)=\frac{1}{n!}p_{t-s_{\sigma(n)}}(x-y_{\sigma(n)})\cdots
p_{s_{\sigma(2)}-s_{\sigma(1)}}(y_{\sigma(2)}-y_{\sigma(1)})\,.
\end{equation}
Set now $\mu(d\xi)\equiv\prod_{i=1}^n \mu(d\xi_i)$.
Using the Fourier transform and Cauchy-Schwarz,  we obtain
\begin{eqnarray*}   
&&n!\|f_n(\cdot,t,x)\|^2_{\mathcal{H}^{\otimes n}}   \\  
 &\leq &
n! \|u_0\|^2_{\infty}\int_{[0,t]^{2n}} \int_{\R^{nd}}
\mathcal{F}g_n(s,\cdot,t,x)(\xi)
\overline{\mathcal{F}g_n(r,\cdot,t,x)(\xi)}  \mu(d\xi)\prod_{i=1}^n
\gamma(s_i-r_i)dsdr   \\ 
\notag &\leq&n! \|u_0\|^2_{\infty}
\int_{[0,t]^{2n}}\left (\int_{\R^{nd}}
|\mathcal{F}g_n(s,\cdot,t,x)(\xi)|^2\mu(d\xi)\right)^{\frac{1}{2}} 
\left(\int_{\R^{nd}}|\mathcal{F}g_n(r,\cdot,t,x)(\xi)|^2\mu(d\xi)\right)^{\frac{1}{2}}  \\
&& \hspace{4.5in}\times \prod_{i=1}^n \gamma(s_i-r_i)ds dr,   \\
\end{eqnarray*}
and thus, thanks to the basic
inequality $ab\leq \frac{1}{2}(a^2+b^2)$
  and the fact that $\gamma$ is locally integrable, this yields:
\begin{eqnarray}   
n!\|f_n(\cdot,t,x)\|^2_{\mathcal{H}^{\otimes n}} &\leq &   
n! \|u_0\|^2_{\infty}
\int_{[0,t]^{2n}}\int_{\R^{nd}}  |\mathcal{F}g_n(s,\cdot,t,x)(\xi)|^2\mu(d\xi) \prod_{i=1}^n \gamma(s_i-r_i)ds dr  \notag \\
&\le&  C^nn! \|u_0\|^2_{\infty}  \label{eq3}
\int_{[0,t]^{2n}}\int_{\R^{nd}}  |\mathcal{F}g_n(s,\cdot,t,x)(\xi)|^2\mu(d\xi)  ds,
\end{eqnarray}
where $C=2 \int_0^t \gamma(r)dr$.
Furthermore, it is readily checked from expression \eqref{eq:def-gn} that there exists a constant $C>0$ such that the Fourier transform of $g_n$ satisfies
\[
|\mathcal{F}g_n(s,\cdot,t,x)(\xi)|^2=\frac{C^n}{(n!)^2}
\prod_{i=1}^n e^{-(s_{\si(i+1)}-s_{\si(i)})|\xi_{\si(i)}+\cdots +
\xi_{\si(1)}|^2},
\]
where we have set $s_{\si(n+1)}=t$.  As a consequence,
\begin{eqnarray}   \notag
&&(n!)^2 \int_{\R^{nd}}|\mathcal{F}g_n(s,\cdot,t,x)(\xi)|^2\mu(d\xi)\\
\notag &\leq&  C^n
 \prod_{i=1}^n \sup_{\eta \in
\R^d}\int_{\R^d}e^{-(s_{\sigma(i+1)}-s_{\sigma(i)})|\xi_{\sigma(i)}+\eta|^2}\mu(d\xi_{\sigma(i)})\\
\notag &=&C^n  \prod_{i=1}^n \sup_{\eta \in \R^d} \left|
\int_{\R^d}\frac{e^{-\frac{|x_{\si(i)}|^2}{4(s_{\sigma(i+1)}-s_{\sigma(i)})}}}{(4\pi
(s_{\si(i+1)}-s_{\si(i)}))^{\frac{d}{2}}} \, 
e^{\imath x_{\si(i)}\cdot \eta} \, \laa(x_{\si(i)})dx_{\si(i)}\right|\\
\label{eq4} &\leq&  C^n
 \prod_{i=1}^n \int_{\R^d}e^{-(s_{\sigma(i+1)}-s_{\sigma(i)})|\xi_{\sigma(i)}|^2}\mu(d\xi_{\sigma(i)}),
 \end{eqnarray}
 where we invoke the fact that $|e^{\imath x_{\si(i)}\cdot \eta}|=1$ to get rid of the supremum in $\eta$. Therefore, from relations~(\ref{eq3}) and (\ref{eq4}) we obtain
\begin{equation}
n!\|f_n(\cdot,t,x)\|^2_{\mathcal{H}^{\otimes n}} \leq \|u_0\|^2_{\infty} C^n \int_{\R^{nd}}  \int_{T_n(t)}\prod_{i=1}^n e^{-(s_{i+1}-s_i)|\xi_i|^2} \, ds \, \mu(d\xi)\,,
\end{equation}
where we denote by $T_n(t)$ the simplex
\begin{equation} \label{simplex}
 T_n(t) =\{0<s_1<\cdots <s_n<t\}.
 \end{equation}
Let us now estimate the right hand side of \eref{eq4}:
making the change of variables $s_{i+1}-s_{i}=w_i$ for $1\leq i \leq n-1$, and $t-s_n=w_n$,  and denoting $dw=dw_1 dw_2 \cdots dw_n$, we end up with
\begin{eqnarray*}
 n!\|f_n(\cdot,t,x)\|^2_{\mathcal{H}^{\otimes n}}  \leq \|u_0\|^2_{\infty} C^n \int_{\R^{nd}}\int_{S_{t,n}}e^{- \sum_{i=1}^n w_i |\xi_i|^2}
 dw \prod_{i=1}^n \mu(d\xi_i)\,,
\end{eqnarray*}
where $S_{t,n}=\{(w_1, \dots, w_n)\in [0,  \infty)^n: w_1 +\cdots +w_n \le t\}$. 
We also split the contribution of $\mu$ in the following way: 
fix $N\ge 1$ and set
\begin{equation}  \label{eqD}
C_N=\int_{|\xi|\geq N}\frac{\mu(d\xi)}{|\xi|^2},
\quad\text{and}\quad
D_N=\mu\{\xi\in \R^d: |\xi|\leq N\}.
\end{equation}
By Lemma \ref{lem1} below, we can write
\begin{equation}\label{eq:eqDD}
 n!\|f_n(\cdot,t,x)\|^2_{\mathcal{H}^{\otimes n}}  \leq     \|u_0\|^2_{\infty}
 C^n \sum_{k=0}^n {n \choose k} \frac{t^k}{k!}D_N^k
(2C_N)^{n-k}\,.
\end{equation}
  Next we choose a sufficiently large $N$ such that
$2CC_N < 1$, which is possible  because of condition  \eref{mu1}. Using the
inequality ${n \choose k } \leq 2^n$ for any positive integers $n$
and $0\leq k \leq n$,  we have
\begin{align*}
&\sum_{n=0}^{\infty}n!\|f_n(\cdot,t,x)\|^2_{\mathcal{H}^{\otimes n}}
\leq
\|u_0\|^2_{\infty} \sum_{n=0}^{\infty}C^n \sum_{k=0}^n {n \choose k}\frac {t^k}{k!}D_N^k (2C_N)^{n-k}\\
&\leq \|u_0\|^2_{\infty} \sum_{k=0}^{\infty}\sum_{n=k}^{\infty}C^n 2^n \frac{t^k}{k!}D_N^k (2C_N)^{n-k}
=
\|u_0\|^2_{\infty} \sum_{k=0}^{\infty}\frac{t^k}{k!}D_N^k (2C_N)^{-k}\sum_{n=k}^{\infty}(2CC_N)^n\\
&\leq
\|u_0\|^2_{\infty} \frac{1}{1-2C C_N} \sum_{k=0}^{\infty}\frac{t^k D_N^k
(2C_N)^{-k}(2CC_N)^k}{k!}< \infty\,.
\end{align*}
This proves the theorem.
\end{proof}

Next we establish the lemma that is used in the proof of Theorem \ref{thmSk1}.
\begin{lemma}\label{lem1}
Let $\mu$ satisfy the condition \eref{mu1}. For any $N
> 0$ we let $D_N$ and $C_N$ be given by (\ref{eqD}). Then
we have
\[
 \int_{\R^{nd}}\int_{S_{t,n}}e^{- \sum_{i=1}^n w_i |\xi_i|^2}
 dw \prod_{i=1}^n \mu(d\xi_i)
\leq\sum_{k=0}^n {n \choose k}
\frac{t^k}{k!}D_N^k (2C_N)^{n-k}\,.
\]

\end{lemma}
\begin{proof}
By our assumption \eref{mu1}$, C_N$ is finite for all
positive $N$. Let $I$ be a subset of $\{1,2,\dots,n\}$ and  $I^{c}
=\{1,2,\dots,n\}\setminus I$. Then we have
\begin{eqnarray*}
&&\int_{\R^{nd}}\int_{S_{t,n}}\prod_{i=1}^n e^{-w_i
|\xi_i|^2} \, dw \, \mu(d\xi)\\
&=&  \int_{\R^{nd}} \int_{ S_{t,n}}\prod_{i=1}^n e^{-w_i|\xi_i|^2}({\bf 1}_{\{|\xi_i|\leq N\}}+ {\bf 1}_{\{|\xi_i|> N\}})\, dw \, \mu(d\xi)\\
&=& \sum_{I \subset \{1,2,\dots,
n\}}\int_{\R^{nd}}\int_{ S_{t,n}}\prod_{i \in I} e^{-w_i
|\xi_i|^2}{\bf 1}_{\{|\xi_i|\leq N\}}\times \prod_{j \in I^{c}}
 e^{-w_j |\xi_j|^2} {\bf 1}_{\{|\xi_j|\geq N\}} \, dw \, \mu(d\xi).
 \end{eqnarray*}
 For the indices $i$ in the set $I$ we estimate $e^{-w_j |\xi_j|^2}$ by $1$. Then, using the inclusion
 \[
 S_{t,n}\subset S^I _{t} \times S^{I^c}_{t},
 \]
 where $S^I_{t} =\{(w_i ,i\in I): w_i\ge 0, \sum_{i\in I} w_i \le t\}$ and $S^{I^c}_{t} =\{(w_i ,i\in I^c): w_i\ge 0, \sum_{i\in I^c} w_i \le t\}$ we obtain
 \begin{eqnarray*}
&&\int_{\R^{nd}}\int_{S_{t,n}}\prod_{i=1}^n e^{-w_i
|\xi_i|^2}\, dw \, \mu(d\xi) \\
&\leq& \sum_{I \subset \{1,2,\cdots, n\}}\int_{\R^{nd}}\int_{ S^I_t\times S_t^{I^c}}\prod_{i \in I} {\bf 1}_{\{|\xi_i|\leq
N\}}\times \prod_{j \in I^{c}}
e^{-w_j |\xi_j|^2} {\bf 1}_{\{|\xi_j|\geq N\}} \, dw \, \mu(d\xi) .
\end{eqnarray*}
Furthermore, one can bound the integral over $S_t^{I^c}$ in the following way
\begin{equation*}
\int_{S_t^{I^c}} \prod_{j \in I^{c}} e^{-w_j |\xi_j|^2}  \, dw
\le
\int_{[0,t]^{I^c}} \prod_{j \in I^{c}} e^{-w_j |\xi_j|^2} \, dw
=
\prod_{j \in I^{c}} \frac{1-e^{-t |\xi_{j}|^{2}}}{|\xi_{j}|^{2}}
\le
\prod_{j \in I^{c}} \frac{1}{|\xi_{j}|^{2}}.
\end{equation*}
We can thus bound $\int_{\R^{nd}}\int_{S_{t,n}}\prod_{i=1}^n e^{-w_i |\xi_i|^2}\, dw \, \mu(d\xi)$ by:
\begin{multline*}
\sum_{I \subset \{1,2,\cdots,n\}}\frac{t^{|I|}}{|I|!}\big(\mu\{\xi\in \R^d: |\xi|\leq N\}\big)^{|I|}2^{|I^{c}|}\int_{|\xi_j|>N, \forall j \in I^{c}}\prod_{j \in I^{c}}\frac{\mu(d\xi_j)}{|\xi_j|^2}\\
= \sum_{I \subset
\{1,2,\cdots,n\}}\frac{t^{|I|}}{|I|!}D_N^{|I|}(2C_N)^{|I^{c}|}=
\sum_{k=0}^n {n \choose k} \frac{t^k}{k!}D_N^k (2C_N)^{n-k}\,,
\end{multline*}
which is our claim.
\end{proof}

\subsection{Feynman-Kac formula for the moments}\label{sec:FK-moments}
Our next objective is to find a formula  for the moments of the mild solution to equation \eref{eqSk1}.
For any $\delta > 0$, we define the function
$\varphi_{\delta}(t)=\frac{1}{\delta}{\bf 1}_{[0,\delta]}(t)$ for $t
\in \R$.  Then, $\varphi_{\delta}(t)p_{\varepsilon}(x)$  provides an
approximation of the Dirac delta function $\delta_0(t,x)$ as
$\varepsilon$ and $\delta$ tend to zero.

We set
\begin{equation}\label{regW}
\dot{W}^{\varepsilon,\delta}_{t,x}=\int_0^t
\int_{\R^d}\varphi_{\delta}(t-s)p_{\varepsilon}(x-y)W(ds,dy)\,.
\end{equation}
Now we consider the approximation of equation \eref{eqSk1} defined
by
\begin{equation}\label{approx}
\frac{\partial u_{t,x}^{\varepsilon,\delta}}{\partial t}=\frac{1}{2}\Delta u_{t,x}^{\varepsilon,\delta}+u_{t,x}^{\varepsilon,\delta}\diamond \dot{W}_{t,x}^{\varepsilon,\delta}\,.
\end{equation}

We recall that the Wick product $u_{t,x}^{{\varepsilon ,\delta }}\diamond
\dot{W}_{t,x}^{{\varepsilon ,\delta }}$ is well defined as a \ square
integrable random variable provided the random variable $u_{t,x}^{{%
\varepsilon ,\delta }}$ belongs to the space $\mathbb{D}^{1,2}$ (see (\ref%
{Wick})), and in this case we have%
\begin{equation}
u_{s,y}^{\varepsilon ,\delta }\diamond \dot{W}_{s,y}^{{\varepsilon ,\delta }%
}=\int_{0}^{s}\int_{\mathbb{R}^d}\varphi _{\delta }(s-r)p_{{\varepsilon }%
}(y-z)u_{s,y}^{\varepsilon ,\delta }\delta W_{r,z}.  \label{eq5}
\end{equation}%
Furthermore, the mild or evolution version of  (\ref{approx})  is
\begin{equation}
u_{t,x}^{\varepsilon ,\delta }=p_tu_0(x)+\int_{0}^{t}\int_{\mathbb{R}%
^d}p_{t-s}(x-y)u_{s,y}^{\varepsilon ,\delta }\diamond \dot{W}_{s,y}^{{%
\varepsilon ,\delta }}dsdy.  \label{eq6}
\end{equation}%
Substituting (\ref{eq5}) into (\ref{eq6}), and formally applying Fubini's
theorem yields
\begin{equation}
u_{t,x}^{\varepsilon ,\delta }=p_tu_0(x)+\int_{0}^{t}\int_{\mathbb{R}%
^d}\left( \int_{0}^{t}\int_{\mathbb{R}^d}p_{t-s}(x-y)\varphi _{\delta
}(s-r)p_{{\varepsilon }}(y-z)u_{s,y}^{\varepsilon ,\delta }dsdy\right)
\delta W_{r,z}.  \label{eq7}
\end{equation}
This leads to the following definition.

\begin{definition}  \label{def2}
An adapted random field $u^{\varepsilon ,\delta }=\{u_{t,x}^{{\varepsilon
,\delta }};  t\geq 0,x\in \mathbb{R}^{d}\}$ is a mild solution to equation (%
\ref{approx}) if for each $(r,z)\in  [0,t]\times \mathbb{R}^d$ the
integral
\begin{equation*}
Y_{r,z}^{t,x}= \int_{0}^{t}\int_{\mathbb{R}^d}p_{t-s}(x-y)\varphi
_{\delta }(s-r)p_{{\varepsilon }}(y-z)u_{s,y}^{\varepsilon ,\delta
}dsdy
\end{equation*}%
exists and $Y^{t,x}$ is a Skorohod integrable process such that (\ref{eq7})
holds for each $(t,x)$.
\end{definition}

Notice that according to relation \eqref{dual}, the above definition is equivalent to saying that $u_{t,x}^{\varepsilon
,\delta }\in L^{2}(\Omega )$, and for any random variable $F\in \mathbb{D}%
^{1,2}$ , we have%
\begin{equation}
 \be \lc Fu_{t,x}^{\varepsilon ,\delta }\rc  =\be \lc F\rc p_tu_0(x)+\be \lc \langle  Y^{t,x},DF\rangle _{\mathcal{H}}\rc.  \label{eq8}
\end{equation}%

In order to derive a Feynman-Kac formula for the moment of order $k\ge 2$  of the solution to equation (\ref{eqSk1}) we need to introduce $k$  independent $d$-dimensional Brownian motions $B^j$, $j=1,\dots, k$, which are independent of the noise $W$ driving the equation. We shall study the probabilistic behavior of some random variables with double randomness, and we thus introduce some additional notation:
\begin{notation}
We denote by $\bp,\be$ the probability and expectation with respect
to the annealed randomness concerning the couple $(B,W)$, where
$B=(B^1, \dots, B^k)$, while we set respectively  $\be_{B}$ and
$\be_{W}$ for the expectation with respect to one randomness only.
\end{notation}

With this notation in mind, define
\begin{equation}\label{eq9}
u_{t,x}^{\varepsilon,\delta}=\be_B \lc \exp \lp  W (
A_{t,x}^{\varepsilon,\delta})-\frac{1}{2}\alpha^{\varepsilon,\delta}_{t,x}\rp
\rc\,,
\end{equation}
where
\begin{equation}  \label{m3}
A_{t,x}^{\varepsilon,\delta}(r,y)=\frac 1\delta
\left(\int_0^{\delta \wedge (t-r)}
p_{\varepsilon}(B_{t-r-s}^x-y)ds\right) \mathbf{1}_{[0,t]} (r),
\quad\text{and}\quad
\alpha^{\varepsilon,\delta}_{t,x}=\|A^{\varepsilon,\delta}_{t,x}\|^2_{\mathcal{H}},
\end{equation}
for  a standard $d$-dimensional Brownian motion $B$ independent of
$W$. Then one can prove that $u_{t,x}^{\varepsilon,\delta}$ is a
mild solution to equation (\ref{approx}) in the sense of Definition
\ref{def2}. The proof is similar to the proof of Proposition 5.2 in
\cite{HN} and we omit the details.

The next theorem asserts that the random variables
$u_{t,x}^{\varepsilon,\delta}$ have moments of all orders, uniformly
bounded in $\varepsilon$ and $\delta$, and converge to the mild
solution of equation \eref{eqSk1}, which is unique by Theorem
\ref{thmSk1}, as $\delta$ and $\varepsilon$ tend to zero. Moreover,
it provides an expression for the moments of the mild solution of
equation \eref{eqSk1}.

\begin{theorem}\label{thmmom1}
Suppose $\gamma$ is locally integrable and $\mu$ satisfies
Hypothesis  \ref{hyp:mu}.  Then for any integer $k \geq 1$ we have
\begin{equation}\label{eq10}
\sup_{\varepsilon,\delta}\be \lc |u_{t,x}^{\varepsilon,\delta}|^k\rc< \infty\,,
\end{equation}
the limit $\lim_{\varepsilon \downarrow 0}\lim_{\delta
\downarrow 0} u_{t,x}^{\varepsilon,\delta}$ exists in $L^p$ for all
$p \geq 1$, and it coincides with the mild solution $u$ of equation
\eref{eqSk1}. Furthermore, we have for any  integer $k \geq 2$
\begin{equation}\label{momSk1}
\be \lc u_{t,x}^k\rc =\be_B \lc \prod_{i=1}^k u_0(B_t^i+x) \exp\left(\sum_{1 \leq i < j
\leq k}\int_0^t \int_0^t \gamma (s-r)\laa(B_s^i-B_r^j)ds
dr\right)\rc\,,
\end{equation}
where $\{B^j; \, j=1,\dots, k\}$  is a family of $d$-dimensional  independent standard Brownian motions independent of $W$.
\end{theorem}

\begin{proof}
To simplify the proof we assume that $u_0$ is identically one. Fix an integer $k \geq 2$. Using \eref{eq9} we have
\begin{equation*}
\be \lc \lp u_{t,x}^{\varepsilon,\delta}\rp ^k\rc=\be_W \lc\prod_{j=1}^k
\be_B\lc \exp \lp   W(A^{\varepsilon,\delta,
B^j}_{t,x})-
\frac{1}{2}\alpha_{t,x}^{\varepsilon,\delta,B^j}\rp \rc \rc\,,
\end{equation*}
where for any $j=1,\dots,k$,  $A_{t,x}^{\varepsilon,\delta,B^j}$ and $\alpha_{t,x}^{\varepsilon,\delta,B^j}$ are evaluations of  \eqref{m3} using the Brownian motion $B^j$. Therefore,  since $W(A^{\varepsilon,\delta, B^j}_{t,x})$ is a Gaussian random variable conditionally on $B$, we obtain
\begin{eqnarray}\label{eq:exp-moments-utx-ep-delta}
\be \lc \lp u_{t,x}^{\varepsilon,\delta}\rp ^k\rc &=&
\be_B \lc 
\exp \lp\frac{1}{2}\|\sum_{j=1}^k A_{t,x}^{\varepsilon,\delta,B^j}\|^2_{\mathcal{H}}
-\frac{1}{2}\sum_{j=1}^k \alpha_{t,x}^{\varepsilon,\delta,B^j}\rp\rc \notag\\
&=& \be_B \lc 
\exp \lp\frac{1}{2}\|\sum_{j=1}^k A_{t,x}^{\varepsilon,\delta,B^j}\|^2_{\mathcal{H}}
-\frac{1}{2}\sum_{j=1}^k \| A_{t,x}^{\varepsilon,\delta,B^j}\|^2_{\mathcal{H}}\rp\rc   \notag\\
&=&\be_B \lc \exp \lp\sum_{1\leq i < j \leq k}\langle
A_{t,x}^{\varepsilon,\delta,B^i},
A_{t,x}^{\varepsilon,\delta,B^j}\rangle _{\mathcal{H}}\rp\rc\,.
\end{eqnarray}

Let us now evaluate the quantities $\langle A_{t,x}^{\varepsilon,\delta,B^i}, A_{t,x}^{\varepsilon,\delta,B^j}\rangle _{\mathcal{H}}$ above: by the definition of $A_{t,x}^{\varepsilon,\delta,B^i}$, for any $i\not= j$  we have
\begin{equation} \label{eq11}
\langle A_{t,x}^{\varepsilon,\delta,B^i},A_{t,x}^{\varepsilon,\delta,B^j}\rangle_{\mathcal{H}} =
 \int_0^t \int_0^t \int_{\R^{d}} \cf A^{\ep,\delta, B^i}_{t,x} (u,\cdot)(\xi) \, \overline{\cf A}^{\ep,\delta,B^j}_{t,x}(v,\cdot) (\xi)\gamma(u-v) \mu(d\xi) dudv.
 \end{equation}
 On the other hand, for $u\in [0,t]$ we can write
 \begin{eqnarray*}
\cf A^{\ep,\delta, B^i}_{t,x}(u,\cdot)(\xi)
&=&  \frac 1\delta \int_0 ^{ \delta \wedge(t-u)}  \cf p_{\ep}(B_{t-u-s}^{i}+x-\cdot) (\xi) ds  \\
&= &\frac 1\delta \int_0 ^{ \delta \wedge(t-u)}    \exp \lp
-\frac{\ep^{2}|\xi|^{2}}{2}+\imath \lla \xi ,B^i_{t-u-s}+x\rra \rp
ds.
\end{eqnarray*}
Thus
\begin{align}
&\langle A_{t,x}^{\varepsilon,\delta,B^i},A_{t,x}^{\varepsilon,\delta,B^j}\rangle_{\mathcal{H}} 
\label{m2} \\ 
&\hspace{0.20in} =\int_{\R^{d}} \left(   \iot \iot
   \left( \frac 1{\delta^2}  \int_0^{\delta\wedge v } \int_0^{\delta\wedge u} e^{\imath \lla \xi ,B^i_{u-s_1}- B^j_{v-s_2}\rra}     ds_1ds_2   \right)\gamma(u-v)    dudv  \right)    
   \times   e^{-\ep^2 |\xi|^2}   \mu(d\xi), \notag
\end{align}
and we divide the proof in several steps.

\smallskip
\noindent \textit{Step 1:}  We claim that,
\begin{equation}\label{eq12}
\lim_{\varepsilon \downarrow 0} \lim_{ \delta \downarrow 0} \langle A_{t,x}^{\varepsilon,\delta,B^i}, A_{t,x}^{\varepsilon,\delta,B^j}\rangle_{\mathcal{H}}=\iot \iot\gamma(u-v)\laa(B_{u}^i-B_{v}^j) du dv \,,
\end{equation}
where the convergence holds in   $L^1(\Omega)$. Notice first that
the right-hand side of equation (\ref{eq12}) is finite almost surely
because
\[
\be_B \lc \iot \iot \gamma(u-v)\laa(B_{u}^i-B_{v}^j) dudv\rc =\iot
\iot    \int_{\R^{ d}} \gamma(u-v) e^{-\frac {1}{2} (u+v) |\xi|^2}
\mu(d\xi) dudv
\]
and we show that this is finite  making the change of variables  $x=u-v$, $y=u+v$, and using  our hypothesis on $\gamma$ and $\mu$ like in the proof of Theorem \ref{thmSk1}.

 In order to show the convergence (\ref{eq12}) we first let  $\delta$ tend to zero. Then, owing to the continuity of $B$ and applying some dominated convergence arguments to \eqref{m2}, we obtain
 the following limit almost surely and in $L^1(\Omega)$
 \begin{equation}\label{eq13}
 \lim_{ \delta \downarrow 0} \langle A_{t,x}^{\varepsilon,\delta,B^i}, A_{t,x}^{\varepsilon,\delta,B^j}\rangle_{\mathcal{H}}
 = 
 \int_{\R^d}  \lp   \iot \iot    e^{\imath \lla \xi ,B^i_{u}- B^j_{v}\rra}     \gamma(u-v) dudv\rp  e^{-\ep^2 |\xi|^2} \  \mu(d\xi) .
\end{equation}
Finally, it is easily checked that the right-hand side of (\ref{eq13}) converges in $L^1(\Omega)$ to the right-hand side of  (\ref{eq12}) as $\varepsilon$ tends to zero, by means of a simple dominated convergence argument again.

\smallskip
\noindent \textit{Step 2:}
For notational convenience, we denote by $B$ and $\widetilde{B}$ two independent $d$-dimensional Brownian motions, and  $\be$ will denote here the expectation with respect to both $B$ and $\widetilde{B}$.  We  claim that for any $\lambda > 0$
\begin{equation}\label{expint}
\sup_{\varepsilon,\delta}\be  \lc  
\exp \lp \lambda\lla A_{t,x}^{\varepsilon,\delta,B}, A_{t,x}^{\varepsilon,\delta,\widetilde{B}}\rra _{\mathcal{H}} \rp \rc< \infty\,.
\end{equation}
Indeed, starting from (\ref{m2}), making the change of variables $u-s_1
\rightarrow u$, $v-s_2 \rightarrow v$,  assuming $\delta \le t$, and
using Fubini's theorem, we can write
\begin{multline*}
  \lla A_{t,x}^{\varepsilon,\delta,B}, A_{t,x}^{\varepsilon,\delta,\widetilde{B}}\rra _{\mathcal{H}}
  =
\frac 1{\delta^2}  \int_0^\delta \int_0^\delta   \int_0^{t-s_1}\int_0^{t-s_2}\int_{\R^d}  \exp\lp-\imath (B_{u}-\widetilde{B}_{v})\cdot \xi \rp \\
\times \exp(-\varepsilon |\xi|^2)  \gamma(u+s_1-v-s_2) \, \mu(d\xi) \, dudvds_1ds_2\,.
\end{multline*}
We now control the moments of $\langle A_{t,x}^{\varepsilon,\delta,B}, A_{t,x}^{\varepsilon,\delta,\tilde{B}}\rangle _{\mathcal{H}}$ in order to reach exponential integrability:
\begin{multline} \label{m7}
 \lla A_{t,x}^{\varepsilon,\delta,B}, A_{t,x}^{\varepsilon,\delta,\widetilde{B}}\rra _{\mathcal{H}}^n
=\frac 1{\delta^{2n}}  \int_{O_{\delta,n}} \int_{\R^{dn}}
\exp\left(-\imath\sum_{l=1}^n (B_{u_l}-\widetilde{B}_{v_l})\cdot\xi_l\right)\\ 
\times e^{-\varepsilon \sum_{l=1}^n|\xi_l|^2} \prod_{l=1}^n
\gamma(u_l+s_l-v_l-\widetilde{s}_l) \, \mu(d\xi) \, dsd\tilde{s}dudv,
\end{multline}
where $\mu(d\xi)=\prod_{l=1}^n \mu(d\xi_l)$, the differentials 
$ds,d\tilde{s},du,dv$ are defined similarly, and
\[
O_{\delta,n}=\left\{ (s, \widetilde{s}, u,v);\, 0\le s_l, \widetilde{s}_l
\le \delta, \, 0\le u_l \le t-s_l, \, 0\le v_l \le
t-\widetilde{s}_l , \text{ for all } 1\leq l \leq n\right\}.
\]
Moreover, we have:
\begin{eqnarray}\label{eq:charac1}
\be\lc  \exp\left(-\imath\sum_{l=1}^n (B_{u_l}-\widetilde{B}_{v_l})\cdot\xi_l\right)\rc
&=&
\exp\lp -\frac12 \var\lp  \sum_{l=1}^n (B_{u_l}-\widetilde{B}_{v_l})\cdot\xi_l\rp \rp 
 \\
&=&
\exp\left(-\frac{1}{2}\sum_{1\leq i,j\leq n}
(u_i\wedge u_j+v_i\wedge v_j)
\xi_i\cdot \xi_j\right). \notag
\end{eqnarray}
Taking into account the fact that $\gamma$ is locally integrable, this yields
\begin{eqnarray*}
\be \lc \left\langle A_{t,x}^{\varepsilon,\delta,B},
A_{t,x}^{\varepsilon,\delta, \widetilde {B}}\right\rangle _{\mathcal{H}}^n \rc
&\leq& C^n
\int_{[0,t]^{2n}}\int_{\R^{dn}}\exp\left(-\frac{1}{2}\sum_{1\leq
i,j\leq n}(s_i\wedge s_j+\widetilde{s}_i\wedge \widetilde{s}_j)\xi_i\cdot
\xi_j\right) \mu(d\xi)dsd\tilde{s}\\
&\leq&C^n \int_{\R^{dn}}\int_{[0,t]^n}\exp\left(-\sum_{1\leq i,j\leq n}(s_i\wedge s_j) \, \xi_i \cdot \xi_j\right)ds \mu(d\xi)\,.
\end{eqnarray*}
 Since
\begin{equation*}
\int_{\R^{dn}}\exp\left(-\sum_{1\leq i,j\leq n}(s_i\wedge s_j)\xi_i \cdot \xi_j\right)\mu(d\xi)
\end{equation*}
is a symmetric function of $s_1,s_2,\dots,s_n$, we can restrict our
integral to $T_n(t) =\{0<s_1< s_2< \cdots < s_n< t\}$. Hence, using the convention $s_0=0$,  we have
\begin{eqnarray}\label{eq:mAn}
\be  \lc \left \langle A_{t,x}^{\varepsilon,\delta,B}, A_{t,x}^{\varepsilon,\delta, \widetilde {B}}\right \rangle _{\mathcal{H}}^n\rc
&\leq&  C^n n!\int_{\R^{dn}}\int_{T_n(t)}\exp\left(-\sum_{1\leq i,j\leq n}(s_i\wedge s_j)\xi_i \cdot \xi_j\right)ds \mu(d\xi)\\
&=&C^n n! \int_{\R^{dn}}\int_{T_n(t)}\exp\left(-\sum_{i=1}^n
(s_i-s_{i-1})|\xi_i+ \cdots +\xi_n|^2\right)ds\mu(d\xi). \notag
\end{eqnarray}
Thus, using the same argument as in the proof of  the estimate (\ref{eq4}),  we end up with
\begin{eqnarray*}
\be  \lc \left \langle A_{t,x}^{\varepsilon,\delta,B}, A_{t,x}^{\varepsilon,\delta, \widetilde {B}}\right\rangle _{\mathcal{H}}^n\rc
&\leq&C^n n! \int_{T_n(t)}\prod_{i=1}^n \left(\sup_{\eta \in \R^d}
\int_{\R^d}e^{-(s_i-s_{i-1})|\xi_i+\eta|^2}\mu(d\xi)\right)ds\\
&\le &C^n n!
\int_{T_n(t)}\prod_{i=1}^n\left(\int_{\R^d}e^{-(s_i-s_{i-1})|\xi_i|^2}\mu(d\xi_i)\right)ds\,.
\end{eqnarray*}
Making  the change of variable $w_i=s_i-s_{i-1}$, the above integral
is equal to
\begin{equation*}
C^n n! \int_{S_{t,n}} \int_{\R^{dn}}\prod_{i=1}^n e^{-w_i
|\xi_i|^2}\mu(d\xi)dw 
\le
C^n n! \sum_{k=0}^n {n \choose k} \frac{t^k}{k!}D_N^k
(2C_N)^{n-k} ,
\end{equation*}
where we have resorted to Lemma \ref{lem1} for the last inequality.
Therefore,
\begin{eqnarray*}
\frac{1}{n!} \be \lc \left\langle A_{t,x}^{\varepsilon,\delta,B},
A_{t,x}^{\varepsilon,\delta, \widetilde {B}}\right\rangle _{\mathcal{H}}^n  \rc
\leq C^n  \sum_{k=0}^n {n \choose k} \frac{t^k}{k!}D_N^k
(2C_N)^{n-k}\,,
\end{eqnarray*}
which is exactly the right hand side of \eqref{eq:eqDD}. Thus, along the same lines as in the proof of Theorem \ref{thmSk1}, we get
\begin{equation*}
\be \lc  \exp \lp \lambda
\left\langle A_{t,x}^{\varepsilon,\delta,B}, A_{t,x}^{\varepsilon,\delta, \widetilde {B}}\right\rangle _{\mathcal{H}} \rp\rc
=
\sum_{n=0}^{\infty}\frac{\lambda^n}{n!}\be \lc \left\langle A_{t,x}^{\varepsilon,\delta,B}, A_{t,x}^{\varepsilon,\delta, \widetilde {B}}\right\rangle _{\mathcal{H}}^n\rc\,<\infty,
\end{equation*}
which completes the proof of \eref{expint}.

\medskip
\noindent \textit{Step 3:} Starting from \eqref{eq:exp-moments-utx-ep-delta}, (\ref{eq12}) and (\ref{expint})  we
deduce that  $\be [ ( u_{t,x}^{\varepsilon,\delta})^k]$
converges as $\delta$ and $\varepsilon$ tend to zero to the
right-hand side of   \eref{momSk1}. On the other hand, we can also
write
\[
\be \lc  u_{t,x}^{\varepsilon,\delta} u_{t,x}^{\varepsilon',\delta'}  \rc=
\be_B \lc  \exp \lp\ \langle A^{\varepsilon,\delta,
B^1}_{t,x} , A^{\varepsilon',\delta',
B^2}_{t,x}  \rangle_{\mathcal{H}}\rp\rc\,.
\]
As before we can show that this converges as $\varepsilon,\delta,
\varepsilon', \delta'$ tend to zero. So,
$u_{t,x}^{\varepsilon,\delta}$ converges in $L^2$ to some limit
$v_{t,x}$, and the limit is actually in  $L^p$ , for all $p \geq 1$.
Moreover, $\be [v^k_{t,x}]$ equals to the right hand side of
\eref{momSk1}. Finally, letting $ \delta$ and $\varepsilon$ tend to
zero in equation \eref{eq8} we get
\begin{equation*}
\be [Fv_{t,x}]= \be[ F]  +\be \lc \langle DF, v p_{t-\cdot}(x-\cdot)\rangle_{\mathcal{H}}\rc
\end{equation*}
which implies that the process $v$ is the solution of equation
\eref{eqSk1}, and by the uniqueness of the solution we have $v=u$.
\end{proof}
\begin{remark}\label{rmk2}
If the space dimension is $1$,  we can consider equation
\eref{eqSk1} assuming that the time covariance function is
$\gamma(t)=H(2H-1)|t|^{2H-2}$,   $\frac 12 <H<1$, and the noise is
white in space, which means $\laa(x)$ is the Dirac delta function
$\delta_0(x)$.  The integral form of this Gaussian noise is a
two-parameter process which is a Brownian motion in space and a
fractional Brownian motion with Hurst parameter $H$ in time. This
equation has been studied in  \cite{HN}, where the existence of a
unique mild solution   has been proved,  and the following  formula
for  the moments of the solution  has been obtained
\begin{equation}\label{momwhite}
\be \lc u_{t,x}^k\rc = \be_B \lc \prod_{i=1}^k u_0(B_t^i+x)
\exp \lp  \alpha_H \sum_{1\leq i< j \leq
k}\int_0^t \int_0^t
|s-r|^{2H-2}\delta_0(B_s^i-B_r^j)dsdr\rp\rc\,,
\end{equation}
where $\alpha_H=H(2H-1)$.
Notice that in the above expression the exponent  is a sum of weighted intersection local times.
\end{remark}

\subsection {Time independent noise}
In this section we consider the following stochastic heat equation
in the Skorohod sense driven by  the multiplicative  time
independent noise introduced in Section  \ref{sec2.2}:
\begin{equation}\label{eqSk2}
\frac{\partial u}{\partial t}=\frac{1}{2}\Delta u+u\diamond
\frac{\partial^{d}W}{\partial x_1 \cdots \partial
x_d}\,.
\end{equation}
The notion of mild solution based on the Skorohod integral is similar to Definition \ref{def1}.
\begin{definition}\label{def3}
An adapted random field $u=\{u_{t,x}; t \geq 0, x \in \mathbb{R}^d\}$ such
that $\be [u^2_{t,x}] < \infty$ for all $(t,x)$ is a mild solution
to equation \eref{eqSk2} with initial condition   $u_0
\in C_b(\mathbb{R}^d)$, if for any $0\le s\le t, x\in
\mathbb{R}^d$,
 the process $\{ p_{t-s}(x-y) u_{s,y};  y \in \mathbb{R}^d\}$ is Skorohod
integrable  in the sense given by relation \eqref{eq:duality-time-indep}, and the following equation holds:
\begin{equation}
u_{t,x}=  p_t u_0(x)+  \int_0^t \left(
\int_{\mathbb{R}^d} p_{t-s}(x-y)u_{s,y}\delta W_{y} \right) ds.
\end{equation}
\end{definition}

Suppose that $u=\{u_{t,x}; t\geq 0, x \in \R^d\}$ is a mild solution to equation \eref{eqSk2}. Then for any fixed $(t,x)$, the random variable $u_{t,x}$ admits the following Wiener chaos expansion:
\begin{equation}
u_{t,x}=\sum_{n=0}^{\infty}I_n(f_n(\cdot,t,x))\,,
\end{equation}
where for each $(t,x)$, $f_n(\cdot,t,x)$ is a symmetric element in
$\mathcal{H}^{\otimes n}$. Notice that here the space $\mathcal{H}$ contains functions of the space variable $y$ only.
Using an iteration procedure  similar to the one described at Section \ref{sec:ex-uniq-chaos}, one can find the
explicit formula for the kernels $f_n$ for $n \geq 1$:
\begin{multline*}
f_n(x_1,\dots,x_n,t,x) \\
=\frac{1}{n!}\int_{[0,t]^n}
p_{t-s_{\si(n)}}(x-x_{\si(n)})\cdots
p_{s_{\si(2)}-s_{\si(1)}}(x_{\si(2)}-x_{\si(1)})  
\, p_{s_{\sigma(1)}}u_0(x_{\sigma(1)}) \, ds_1 \cdots ds_n\,,
\end{multline*}
where $\si$ denotes the permutation of $\{1,2,\dots,n\}$ such that $0<s_{\si(1)}<\cdots<s_{\si(n)}<t$.
Then, to show the existence and uniqueness of the solution it suffices to show that for all $(t,x)$ we have
\begin{equation}\label{chaos2}
\sum_{n=0}^{\infty}n!\|f_n(\cdot,t,x)\|^2_{\mathcal{H}^{\otimes n}}< \infty\,.
\end{equation}
\begin{theorem}
Assume that  $\mu$  satisfies  Hypothesis \ref{hyp:mu}. Then
\eref{chaos2} holds for each $(t,x)$ and  equation \eref{eqSk2} has
a unique mild solution.
\end{theorem}
The proof of this theorem is analogous to the proof of Theorem  \ref{thmSk1} and is omitted for sake of conciseness.
As in  the previous subsection, we can  deduce the following moment formula for the solution to equation \eref{eqSk2}.
\begin{equation}\label{momSk2}
\be \lc u^k_{t,x}\rc=\be_B  \lc  \prod_{i=1}^k
u_0(B_t^i+x)\exp \lp\sum_{1\leq i < j \leq k}\int_0^t \int_0^t
\laa(B_s^i-B_r^j)dsdr\rp\rc\,,
\end{equation}
where $B^i$, $i=1,\dots, k$,  are  $d$-dimensional  independent Brownian motions.

\section{Feynman-Kac functional}\label{sec:eq-strato}
In this section we construct a candidate solution for equation  (\ref{e1}) using a suitable version of Feynman-Kac formula. The construction has been inspired by the approach developed in~\cite{{Song Jian}} for the case of fractional noises.  We will establish the existence and H\"older continuity properties of the Feynman-Kac functional.

\subsection{Construction of the Feynman-Kac functional}
We first consider the  time dependent noise introduced in Section \ref{sec2.1}, and later we deal with the time independent noise introduced in Section  \ref{sec2.2}.
\subsubsection{Time dependent noise}
Suppose first  that $W$ is the time dependent noise introduced in Section  \ref{sec2.1}. If the initial condition of equation (\ref{e1}) is a continuous and bounded function
 $u_0$, analogously to \cite{Song Jian} we define
\begin{equation}  \label{FK}
u_{t,x}=\be_{  B} \lc   u_0(B^x_t)  \exp\left( \int_0^t
\int_{\R^d}\delta_0(B_{t-r}^x-y)W(dr,dy)\right)\rc \,,
\end{equation}
where $B^x$ is a $d$-dimensional Brownian motion independent of $W$ and starting at $x\in \R^d$.  Our first goal is thus to give a meaning to the functional
\begin{equation}\label{Vtx}
V_{t,x}=\int_0^t \int_{\R^d}\delta_0(B_{t-r}^x-y)W(dr,dy)\,
\end{equation}
appearing in the exponent of the Feynman-Kac formula \eqref{FK}.    To this aim, like in the case of the formula for moments (see \eqref{regW}), we will proceed by approximation. Namely,  we will approximate $V$ by the process
\begin{equation}\label{equ3}
V^\varepsilon_{t,x}=\int_0^t \int_{\R^d}p_\ep(B_{t-r}^x-y)W(dr,dy)\,
,\quad \varepsilon >0\,,
\end{equation}
which is well defined as a Wiener integral for a fixed path of the
Brownian motion $B$. The convergence of the approximation $V^\ep$ is
obtained in the next  proposition, for which we need to impose the following
conditions on the function $\gamma$ and the measure $\mu$.

\begin{hypothesis}\label{hyp:mu2}
There exists  a constant   $0 < \beta < 1$  such that for any  $t\in \R$,
\begin{equation}  \label{t1}
0 \le \gamma(t)\leq C_{\beta}|t|^{-\beta}
\end{equation}
for some constant  $C_{\beta}>0$,  and the measure $\mu$ satisfies
\begin{equation}   \label{t2}
\int_{\R^d}\frac{\mu(d\xi)}{1+|\xi|^{2-2\beta}}< \infty\,.
\end{equation}
\end{hypothesis}

\begin{proposition}\label{propdelta}
Let $V^\ep_{t,x}$ be the functional defined in (\ref{equ3}) and suppose that Hypothesis \ref{hyp:mu2} holds. Then for fixed
$t\ge 0$ and $x\in \R^d$, the random variable $V^\ep_{t,x}$  converges in $L^2(\Omega)$ towards a functional denoted by
$V_{t,x}$. Moreover, conditioned by $B$, $V_{t,x}$ is a Gaussian random variable with mean $0$ and variance
\begin{equation}\label{var}
{\bf Var}_W(V_{t,x})=\int_0^t \int_0^t \gamma(r-s)\laa(B_r-B_s)drds\,.
\end{equation}
\end{proposition}

\begin{proof}
Our first goal is to find
\begin{equation}  \label{m1}
\lim_{\ep_{1},\ep_{2}\to 0} \be\lc V_{t,x}^{\ep_{1}}\, V_{t,x}^{\ep_{2}} \rc.
\end{equation}
To this aim, we set  $A^\ep_{t,x}(r,y)=p_{\ep}(B_{t-r}^{x}-y) \mathbf{1}_{[0,t]}(r) $. Then
\begin{eqnarray*}
\be\lc V_{t,x}^{\ep_{1}}\, V_{t,x}^{\ep_{2}} \rc &=&
 \be\lc W\lp   A^{\ep_1}_{t,x} \rp  W\lp A^{\ep_2}_{t,x} \rp \rc
= \be_{B}\lc \lla A^{\ep_1}_{t,x},\,  A^{\ep_2}_{t,x} \rra_{\ch} \rc  \\
&=&   \be_B\lc \int_0^t \int_0^t \int_{\R^{d}}
\cf A^{\ep_1}_{t,x} (u,\cdot)(\xi) \, \overline{\cf
A^{\ep_2}_{t,x}}(v,\cdot) (\xi)\gamma(u-v) \mu(d\xi) dudv \rc \, .
\end{eqnarray*}
Furthermore, we can write for $u\le t$
\begin{equation*}
\cf A^{\ep_1}_{t,x}(u,\cdot)(\xi)
= \cf p_{\ep_1}(B_{t-u}^{x}-\cdot) (\xi)
= e^{-\frac 12 \ep_1^{2}|\xi|^{2}  +\imath \lla \xi ,B_{u}^{x}\rra}  ,
\end{equation*}
and thus
\begin{equation}\label{eq:cov-varphi-epsilon}
\lla A^{\ep_1}_{t,x},\,  A^{\ep_2}_{t,x} \rra_{\ch}
=
\int_{\R^{d}} \lp \int_{[0,t]^{2}}e^{\imath \lla \xi ,  \, B_{v}-B_{u}\rra} \gamma(u-v) dudv  \rp
e^{-\frac 12 (\ep_{1}^{2}+\ep_{2}^{2})|\xi|^{2} } \mu(d\xi).
\end{equation}
This yields
\begin{eqnarray}\label{eq:cov-Vt-identity}
\be\lc V_{t,x}^{\ep_{1}}\, V_{t,x}^{\ep_{2}} \rc
&=& \be_{B}\lc \lla A^{\ep_1}_{t,x},\,  A^{\ep_2}_{t,x} \rra_{\ch} \rc  \notag\\
&=& \int_{\R^{d}} \lp \int_{[0,t]^{2}}e^{-\frac 12 |\xi|^{2}|v-u|} \gamma(u-v) dudv \rp
e^{-\frac 12 (\ep_{1}^{2}+\ep_{2}^{2})|\xi|^{2}} \mu(d\xi).
\end{eqnarray}
Set now
\begin{equation*}
\si_{t}^{2} :=
\int_{\R^{d}} \lp \int_{[0,t]^{2}}e^{-\frac12|\xi|^{2}|v-u|} \gamma(u-v)dudv\rp
\, \mu(d\xi).
\end{equation*}
 Is easily seen by direct integration and by using the
hypothesis  (\ref{t1}) that
\begin{equation*}
\int_{[0,t]^{2}}e^{-\frac12 |\xi|^{2}|v-u|}  \gamma(u-v) dudv \le   c_\beta\int_{[0,t]^{2}}e^{-\frac12|\xi|^{2}|v-u|}  |u-v|^{-\beta} dudv
\le
\frac{c}{1+|\xi|^{2-2\beta}}.
\end{equation*}
Thus
\begin{equation*}
\si_{t}^{2} \le c \,
\int_{\R^{d}} \frac{\mu(d\xi)}{1+|\xi|^{2-2\beta}} \,,
\end{equation*}
which is a finite quantity by  hypothesis (\ref{t2}). As a
consequence, for every sequence $\ep_n$ converging to zero,
$V_{t,x}^{\varepsilon_n}$ converges in $L^2$ to a limit denoted by
$V_{t,x}$ which does not depend on the choice of the sequence
$\varepsilon_n$. Finally, by a similar argument, we show
(\ref{var}). This completes the proof of the proposition.
\end{proof}

\begin{remark} \label{rem4.3}
We could also regularize the noise in time, and define
\begin{equation} \label{rem5}
V^{\ep, \delta}_{t,x} = W( A^{\ep,\delta}_{t,x}),
\end{equation}
where $A^{\ep,\delta}_{t,x}$ has been introduced in (\ref{m3}). Then it is easy to check that  $V^{\ep, \delta}_{t,x}$ converges as $\delta$ tend to zero in $L^2(\Omega)$ to $V^{\ep}_{t,x}$.
\end{remark}

In  order to give a meaning to formula (\ref{FK}) we need to establish the existence of 
exponential moments for $V_{t,x}$.  To  complete this task,  we need
the following lemma.

\begin{lemma}   \label{lemeps}
Suppose that Hypothesis \ref{hyp:mu2} holds. Then for any
$\ep>0$ there exists a constant $C_{\ep}$ such that for any $v>0$
we have:
\begin{equation}\label{eq:bnd-exp-mu}
\sup_{\eta\in \R^d} \int_{\R^{d}} e^{-\frac{v}{2}|\xi-\eta|^{2}} \,
\mu(d\xi) \le C_{\ep} + \frac{\ep}{v^{1-\beta}}.
\end{equation}
\end{lemma}

\begin{proof}
 The fact that the left hand side of \eqref{eq:bnd-exp-mu} is uniformly bounded in $\eta$ is proven similarly to \eqref{eq4}, but is included here for sake of readability. 
Indeed, consider $\eta\in\R^{d}, v>0$ and define a function
$\vp_{\eta}:\R^{d}\to\R_{+}$ by
$\vp_{\eta}(\xi)=e^{-\frac{v}{2}|\xi-\eta|^{2}}$.  Then according to
Parseval's identity we have
\begin{equation*}
\int_{\R^{d}} \vp_{\eta}(\xi) \, \mu(d\xi) = c \int_{\R^{d}}
\cf\vp_{\eta}(x) \, \laa(x)dx = c \int_{\R^{d}}v^  {-d/2 }
e^{-\frac{|x|^{2}}{2 v}} \, e^{\imath\langle x, \, \eta\rangle}\,
\laa(x)dx.
\end{equation*}
We now use the fact that $\laa$ is assumed to be   nonnegative
in order to get the following uniform bound in $\eta$
\begin{equation*}
\int_{\R^{d}} \vp_{\eta}(\xi) \, \mu(d\xi) \le c
\int_{\R^{d}}v^ {-d/2} e^{-\frac{|x|^{2}}{2 v}} \,
\laa(x)dx = \int_{\R^{d}} \vp_{0}(\xi) \, \mu(d\xi) = \int_{\R^{d}}
e^{-\frac{v}{2}|\xi|^{2}} \, \mu(d\xi) .
\end{equation*}
To  estimate  the right-hand side of the above inequality we
introduce a constant $M>0$, whose exact value is irrelevant for our
computations, and let $c_{M,1}=\mu(B(0,M))$, where $B(0,M)$ stands
for the ball of radius $M$ centered at 0 in $\R^{d}$. Then the trivial bound $e^{-\frac{v}{2}|\xi|^{2}} \le 1$ yields 
\begin{equation*}
\int_{\R^{d}} e^{-\frac{v}{2}|\xi|^{2}} \, \mu(d\xi)
\le c_{M,1} + \int_{|\xi|>M} e^{-\frac{v}{2}|\xi|^{2}} \, \mu(d\xi).
\end{equation*}
Invoking the fact that the function $x \mapsto x^{1-\beta}e^{-x}$ is bounded on $\R_{+}$, we thus get
\begin{equation*}
\int_{|\xi|>M} e^{-\frac{v}{2}|\xi|^{2}} \, \mu(d\xi)  \le
\frac{c_{2}}{v^{1-\beta}} \int_{|\xi|>M}  \, \frac{\mu(d\xi)}{
|\xi|^{2-2\beta }}   \le  \frac{c_{2}}{v^{1-\beta}} \int_{|\xi|>M}
\, \frac{\mu(d\xi)}{1+|\xi|^{2-2\beta }}.
\end{equation*}
Summarizing the above, we have obtained that
\begin{equation*}
\int_{\R^{d}} e^{-\frac{v}{2}|\xi-\eta|^{2}} \, \mu(d\xi) \le
c_{M,1} + \frac{c_{2}}{v^{ 1-\beta}} \int_{|\xi|>M}  \,
\frac{\mu(d\xi)}{1+|\xi|^{2-2\beta}},
\end{equation*}
uniformly in $\eta\in\R^{d}$. Our claim is thus obtained by choosing
$M$ large enough so that $c_{2}\int_{|\xi|>M}  \,
\frac{\mu(d\xi)}{1+|\xi|^{2-2\beta}} \leq \ep$, which is possible by
 hypothesis \eref{t2}.
\end{proof}

\smallskip

The following elementary integration result will also be crucial for the moment estimates  we deduce later.

\begin{lemma}\label{lemsimplex}
Let $\alpha \in (-1+\ep, 1)^m$  with $\ep>0$ and  set $|\alpha |= \sum_{i=1}^m
\alpha_i  $. Recall (see (\ref{simplex})) that
$T_m(t)=\{(r_1,r_2,\dots,r_m) \in \R^m: 0<r_1  <\cdots < r_m < t\}$.
Then there is a constant $\kappa$ such that
\[
J_m(t, \alpha):=\int_{T_m(t)}\prod_{i=1}^m (r_i-r_{i-1})^{\alpha_i}
dr \le \frac { \kappa^m t^{|\alpha|+m } }{ \Gamma(|\alpha|+m +1)},
\]
where by convention, $r_0 =0$. 
\end{lemma}

\begin{proof}
 Using identities on Beta functions and a recursive algorithm we can snow that
 \[
 J_m(t, \alpha)=\frac{\prod_{i=1}^m \Gamma (\alpha_i +1)t^{| \alpha| +m }}{\Gamma(|\alpha|+m +1)}\,,
\]
and the result follows thanks to the fact that the $\Gamma$ function is bounded on $(\ep,2)$. 
\end{proof}

With these preliminary results in hand, we can now prove the exponential integrability of  the random
variable $V_{t,x}$ defined in  Proposition \ref{propdelta}.

\begin{theorem}\label{thmexp}
Let $V_{t,x}$ be the functional defined in  Proposition \ref{propdelta}, and assume Hypothesis \ref{hyp:mu2}. Then for any $\lambda\in\R$  and $T>0$, we have $\sup_{t\in [0,T], \, x\in \R^d} \be[\exp(\lambda \, V_{t,x})] < \infty$. In particular, the functional  \eqref{FK} is well defined.
\end{theorem}

\begin{proof}  Fix $t>0$ and $x\in \R^d$. Conditionally to $B$, the random
variable $V_{t,x}$ is Gaussian and centered.
 From (\ref{var}), we obtain
\begin{equation*}
\be \left[\exp(\lambda V_{t,x} )\right]=\be_B \lc \exp
\left(\frac{\lambda^2}{2}\int_0^t \int_0^t
\gamma(r-s)\laa(B_r-B_s)drds\right)\rc=\be_B \lc \exp \left(\frac{\lambda^2}{2}Y\right)\rc\,,
\end{equation*}
where
\begin{equation*}
Y=\int_0^t \int_0^t \gamma(r-s)\laa(B_r-B_s)drds\,.
\end{equation*}
 In order to show that $\be \lc \exp (\lambda Y) \rc<\infty$ for any $\lambda \in \R$, we are going to use an elaboration of a method introduced by Le Gall \cite{Le Gall} (see also \cite{HN, Song Jian}). With respect to those contributions, our case requires a careful handling of the weights $\Lambda$ and $\ga$. Notice in particular that in our general setting we do not have scaling properties, and some additional work is necessary to overcome this difficulty.

\smallskip
Le Gall's method starts from the following construction: for $n\ge 1$ and $k=1,\ldots,2^{n-1}$ we set 
\begin{equation*}
J_{n,k} := \lc \frac{(2k-2) \, t}{2^{n}} , \, \frac{(2k-1) \, t}{2^{n}} \rp, \quad
I_{n,k} := \lc \frac{(2k-1) \, t}{2^{n}} , \, \frac{2k \, t}{2^{n}} \rp,
\quad\text{and}\quad
A_{n,k} := J_{n,k} \times I_{n,k}.
\end{equation*}
Notice then that $\{A_{n,k}; \, n\ge 1, k=1,\ldots,2^{n-1}\}$ is a partition of the simplex $T_{2}(t)$, and in addition $I_{n,k-1} \cap I_{n,k}=\varnothing$ and  $J_{n,k-1} \cap J_{n,k}=\varnothing$ (see Figure \ref{fig:legall} for an illustration). 
We can thus write
\begin{equation*}
Y =  \sum_{n=1}^{\infty} \sum_{k=1}^{2^{n-1}} a_{n,k},
\quad\text{where}\quad
a_{n,k} = 
\int_{A_{n,k}}\gamma(r-s)\laa(B_r-B_s)dr ds\,.
\end{equation*}
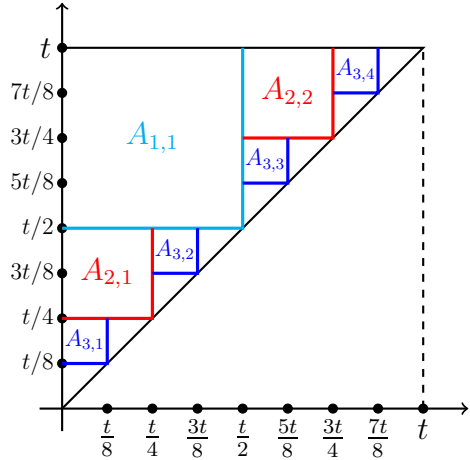
\begin{figure}[htbp]
\begin{center}
\caption{Le Gall's partition of $T_{2}(t)$ into disjoint rectangles of decreasing area.}
\label{fig:legall}
\end{center}
\begin{center}
\begin{tikzpicture}[xscale=0.3,yscale=0.3]

\draw [->,  thick] (-1,0) --(18,0);
\draw [->,  thick] (0,-1) --(0,18);
\draw [dashed,  thick] (16,0) --(16,16);

\draw [fill,black] (16,0) circle [radius=.2];
\node [below,scale=1.0, black] at (16,0) {$t$}; 
\draw [fill,black] (2,0) circle [radius=.2];
\node [below,scale=1.0, black] at (2,0) {$\frac{t}{8}$}; 
\draw [fill,black] (4,0) circle [radius=.2];
\node [below,scale=1.0, black] at (4,0) {$\frac{t}{4}$}; 
\draw [fill,black] (6,0) circle [radius=.2];
\node [below,scale=1.0, black] at (6,0) {$\frac{3t}{8}$}; 
\draw [fill,black] (8,0) circle [radius=.2];
\node [below,scale=1.0, black] at (8,0) {$\frac{t}{2}$}; 
\draw [fill,black] (10,0) circle [radius=.2];
\node [below,scale=1.0, black] at (10,0) {$\frac{5t}{8}$}; 
\draw [fill,black] (12,0) circle [radius=.2];
\node [below,scale=1.0, black] at (12,0) {$\frac{3t}{4}$}; 
\draw [fill,black] (14,0) circle [radius=.2];
\node [below,scale=1.0, black] at (14,0) {$\frac{7t}{8}$}; 

\draw [fill,black] (0,16) circle [radius=.2];
\node [left,scale=1.0, black] at (0,16) {$t$}; 
\draw [fill,black] (0,2) circle [radius=.2];
\node [left,scale=0.75, black] at (0,2) {$t/8$}; 
\draw [fill,black] (0,4) circle [radius=.2];
\node [left,scale=0.75, black] at (0,4) {$t/4$}; 
\draw [fill,black] (0,6) circle [radius=.2];
\node [left,scale=0.75, black] at (0,6) {$3t/8$}; 
\draw [fill,black] (0,8) circle [radius=.2];
\node [left,scale=0.75, black] at (0,8) {$t/2$}; 
\draw [fill,black] (0,10) circle [radius=.2];
\node [left,scale=0.75, black] at (0,10) {$5t/8$}; 
\draw [fill,black] (0,12) circle [radius=.2];
\node [left,scale=0.75, black] at (0,12) {$3t/4$}; 
\draw [fill,black] (0,14) circle [radius=.2];
\node [left,scale=0.75, black] at (0,14) {$7t/8$}; 

\draw [ thick] (0,0)--(16,16)--(0,16);

\draw [very thick, cyan] (0,8) -- (8,8) -- (8,16);
\node [scale=1.0, cyan] at (4,12) {$A_{1,1}$}; 

\draw [very thick, red] (0,4) -- (4,4) -- (4,8);
\node [scale=1.0, red] at (2,6) {$A_{2,1}$}; 
\draw [very thick, red] (8,12) -- (12,12) -- (12,16);
\node [scale=1.0, red] at (10,14) {$A_{2,2}$};

\draw [very thick, blue] (0,2) -- (2,2) -- (2,4);
\node [scale=0.75, blue] at (1,3) {$A_{3,1}$}; 
\draw [very thick, blue] (4,6) -- (6,6) -- (6,8);
\node [scale=0.75, blue] at (5,7) {$A_{3,2}$}; 
\draw [very thick, blue] (8,10) -- (10,10) -- (10,12);
\node [scale=0.75, blue] at (9,11) {$A_{3,3}$}; 
\draw [very thick, blue] (12,14) -- (14,14) -- (14,16);
\node [scale=0.75, blue] at (13,15) {$A_{3,4}$}; 

\end{tikzpicture}

\end{center}
\end{figure}

Observe that for fixed $n$ the random variables $\{a_{n,k};\, k=1,\ldots,2^{n-1}\}$ are independent, owing to the fact that they depend on the increments of $B$ on disjoint sets.
Now, thanks to the fact that $J_{n,k} \cap I_{n,k}=\varnothing$, for all $(r,s)\in A_{n,k}$ we can decompose $B_{r}-B_{s}$ into $(B_{r}-B_{\frac{(2k-1) \, t}{2^{n}}}) - (B_{s}-B_{\frac{(2k-1) \, t}{2^{n}}})$, where the two pieces of the difference are independent Brownian motions. Thus the following identity in law holds true:
\begin{equation*}
\left\{  B_{r}-B_{s} ; \, (r,s) \in A_{n,k} \right\}
\stackrel{(d)}{=}
\lcl  B_{r-\frac{(2k-1) \, t}{2^{n}}}- \widetilde {B}_{s-\frac{(2k-1) \, t}{2^{n}}} ; \, (r,s) \in A_{n,k}
 \rcl,
\end{equation*}
where $B$ and $ \widetilde {B}$ are two independent Brownian motions. With an additional change of variables $r-\frac{(2k-1)t}{2^n} \mapsto r$ and
$\frac{(2k-1)t}{2^n}-s \mapsto s$, this easily yields the following identity
\begin{eqnarray*}
a_{n,k}
&\stackrel{(d)}{=}&
\int_{A_{n,k}}  \gamma(r+s) \, 
\laa\left(B_{\frac{(2k-1)t}{2^n}+r}-\widetilde{B}_{\frac{(2k-1)t}{2^n}-s}\right)\, dsdr  \\
&\stackrel{(d)}{=}&
\int_0^{\frac{t}{2^n}}   \int_0^{\frac{t}{2^n}}
\gamma(r+s) \, \laa(B_r+\widetilde{B}_s) \, dsdr := a_{n}.
\end{eqnarray*}
Summarizing the considerations above, we have found that 
\begin{equation}\label{eq:dcp-Y}
Y
= 
 \sum_{n=1}^{\infty} \sum_{k=1}^{2^{n-1}} a_{n,k},
\end{equation}
where for each $n\ge 1$ the collection $\{a_{n,k};\, k=1,\ldots,2^{n-1}\}$ is a family of independent random variables such that 
\begin{equation*}
a_{n,k} \stackrel{(d)}{=} a_{n},
\quad\text{with}\quad
a_{n} \:=
\int_0^{\frac{t}{2^n}}   \int_0^{\frac{t}{2^n}}
\gamma(r+s) \, \laa(B_r+\widetilde{B}_s) \, dsdr,
\end{equation*}
where $B,\widetilde{B}$ are two independent Brownian motions. Notice that the transformation of $B_{r}-B_{s}$ into $B_r+\widetilde{B}_s$ we have achieved is essential for our future computations. Indeed, it will be translated into some singularities $(r-s)^{-1}$ in a neighborhood of 0 in $\R_{+}^{2}$ becoming some more harmless singularities of the form $(r+s)^{-1}$.
The proof is now decomposed in several steps.

\medskip
\noindent \textit{Step 1.} First we need to estimate the moments of
the random variable   $a_{n}$. We claim that for any
$\varepsilon>0$ there exist  constants $C_{\varepsilon,1}>0$ and
$C_2>0$  (which depend on $t$) such that
\begin{equation}  \label{equ5}
\be[ a_{n} ^m] \le C_{\varepsilon,1} m! \left(\frac{C_2 \varepsilon }{2^{n}}\right)^m.
\end{equation}
In order to show (\ref{equ5}), we first write
\begin{eqnarray*}
\be \lc a_{n} ^m\rc=\int_{[0,\frac{t}{2^n}]^m}
\int_{[0,\frac{t}{2^n}]^m}  \prod_{i=1}^m \gamma(r_i+s_i)\be
\lc\prod_{i=1}^m \laa (B_{r_i}+\widetilde{B}_{s_i})\rc dsdr  \,.
\end{eqnarray*}
Let $p_B$ be the joint density of $(B_{r_1}+\widetilde{B}_{s_1}, \dots,
B_{r_m}+\widetilde{B}_{s_m})$, which is a Schwartz function. Hence,
  using the Fourier transform  and the same considerations as for \eqref{eq:charac1}, we get 
\begin{eqnarray*}
\be \lc \prod_{i=1}^m
\laa(B_{r_i}+\widetilde{B}_{s_i})\rc =\int_{\R^{md}}\prod_{i=1}^m
\laa(x_i)p_B(x)dx=\int_{\R^{md}}  e^{-\frac{1}{2}\sum_{i,j=1}^m \xi_i\cdot \xi_j (r_i\wedge
r_j + s_i\wedge s_j)}\prod_{i=1}^m
\mu(d\xi_i)\,.
\end{eqnarray*}
 We now proceed as in the proof of Theorem \ref{thmmom1}, with an additional care in the computation of terms. Thanks to our assumption \eqref{t1} on $\gamma$  and the basic inequality $a+b \geq
2\sqrt{ab}$ for nonnegative $a$, $b$,  we have
\begin{align*}
&\be \lc a_{n} ^m\rc =\int_{[0,\frac{t}{2^n}]^m}
\int_{[0,\frac{t}{2^n}]^m}
\int_{\R^{md}} e^{-\frac{1}{2}\sum_{i,j=1}^m
\xi_i\cdot \xi_j(r_i\wedge r_j +s_i\wedge s_j)} \prod_{i=1}^m\mu(d\xi_i)  \prod_{i=1}^m
\gamma(r_i+s_i) dsdr\\
&\leq (2^{-\beta }C_{\beta})^m\int_{[0,\frac{t}{2^n}]^m}  \int_{[0,\frac{t}{2^n}]^m}
\int_{\R^{md}} e^{-\frac{1}{2}\sum_{i,j=1}^m
\xi_i\cdot \xi_j (r_i\wedge r_j+ s_i\wedge s_j)}  \prod_{i=1}^m \mu(d\xi_i) 
 \prod _{i=1}^m (r_is_i)^{-\frac{\beta}{2}} dsdr  ,
\end{align*}
and thus, invoking Cauchy-Schwarz inequality
with respect to the measure $\prod_{i=1}^m
(r_is_i)^{-\frac{\beta}{2}} dr ds$, we end up with
\begin{multline*}
\be \lc a_{n} ^m\rc 
\leq (2^{-\beta}C_{\beta })^m \int_{[0, \frac{t}{2^n}]^m}
\int_{[0, \frac{t}{2^n}]^m}  \left(\int_{\R^{md}} e^{-\sum_{i,j=1}^m
\xi_i\cdot \xi_j (r_i\wedge r_j)} \prod_{i=1}^m \mu(d\xi_i)
\right)^{\frac{1}{2}}\\ 
\times\left(\int_{\R^{md}}
e^{-\sum_{i,j=1}^m \xi_i\cdot \xi_j (s_i\wedge s_j)}\prod_{i=1}^m
\mu(d\xi_i) \right)^{\frac{1}{2}}\prod_{i=1}^m
(r_is_i)^{-\frac{\beta}{2}} dsdr.
\end{multline*}
Since in the above expression, both integrals with respect to the
measure $\prod_{i=1}^m \mu (d\xi_i)$ are  symmetric functions of the
$r_i$'s and $s_i$'s, we can restrict the integral to the region
$T_m(\frac{t}{2^n})$, where $T_m(t)$ has been defined in
(\ref{simplex}).  Therefore,  similarly to \eqref{eq:mAn} and with  the convention $r_0=0$, we obtain that for any $\varepsilon>0$ the expectation $\be [ a_{n} ^m]$ is bounded by 
\begin{align*}
& (2^{-\beta }C_{\beta })^m (m!)^2 
 \left(\int_{T_m(\frac{t}{2^n})}  \left(
\int_{\R^{md}} e^{-\sum_{i=1}^m(r_i-r_{i-1})|\xi_i+\cdots
+\xi_m|^2} \prod_{i=1}^m\mu(d\xi_i)\right)^{\frac{1}{2}}\prod_{i=1}^m |r_i|^{-\frac{\beta}{2}} dr\right)^2  \\
&\leq (2^{-\beta}C_{\beta})^m (m!)^2
\left(\int_{T_m(\frac{t}{2^n})} \prod_{i=1}^m
\left(C_{\varepsilon}+\frac{\varepsilon}{(r_i-r_{i-1})^{1-\beta}}\right)^{\frac{1}{2}}\prod_{i=1}^m
(r_i-r_{i-1})^{-\frac{\beta}{2}}dr\right)^2\,,
\end{align*}
where we have used Lemma \ref{lemeps} and 
we have bounded
$r_i^{-\frac{\beta}{2}}$ by $(r_i-r_{i-1})^{-\frac{\beta}{2}}$. We now resort to the
inequality $(a+b)^{\frac{1}{2}} \leq
a^{\frac{1}{2}}+b^{\frac{1}{2}}$ in order to get
\begin{eqnarray*}
\be \lc   a_{n} ^m \rc &\le & (2^{-\beta }C_{\beta })^m (m!)^2\left(
\int_{T_m(\frac{t}{2^n})}
\prod_{i=1}^m\left(\sqrt{C_{\varepsilon}}+\frac{\sqrt{\varepsilon}}{(r_i-r_{i-1})^{\frac{1-\beta
}{2}}}\right)\prod_{i=1}^m (r_i-r_{i-1})^{-\frac{\beta
}{2}}dr\right)^2\\ &=&(2^{-\beta }C_{\beta})^m (m!)^2
\left(\int_{T_m(\frac{t}{2^n})} \sum_{\theta \in
\{0,1\}^m}\prod_{i=1}^m C_{\varepsilon}^{\frac{\theta_i}{2}
}\varepsilon^{\frac{1}{2}(1-\theta_i)}(r_i-r_{i-1})^{  \frac{\beta
-1}{2}(1-\theta_i)-\frac{\beta}{2}}dr\right)^2\,.
\end{eqnarray*}
 Hence, a direct application of Lemma \ref{lemsimplex}   shows that  there exists a positive constant $K$ such that
\begin{equation*}
\be \lc a_{n} ^m\rc 
\le  
K ^{m} (m!)^2 \left( \sum_{l=0}^m {m\choose l}
C_{\varepsilon}^{\frac{l}{2}}\varepsilon^{\frac{m-l}{2}}\frac{
(\frac{t}{2^n})^{\frac{(1-\beta)l}{2}+\frac{m}{2}}}{\Gamma
(\frac{(1-\beta)l}{2}+\frac{m}{2}+1)}\right)^2.
\end{equation*}
 We now simply bound ${m\choose l}$ by $2^{m}$ and recall that $(x/3)^{x}\le \Gamma(x+1)\le x^{x}$ for $x\ge 0$. Allowing the constant $K$ to change from line to line, this yields 
\begin{eqnarray*}
\be \lc a_{n} ^m\rc  
&\leq&K ^{m}
(m!)^2\left(\sum_{l=0}^m
C_{\varepsilon}^{\frac{l}{2}}\varepsilon^{\frac{m-l}{2}}\frac{
\left(\frac{t}{2^n}\right)^{\frac{(1-\beta)l}{2}}(\frac{t}{2^n})^{\frac{m}{2}}}{(m!)^{\frac{1}{2}}
(l!)^{\frac{1-\beta}{2}}}
\right)^2\\ 
&\leq&
K ^{m} m!  \varepsilon^m  \left(\frac{t}{2^{n}}\right)^m  \left(
\sum_{l=0}^{\infty}
\frac{C_{\varepsilon}^{\frac{l}{2}}\varepsilon^{-\frac{l}{2}} t^{\frac{1-\beta}{2}l}
}
{(l!)^{\frac{1-\beta}{2}}}
\right)^2\,.
\end{eqnarray*}
This completes the proof of (\ref{equ5}) with
 $C_{\varepsilon,1}=  (  \sum_{l=0}^{\infty}C_{\varepsilon}^{\frac{l}{2}}
 \varepsilon^{-\frac{l}{2}}t^{\frac{1-\beta}{2}l}(l!)^{\frac{\beta-1}{2}})^2$, which is finite because this series is
 convergent,
  and  $C_2=    K t$.

\medskip
\noindent
\textit{Step 2.}
 We now start from relation \eqref{equ5} and prove the finiteness of exponential moments for the random variable $Y$. It turns out that centering is useful in this context, and we thus define $\overline{a}_{n,k}=a_{n,k}- \be[a_{n,k}]$. Then $\be [
\overline{a}_{n,k}]=0$, and for any integer $m \geq 2$ notice that:
\begin{eqnarray*}
\be \lc (\overline{a}_{n,k})^m\rc \leq 2^{m-1}\lp \be \lc a_{n,k}^m\rc + \lp  \be [
a_{n,k}] \rp ^m\rp \leq 2^m  \be [a_{n,k}^m]\,.
\end{eqnarray*}
 Also recall that $a_{n,k}\stackrel{(d)}{=}a_{n}$.  Thus, using (\ref{equ5})
\begin{eqnarray*}
\be \lc \exp (\lambda \overline{a}_{n,k})\rc &=&1+
\sum_{m=2}^{\infty}\frac{\lambda ^m}{m!} \be \lc (\overline{a}_{n,k})^m \rc
\le 1+\sum_{m=2}^{\infty}\frac{(2\lambda)^m}{m!} \be \lc (a_{n,k})^m\rc\\
&\leq&1+\sum_{m=2}^{\infty} C_{\varepsilon,1} 
\left( \frac{2 C_2 \lambda \varepsilon}{  2^{n}}\right)  ^m  \,.
\end{eqnarray*}
Now choose and fix $\varepsilon$ such that $C_2 \lambda \varepsilon  2^{-n+1} \leq \frac{1}{2}$, and we obtain the bound
\begin{equation}  \label{equ6}
\be \lc  \exp (\lambda \overline{a}_{n,k})\rc \leq 1 + \frac{C_{\varepsilon,2}  \lambda^2}{2^{2n}} ,
\end{equation}
for some positive constant $C_{\varepsilon,2}$. Next we choose
$0<h<1$, define $b_N=\prod_{j=2}^N\big(1-2^{-h(j-1)}\big)$, and
notice that $\lim_{N \to \infty}b_N=b_{\infty}> 0$. Then, by 
H\"older's inequality, for all $N \geq 2$ we have
\begin{eqnarray*}
&&\be \lc \exp \left(\lambda b_N \sum_{n=1}^N
\sum_{k=1}^{2^{n-1}}\overline{a}_{n,k}\right)\rc\\
&\leq&\left[\be \lc \exp \left(\lambda
b_{N-1}\sum_{n=1}^{N-1}\sum_{k=1}^{2^{n-1}}\overline{a}_{n,k}\right)\rc \right]^{1-2^{-h(N-1)}}\left[\be
\lc \exp\left(\lambda b_N
2^{h(N-1)}\sum_{k=1}^{2^{N-1}}\overline{a}_{N,k}\right)\rc\right]^{2^{-h(N-1)}},
\end{eqnarray*}
and taking into account the independence of the $\{a_{N,k}; k \le 2^{N-1}\}$ plus the identity $a_{N,k}\stackrel{(d)}{=}a_{N}$, the above expression is equal to
\begin{equation*}
\left(\be\lc\exp \left(\lambda
b_{N-1}\sum_{n=1}^{N-1}\sum_{k=1}^{2^{n-1}}\overline{a}_{n,k}\right)\rc\right)^{1-2^{-h(N-1)}}
\!\!
\left(\be
\lc\exp\left(\lambda b_N
2^{h(N-1)}\overline{a}_{N}\right)\rc\right)^{2^{(1-h)(N-1)}} 
:=A_N B_N\,.
\end{equation*}
We now appeal to the estimate (\ref{equ6}) and the elementary inequality $1+x \leq e^{x}$, valid for any $x \in
\R$. This yields 
\begin{eqnarray*}
B_N &\leq&\left(1+C_{\varepsilon,2} b_N^2  2^{-2N}\lambda^2 2^{2h(N-1)}\right)^{2^{(1-h)(N-1)}}
\leq\exp \left(C_{\varepsilon,3} 2^{-N(1-h)}\right),
\end{eqnarray*}
for some positive constant $C_{\varepsilon,3}$. 
Notice that this is where the centering argument on $a_{n,k}$ is crucial. Indeed, without centering  we would get a factor $2^{-N}$ instead of $2^{-2N}$ in the left hand side of the above expression, and $B_N$ would not be uniformly bounded.
Thus, recursively we get
\begin{eqnarray*}
\be \lc \exp \left(\lambda b_N \sum_{n=1}^N
\sum_{k=1}^{2^{n-1}}\overline{a}_{n,k}\right)\rc
\leq 
\exp
\left(\sum_{n=2}^N C 2^{-n(1-h)}\right)  \,
\be \lc \exp(\overline{a}_{1,1})\rc< \infty\,.
\end{eqnarray*}
 Recalling now from \eqref{eq:dcp-Y} that $Y-\be[Y]=\sum_{n=1}^{\infty} \sum_{k=1}^{2^{n-1}} \bar{a}_{n,k}$ and applying Fatou's lemma, we finally get 
\begin{eqnarray*}
\be \lc \exp\lp\lambda b_{\infty}(Y-\be [Y])\rp \rc< \infty ,
\end{eqnarray*}
which completes the proof.
\end{proof}

\smallskip
Our next result is an approximation result  for the Feynman-Kac functional which will be used in the next section (see Theorem \ref{thmFK1}). Towards this aim, for any $\ep, \delta>0 $  we define
\begin{equation}  \label{FKaprox}
u^{\varepsilon,\delta}_{t,x}=\be \lc u_0(B^x_t)\exp
(V_{t,x}^{\varepsilon,\delta})\rc\,,
\end{equation}
 where  $ V_{t,x}^{\varepsilon,\delta}=W(A^{\ep.\delta}_{t,x})$   and $A_{t,x}^{\varepsilon,\delta}$
 is defined
 in
(\ref{m2}).

\begin{proposition} \label{proplim}
For any $p\ge 2$ and $T>0$ we have
\begin{equation} \label{m9}
\lim_{\ep \downarrow 0} \lim_{ \delta \downarrow 0}  \sup_{(t,x) \in
[0,T] \times \R^d} \be \lc  |u^{\varepsilon,\delta}_{t,x}
-u_{t,x}|^p \rc=0.
\end{equation}
\end{proposition}

 \begin{proof}
First,  we recall that (see  Proposition \ref{propdelta} and Remark  \ref{rem4.3})  for any fixed $t\ge0 $ and $x\in \R^d$ the random variable   $V^{\varepsilon,\delta}_{t,x} $ converges in $L^2(\Omega)$ to
 $V_{t,x}$
 if we let first $\delta$ tend to zero and later $\ep$ tend to zero. Then in order to show (\ref{m9})  it suffices to check that for any  $\lambda \in \R$
 \begin{equation}   \label{m6}
 \sup_{\ep, \delta}\sup_{(t,x) \in [0,T] \times \R^d}  \be \lc \exp \lp \lambda V^{\varepsilon,\delta}_{t,x}    \rp     \rc <\infty.
 \end{equation}
 Taking first the expectation with respect to the noise $W$ yields
 \[
  \be \lc \exp \lp \lambda V^{\varepsilon,\delta}_{t,x}    \rp     \rc
  =\be_B \lc \exp \lp \frac {\lambda^2} 2 \| A^{\ep,\delta}_{t,x} \| ^2_{\mathcal{H}} \rp\rc.
  \]
  Expanding the exponential into a power series, we will need to bound  the moments of the
  random variable   $\| A^{\ep,\delta}_{t,x} \| ^2_{\mathcal{H}}$.
  To do this, we use  formula (\ref{m7}) with $B= \widetilde{B}$ and $\ep=\ep'$,  
  $\delta=\delta'$. Computing the mathematical expectation of this expression,
  we end up with:
  \begin{multline*}
  \be \lc \left  \| A_{t,x}^{\varepsilon,\delta} \right\|^{2n}  _{\mathcal{H}}\rc
=\frac 1{\delta^{2n}}  \int_{O_{\delta,n}} \int_{\R^{dn}}
\exp\lp- \frac 12 \sum_{i,j =1}^n  \be _B [(B_{u_i}-B_{v_i})(B_{u_j}-B_{v_j}) ] \langle \xi_i,\xi_j \rangle\rp\\
\times e^{-\varepsilon \sum_{l=1}^n|\xi_l|^2} \prod_{l=1}^n
\gamma(u_l+s_l-v_l-\widetilde{s}_l) \, \mu(d\xi) \, dsd\widetilde{s}dudv.
\end{multline*}
Thanks to the estimate
\begin{equation} \label{j3}
\sup_{0\le \delta \le 1} \frac 1{\delta^2} \int_0^\delta \int_0^\delta |u+ s -v-r|^{-\beta} dsdr \le c_{T,\beta} |u-v| ^{-\beta},
\end{equation}
 which holds  for any $u,v \in [0,T]$,  and  owing to assumption (\ref{t1}),  we get
\begin{equation} \label{m8}
  \be \lc \left  \| A_{t,x}^{\varepsilon,\delta} \right\|^{2n}  _{\mathcal{H}}\rc
\le c_{T,\beta}^n      \be_B \lc \left|  \iot \iot |u-v|^{-\beta} \laa(B_u-B_v) dudv \right|^{n}     \rc.
\end{equation}
It is now readily checked that  (\ref{m6}) follows from (\ref{m8}) and Theorem  \ref{thmexp}.
 \end{proof}

\subsubsection{Time independent noise}
Suppose that $W$ is the time independent noise introduced in Section \ref{sec2.2}. 
The Feynman-Kac functional is defined as
\begin{equation}\label{FK2}
u_{t,x}=\be \lc u_0(B^x_t)\exp \left(\int_0^t \int_{\R^d} \delta _0(B_{r}^x-y)W(dy)dr\right)\rc\,,
\end{equation}
where  $B^x=\{B_t+x,  t\ge 0\}$ is a $d$-dimensional Brownian motion
independent of $W$, starting from $x$m and $ u_0\in C_b(\mathbb{R}^d)$  is the initial condition. 

  As in the case of a time dependent noise,  to give a meaning to this functional for every $t>0$,  $x\in \mathbb{R}^d$ and $\varepsilon>0$ we introduce the family of random variables
\[
V_{t,x}^{\varepsilon}=\int_0^t \int_{\R^d} p_{\varepsilon}(B_{r}^x-y)W(dy)dr\,,
\]
  Then, if the spectral measure
of the noise $\mu$ satisfies  condition (\ref{mu1}), the family
$V_{t,x}^{\varepsilon}$ converges in $L^2$ to a limit denoted by
\begin{equation}\label{equ8}
V_{t,x}=\int_0^t \int_{\R^d} \delta_0(B_{r}^x-y) W(dy)dr\,.
\end{equation}
Conditional on $B$, $V_{t,x}$ is a Gaussian random variable with mean $0$ and variance
\begin{equation}
{\bf Var}_W (V_{t,x})=\int_0^t \int_0^t \laa(B_r-B_s)drds\,.
\end{equation}
Furthermore, for any $\lambda \in \R$,  we have $\be \lc \exp (\lambda V_{t,x})\rc< \infty$. These properties can be obtained using the same arguments as in the time dependent case and we omit the details.

\subsection{H\"older continuity of the Feynman-Kac functional} 
In this
subsection, we establish the H\"older continuity in space and time of
the the Feynman-Kac functional given by formulas (\ref{FK}) and
(\ref{FK2}). These regularity properties will hold under some
additional integrability assumptions on the measure $\mu$.  To
simplify the presentation we will assume that $u_0=1$,  and as usual we separate the time dependent and independent cases.

\subsubsection{Time dependent noise }
For the case
of a time dependent noise, we will make use of the following
condition  in order to ensure H\"older type regularities.

\begin{hypothesis}\label{hyp:mu-holder}
Let $W$ be a space-time stationary Gaussian noise with covariance
structure encoded by $\gamma$ and $\laa$. We assume that condition
(\ref{t1}) in Hypothesis \ref{hyp:mu2} holds for some $\beta>0$ and
the spectral measure $\mu$  satisfies \begin{equation*}
\int_{\R^{d}}\frac{\mu(d\xi)}{1+|\xi|^{2(1-\beta-\al)}} <\infty
\end{equation*}
for some $\al \in (0,1-\beta)$.
\end{hypothesis}

\begin{theorem}\label{thmHolder}
Assume Hypothesis \ref{hyp:mu-holder}.   Let $u$ be the process
introduced by relation \eqref{FK} with $u_0=1$, namely:
 \begin{equation}\label{eq:def-FK-initial-1}
u_{t,x} = \be_{B}\lc \exp\lp V_{t,x} \rp \rc, \quad\text{where}\quad
V_{t,x} = \int_{0}^{t}  \int_{\R^d} \delta_0(B^x_{t-r}-y) W(dr,dy).
\end{equation}
Then $u$ admits a version which is $(\ga_{1},\ga_{2})$-H\"older
continuous on any compact set of the form $[0,T]\times[-M,M]^{d}$,
with any $\ga_{1} < \frac{\al}{2}$, $\ga_{2}< \al$ and $T,M>0$.
\end{theorem}

\begin{proof}
Owing to standard considerations involving Kolmogorov's criterion,
it is sufficient to prove the following bound for all large $p$ and
$s,t \in [0,T] $, $x,y \in\R^{d}$ with $T>0$:
\begin{equation}\label{eq:bnd-moments-u}
\be\lc |u_{t,x} - u_{s,y}|^{p} \rc \le c_{p,T} \lp |t-s|^{\frac{\al
p}{2}} + |x-y|^{\al p} \rp .
\end{equation}
We now focus on the proof of (\ref{eq:bnd-moments-u}). From   the elementary relation $|e^{x}-e^{y}|\le (e^{x }+e^{ y})
|x-y|$, valid for $x,y\in\R$ and applying the Cauchy-Schwarz inequality  it
follows
\begin{eqnarray}\label{eq:reduc-to-bnd-V}
\be\lc |u_{t,x} - u_{s,y}|^{p} \rc &=& \be_{W}\lc \left|\be_{B}\lc
\exp\lp V_{t,x}) \rp \rc
- \be_{B}\lc \exp\lp V_{s,y}) \rp \rc\right|^{p} \rc  \notag \\
&\le&
 \be_W  \lcl \be_B^{p}  \lc   \left(\exp  (V_{t,x})+\exp (V_{s,y}) \right) |V_{t,x} -V_{s,y}|\rc \rcl \\
&\leq&\be_W^{1/2} \! \lcl \be_B^{p} \left[   \left(\exp(V_{t,x})+\exp
(V_{s,y})\right)^2\right]  \rcl
\be_W^{1/2} \! \lcl
\be_B^{p} \left[
|V_{t,x}-V_{s,y}|^2\right]\rcl. \notag
\end{eqnarray}
We now resort to our exponential bound of Theorem \ref{thmexp} for
$V_{t,x}$, Minkowsky inequality and the relation between $L^{p}$ and
$L^{2}$ moments for Gaussian random variables in order to obtain:
\begin{equation*}
\be\lc |u_{t,x} - u_{s,y}|^{p} \rc \le c_{p} \,  \lc \be\lc \lln
V_{t,x} - V_{s,y}  \rrn^{2} \rc \rc^{p/2}.
\end{equation*}
We now evaluate the right hand side of this inequality.

\smallskip

Let us start by studying a difference of the form $V_{t,x} -
V_{t,y}$, for $t\in(0,T]$ and $x,y\in\R$. The variance of
$V_{t,x}-V_{t,y}$ conditioned by $B$ can be computed as in
(\ref{var}) and we can write
\begin{eqnarray*}
&&\be \lc |V_{t,x}-V_{t,y}|^2\rc\\
&=&2\be_B \lc \int_0^t \int_0^t \gamma(r-s)\laa(B_r-B_s)dr
ds-\int_0^t \int_0^t
\gamma(r-s)\laa(B_r-B_s+x-y)drds\rc\\
&=&2\int_0^t \int_0^t  \int_{\R^d} \gamma(r-s)  \lp 1-\cos\langle
\xi,x-y\rangle\rp  e^{- \frac 12 |r-s||\xi|^2}  \mu(d\xi) drds.
\end{eqnarray*}
Using condition (\ref{t1}) and  the estimate $|1-\cos\langle \xi,
x-y\rangle|\leq |\xi|^{2\alpha}|x-y|^{2\alpha}$,  where $0< \al <
1-\beta$, yields
\[
\be \lc |V_{t,x}-V_{t,y}|^2\rc \leq  C |x-y|^{2\al}  \int_0^t
\int_0^t \int_{\R^d} |r-s|^{-\beta}e^{- \frac 12 |r-s||\xi|^2}  |\xi|^{2\al}
\mu(d\xi)drds\,.
\]
Finally, as in the proof of Proposition  \ref{propdelta},
Hypothesis \ref{hyp:mu-holder} implies
\[
\int_0^T \int_0^T \int_{\R^d} |r-s|^{-\beta}e^{-\frac 12 |r-s||\xi|^2}
|\xi|^{2\al}  \mu(d\xi)drds <\infty,
\]
and thus $\be [ |V_{t,x}-V_{t,y}|^2] \leq  C |x-y|^{2\al}$.

\smallskip
The evaluation of the variance of $V_{t,x} -V_{s,x}$, with $0\le
s<t\le T$, $x\in \R^d$ goes along the same lines. Indeed, we write 
$\be [ |V_{t,x}-V_{s,x}|^2]\le 2(A_1+A_2)$, with
\begin{eqnarray*}
A_{1}&=& \be \lc  \left |\int_s^t
\int_{\R^d}\delta_0 (B_{t-r}^x-y)W(dr,dy)\right |^2  \rc \\
A_{2}&=& 
\be  \lc \left |\int_0^s
\int_{\R^d}\left(\delta_0(B_{t-r}^x-y)-\delta_0(B_{s-r}^x-y)\right)W(dr,dy)\right|^2\rc.
\end{eqnarray*}
For the term $A_1$, computing the variance as in (\ref{var}) and
using condition (\ref{t1}), we obtain
\begin{eqnarray*}
A_1 
&=&\be_B  \lc
\int_0^{t-s}\int_0^{t-s}\gamma(u-v)\laa(B_{u}-B_{v}) \, dudv\rc\\
&\leq & C_{\beta} \int_0^{t-s}\int_0^{t-s}   \int_{\R^d}  |u-v|^{-\beta} e^{- \frac 12|u-v||\xi|^2} \, \mu(d\xi) du dv\\
&\leq& C (t-s) \int_0^{t-s}   \int_{\R^d}  u^{-\beta} e^{-\frac 12u
|\xi|^2} \,\mu(d\xi) du.
\end{eqnarray*}
Then,  Hypothesis \ref{hyp:mu-holder} allows us to write
\[
 \int_{\R^d}   e^{-\frac 12u |\xi|^2}\mu(d\xi)= C_1 +u^{\alpha+\beta -1}  \int_{|\xi| > 1}  |\xi|^{2( \al+\beta-1)}\mu(d\xi)
\]
for any $\alpha <1-\beta$, which leads to the bound $A_1 \le C (t-s)^{1+\al}$.

\smallskip

The term $A_2$ can be handled as follows:  as in \eref{var} we write:
\begin{multline}
A_2 =
\be_B   \Big[ \int_0^s \int_0^s
\gamma(u-v )
\big[\laa(B_{t-u}-B_{t-v }) + \laa(B_{s-u}-B_{s-v }) \\
- 2 \laa(B_{t-u}-B_{s-v }) \big]
\, du dv
\Big],
\end{multline}
and changing to Fourier coordinates, this yields:
\begin{equation}\label{e.A2}
A_2 \le
2 \int_0^s \int_0^s \gamma(u-v)\int_{\R^d}\left|e^{-\frac
12|u-v||\xi|^2}- e^{-\frac 12|t-s-u+v||\xi|^2}\right|\mu(d\xi)du
dv\,. 
\end{equation}
Using the estimate $|e^{-x}-e^{-y}|\leq (e^{-x}+e^{-y})|x-y|^{\al}$,
for any $0< \al <1-\beta$ and $x,y \geq 0$ and condition (\ref{t1}),
we obtain
\begin{eqnarray*}
A_2 \le C  |t-s|^{\al}\int_0^s \int_0^s   \int_{\R^d} |u-v|^{-\beta}
\left(e^{-\frac 12|u-v ||\xi|^2}+e^{-\frac
12|t-s-u+v||\xi|^2}\right)|\xi|^{2\al}\mu(d\xi) dudv\,.
\end{eqnarray*}
Then,  in order to achieve the bound $A_{2}\le |t-s|^{\al}$, it suffices to prove that
\begin{equation*}
\int_0^s \int_0^s |u-v |^{-\beta}\int_{\R^d} e^{-\frac 12|t-s-u+v
||\xi|^2}|\xi|^{2\al } \mu(d\xi)du dv
\end{equation*}
is  uniformly bounded for $0\le s< t \leq T$.  We decompose the
integral with respect to the measure $\mu$ into the regions $\{|\xi|
\le 1\}$ and $\{|\xi| >1\}$. The integral on  $\{|\xi| \le 1\}$ is
clearly bounded because $\mu$ is finite on compact sets. Taking into
account  of the hypothesis \ref{hyp:mu-holder}, the integral over
$\{ |\xi |>1\}$ can be handled using the estimate
\[
\sup_{s,t \in [0,T]} \int_0^s \int_0^s |u-v|^{-\beta}e^{-\frac 12
|t-s-u+v||\xi|^2} dudv \le C |\xi|^{2\beta-2}.
\]
 Putting together our bounds on $A_{1}$ and $A_{2}$, we have been able to prove that $\be [ |V_{t,x}-V_{s,x}|^2] \le |t-s|^{\al}$. Furthermore, gathering our estimates for $V_{t,x}-V_{t,y}$ and $V_{t,x}-V_{s,x}$,  this completes the proof of the     theorem.
\end{proof}

\begin{remark}

The results of Theorem \ref{thmHolder} do not give the optimal H\"older continuity exponents for the process $u$ defined by \eqref{FK}. Another strategy could be implemented, based on the Feynman-Kac representation for the $(2p)$-th moments of $u$. This method is longer than the one presented here, but should lead to some better estimates of the continuity exponents. We stick to the shorter version of Theorem \ref{thmHolder} for sake of conciseness, and also because optimal exponents will be deduced from the pathwise results of Section \ref{sec:pathwise-solution} (in particular Proposition~\ref{prop:identif-FK-pathwise}).

\end{remark}

\subsubsection{Time independent noise}
In the case of time independent noise, the next result provides a
result on the H\"older continuity of the  Feynman-Kac functional
defined in (\ref{FK2}).   In this case we impose the following
additional integrability condition on $\mu$.

\begin{hypothesis}\label{hyp:mu-holder2}
Let $W$ be a spatial Gaussian noise with covariance structure
encoded by  $\laa$. Suppose that   the spectral measure $\mu$
satisfies
\begin{equation*}
\int_{\R^{d}}\frac{\mu(d\xi)}{1+|\xi|^{2(1-\al)}} <\infty
\end{equation*}
for some $\al \in (0,1)$.
\end{hypothesis}

\begin{theorem}\label{thmHolder2}
Let $u$ be the  Feynman-Kac functional  defined in (\ref{FK2}) with $u_0\equiv 1$,
namely:
\begin{equation*}
u_{t,x} = \be_{B}\lc \exp\lp V_{t,x} \rp \rc, \quad\text{where}\quad
V_{t,x} = \int_{0}^{t}  \lp \int_{\R^d} \delta_0(B^x_{r}-y) W(dy)
\rp dr.
\end{equation*}
Then $u$ admits a version which is $(\ga_{1},\ga_{2})$-H\"older
continuous on any compact set of the form $[0,T]\times[-M,M]^{d}$,
with any $\ga_{1} < \frac{1+\al}{2}$, $\ga_{2}< \al$ and $T,M>0$.
\end{theorem}

\begin{proof} The proof is similar to  the proof  
of Theorem \ref{thmHolder}  and we omit the details.
\end{proof}

\subsection{Examples}
Let us  discuss  the validity of Hypothesis \ref{hyp:mu-holder} and
Hypothesis \ref{hyp:mu-holder2} in the  examples  presented in the
introduction. In the case of a time dependent noise we assume that
the time covariance has the form  $\gamma(x)=|x|^{-\beta}$, $0<
\beta < 1$.

\smallskip
For the Riesz kernel $\laa(x)=|x|^{-\eta}$, where
$\mu(d\xi)=C_{\beta}|\xi|^{\eta-d}d\xi$,  we already know that
Hypothesis \ref{hyp:mu} holds if $\eta<2$.  On the other hand,
Hypothesis \ref{hyp:mu2}, which allows us to define the Feynman-Kac
functional in the time dependent case,  is satisfied if $\eta <
2-2\beta$.  For the H\"older continuity,  Hypothesis
\ref{hyp:mu-holder} holds for any   $\al \in (0,
1-\beta-\frac{\eta}{2} )$ and  Hypothesis \ref{hyp:mu-holder2} holds
for any  $\al' \in (0, 1-\frac{\eta}{2})$. Then, by Theorem
\ref{thmHolder} and \ref{thmHolder2}, for any $\al \in (0,
1-\beta-\frac{\eta}{2} )$, $\al' \in (0, 1-\frac{\eta}{2})$,  assuming $u_0\equiv 0$, the
Feynman-Kac functional (\ref{FK})  is H\"older continuous of order
$\al$ in the space variable and of order $\frac{\al}{2}$ in the time
variable, and  the Feynman-Kac functional  (\ref{FK2}) is H\"older
continuous of order $\al$ in the space variable and of order
$\frac{\al'+1}{2}$ in the time variable.

\smallskip

For the  Bessel kernel, we know that Hypothesis \ref{hyp:mu}  is
satisfied when  $\eta> d -2$, and Hypothesis \ref{hyp:mu-holder}
holds when $\eta >d+2\beta-2$.  By Theorem \ref{thmHolder} and
\ref{thmHolder2}, for any $\al \in (0,\min (\frac{\eta
-d}{2}-\beta+1, 1))$ and $\al' \in (0,
\min(\frac{\beta-d}{2}+1,1))$,  assuming $u_0\equiv 0$,  the   Feynman-Kac functional (\ref{FK})
is H\"older continuous of order $\al$ in the space variable and of
order $\frac{\al}{2}$ in the time variable, the Feynman-Kac functional  (\ref{FK2}) is H\"older continuous of order $\al'$ in the space
variable and of order $\frac{\al'+1}{2}$ in the time variable.

\smallskip
 Consider the case of a  fractional noise with
covariance function $\gamma(t)=H(2H-1)  |t|^{2H-2}$ and
$\laa(x)=\prod_{i=1}^d H_i(2H_i-2) |x_i|^{2H_i-2}$.  We know that
Hypothesis \ref{hyp:mu} holds when $\sum_{i=1}^d H_i > d-1$.
Moreover, when $\sum_{i=1}^d H_i > d-2H+1$, Hypothesis
\ref{hyp:mu-holder}  is satisfied. By Theorem \ref{thmHolder} and
\ref{thmHolder2}, for any $\al \in (0, \sum_{i=1}^d H_i+2H-d-1)$ and
$\al' \in (0, \sum_{i=1}^d H_i-d+1)$,  assuming $u_0 \equiv 0$, Feynman-Kac functional (\ref{FK})  is H\"older continuous of order $\al$ in the space
variable and of order $\frac{\al}{2}$ in the time variable, which
recovers the result in \cite{Song Jian}). On the other hand,  Feynman-Kac functional (\ref{FK2}) is H\"older continuous of order
$\al^{\prime}$ in the space variable and of order
$\frac{\al^{\prime}+1}{2}$ in the time variable.

\section{Equation in the Stratonovich sense}\label{sec:pathwise-solution}

 In    this section we consider the following  stochastic heat equation of
 Stratonovich type with the multiplicative Gaussian noise introduced in Section \ref{sec2.1}:
\begin{equation}\label{eqSt1}
\frac{\partial u}{\partial t}=\frac{1}{2}\Delta u + u\frac{\partial
^{d+1} W}{\partial t \partial x_1 \cdots \partial x_d}\,.
\end{equation}
As in the previous sections, the initial condition is a continuous and bounded function $u_0$.
 We will  discuss two notions of
solution.  The  first one is based on the Stratonovich integral,
which is controlled using techniques of Malliavin calculus and a
second one is completely pathwise  and is based on Besov spaces. We
will show that the   Feynman-Kac functional (\ref{FK})  is a
solution in both senses,  and in the pathwise formulation it is the    unique solution to equation \eqref{eqSt1}.

We will also discuss the case of  a time independent multiplicative Gaussian
noise introduced in Section \ref{sec2.2}, that is
\begin{equation}\label{eqSt2}
\frac{\partial u}{\partial t}=\frac{1}{2}\Delta u+u
\frac{\partial^{d}W}{\partial x_1 \cdots \partial
x_d}\,,
\end{equation}
with an initial condition  $u_0\in C_b(\mathbb{R}^d)$. As in the case of a time dependent noise,   we will show that the Feynman-Kac functional (\ref{FK2}) is both a mild Stratonovich solution and a pathwise solution.

\subsection{Stratonovich solution}\label{sec:strato-solution}

Our aim  is to define a notion of solution
 to equation~(\ref{eqSt1}) by means of a Russo-Vallois type approach, which happens to be compatible with Malliavin calculus tools. As usual, we divide our study into time dependent and time independent cases.
 
\subsubsection{Time dependent case}
 Let $W$ be the  time dependent noise  introduced in Section \ref{sec2.1}. In this case, we make use of  the following definition of  Stratonovich integral.
\begin{definition}\label{defSt}
Given a random field $v=\{v_{t,x}; t\geq 0, x\in \R^d\}$ such that
\begin{equation*}
\int_0^T \int_{\R^d}|v_{t,x}| \, dxdt < \infty
\end{equation*}
almost surely for all $T> 0$, the Stratonovich integral $\int_0^T
\int_{\R^d} v_{t,x}W(dt,dx)$ is defined as the following limit in
probability, if it exists:
\begin{equation*}
\lim_{\ep \downarrow 0} \lim_{ \delta \downarrow 0} \int_0^T \int_{\R^d} v_{t,x}
\dot{W}^{\varepsilon,\delta}_{t,x}dx dt\,,
\end{equation*}
where $\dot{W}^{\varepsilon,\delta}_{t,x}$ is the regularization of $W$
defined in (\ref{regW}).
\end{definition}
 With this definition of integral, we have the following notion of
solution for equation (\ref{eqSt1}).

\begin{definition}\label{def4}
A random field $u=\{u_{t,x}; t\geq 0, x \in \R^d\}$ is a mild
solution of equation (\ref{eqSt1})  with initial condition $u_0\in
C_b (\R^d)$ if for any $t\ge0$ and $x\in \R^d$ the following
equation holds
\begin{equation}\label{eq:mild-russo-vallois}
u_{t,x}=p_t u_0(x)+\int_0^t\int_{\mathbb{R}^d}p_{t-s}(x-y)u_{s,y} \,  W(ds,dy)  ,
\end{equation}
 where the last term is a
Stratonovich stochastic integral in the sense of Definition \ref{defSt}.
\end{definition}

 The next result asserts the existence of a solution to equation \eqref{eq:mild-russo-vallois} based on the Feynman-Kac representation. 

\begin{theorem}\label{thmFK1}
Assume Hypothesis \ref{hyp:mu2} holds true.  Then,
the process $u$ defined in (\ref{FK})
is a mild solution of equation (\ref{eqSt1}), in the sense given by Definition \ref{def4}.
\end{theorem}

\begin{proof}
We proceed similarly to Section \ref{sec:FK-moments}. Consider the following approximation to equation~(\ref{eqSt1})
\begin{equation}\label{approxSt}
\frac{\partial u^{\varepsilon,\delta}}{\partial t}=\frac{1}{2}\Delta
u^{\varepsilon,\delta}+u^{\varepsilon,\delta}\dot{W}^{\varepsilon,\delta}_{t,x},
\end{equation}
with initial condition $u_0$, where
$\dot {W}_{t,x}^{\varepsilon,\delta}$ is defined in (\ref{regW}).
From the classical Feynman-Kac formula, we know that
 \begin{equation}\label{eq:FK-approx-u}
u^{\varepsilon,\delta}_{t,x}= \be_B  \lc u_0(B^x_t)  \exp \left(\int_0^t
\dot{W}^{\varepsilon,\delta}(t-s,B_{s}^x)ds\right)\rc\,.
\end{equation}
 Moreover, thanks to Fubini's theorem, we can write
\begin{eqnarray*}
\int_0^t \dot{W}^{\varepsilon,\delta}(t-s,B_s^x)ds &=& \frac 1\delta
\int_0^t \left(\int_{(t-s-\delta)_+}^{t-s} \int_{\R^d}
p_{\varepsilon}(B_s^x-y)W(dr,dy)\right)ds\\
&=& W(A^{\ep.\delta}_{t,x}) =V_{t,x}^{\varepsilon,\delta}\,,
\end{eqnarray*}
where $A_{t,x}^{\varepsilon,\delta}$ is defined in
 (\ref{m3}) and $V_{t,x}^{\varepsilon,\delta}$ is defined in (\ref{rem5}).  Therefore, the process
$u^{\varepsilon,\delta}_{t,x}$ is given by (\ref{FKaprox}), and Proposition \ref{proplim} implies that
(\ref{m9}) holds.

\smallskip

Next we prove that $u$ is a mild solution of equation (\ref{eqSt1})
in the sense of Definition \ref{def4}.  Taking into account of the
definition of the Stratonovich integral, is suffices to show that  
\[
G^{\varepsilon,\delta}: =\int_0^t
\int_{\R^d} p_{t-s}(x-y) \lp u^{\varepsilon,\delta}_{s,y}-u_{s,y}\rp \dot{W}^{\varepsilon,\delta}_{s,y} \,
dyds
\end{equation*}
converges in $L^2(\Omega)$ to zero when first $\delta$ tends to zero and later $\ep$ tends to zero. To this aim, we are going to use  the following notation:
\[
\psi^{\ep,\delta}_{s,y} (r,z)= \frac 1\delta\mathbf{1}_{[(s-\delta)_+,s]} (r) p_\ep(y-z),
\quad\text{and}\quad
\widetilde{u}^{\ep,\delta}_{s,y}=  u^{\varepsilon,\delta}_{s,y}-u_{s,y}.
\]
In particular, notice that $ \dot{W}^{\varepsilon,\delta}_{s,y}=W\lp \psi^{\ep,\delta}_{s,y} \rp$.   Then,
\[
\be\lc   \lp G^{\varepsilon,\delta} \rp^2 \rc =   \iot \iot \int_{\R^{2d}} p_{t-s}(x-y)p_{t-r}(x-z) \be \lc  \widetilde{u}^{\ep,\delta}_{s,y}\widetilde{u}^{\ep,\delta}_{r,z}
W\lp \psi^{\ep,\delta}_{s,y} \rp W\lp \psi^{\ep,\delta}_{r,z}  \rp  \rc  dydz dsdr,
\]
and the expected value above can be analyzed by integration by parts.   Indeed, according to relation \eqref{eq:FK-approx-u}, it is readily checked that $\widetilde{u}^{\ep,\delta}_{s,y}\widetilde{u}^{\ep,\delta}_{r,z} =\be_{B, \widetilde {B}}[Z^{\ep,\delta}_{s,y,r,z}]$, with
\[
Z^{\ep,\delta}_{s,y,r,z}= u_0(B^y_s)  \lc\exp\lp
V^{\ep,\delta,B}_{s,y} \rp - \exp\lp V_{s,y}^B \rp \rc
u_0(\widetilde{B}^z_r)  
\lc\exp\lp
V^{\ep,\delta,\widetilde{B}}_{r,z} \rp - \exp\lp
V_{r,z}^{\widetilde{B}} \rp \rc,
\]
and where $B,\widetilde{B}$ designate two independent $d$-dimensional Brownian motions.
Moreover, a straightforward application of Fubini's theorem yields:
\begin{equation*}
\be \lc  \widetilde{u}^{\ep,\delta}_{s,y}\widetilde{u}^{\ep,\delta}_{r,z} \,
W\lp \psi^{\ep,\delta}_{s,y} \rp W\lp \psi^{\ep,\delta}_{r,z}  \rp  \rc
=
\be_{B, \widetilde {B}}\lcl  \be_{W}\lc Z^{\ep,\delta}_{s,y,r,z} \,  
W\lp \psi^{\ep,\delta}_{s,y} \rp W\lp \psi^{\ep,\delta}_{r,z}  \rp
\rc \rcl.
\end{equation*}
We can now invoke formula  (\ref{k1})  plus some easy computations of Malliavin derivatives in order to get:
\begin{equation}\label{eq:dcp-G-ep-delta}
\be\lc   \lp G^{\varepsilon,\delta} \rp^2 \rc = A_1+A_2,
\end{equation}
where
\[
A_1=\iot \iot \int_{\R^{2d}} p_{t-s}(x-y)p_{t-r}(x-z) \be \lc  \widetilde{u}^{\ep,\delta}_{s,y}\widetilde{u}^{\ep,\delta}_{r,z}  \rc
\langle \psi^{\ep,\delta}_{s,y} , \psi^{\ep,\delta}_{r,z}  \rangle_{\mathcal{H}}    dydz dsdr
\]
and
\[
A_2=\iot \iot \int_{\R^{2d}} p_{t-s}(x-y)p_{t-r}(x-z)\be \lc Z^{\ep,\delta}_{s,y,r,z}
 \Gamma^{\ep,\delta}_{s,y,r,z}  \rc dydz dsdr,
\]
with the notation
\begin{eqnarray*}
\Gamma^{\ep,\delta}_{s,y,r,z} &=& \langle \psi^{\ep,\delta}_{s,y} , A^{\ep,\delta,\widetilde{B}}_{r,z}- \delta_0(\widetilde{B}^z_{r-\cdot}-\cdot)  \rangle_{\mathcal{H}}
 \langle \psi^{\ep,\delta}_{r,z} , A^{\ep,\delta,B}_{s,y}- \delta_0(B^y_{s-\cdot}-\cdot)  \rangle_{\mathcal{H}} \\
 &&+  \langle \psi^{\ep,\delta}_{s,y} , A^{\ep,\delta,B}_{s,y}- \delta_0(B^y_{s-\cdot}-\cdot)  \rangle_{\mathcal{H}}
 \langle \psi^{\ep,\delta}_{r,z} , A^{\ep,\delta,\widetilde{B}}_{r,z}- \delta_0(\widetilde{B}^r_{z-\cdot}-\cdot)  \rangle_{\mathcal{H}} \\
&&+  \langle \psi^{\ep,\delta}_{s,y} , A^{\ep,\delta,B}_{s,y}- \delta_0(B^s_{y-\cdot}-\cdot)  \rangle_{\mathcal{H}}
 \langle \psi^{\ep,\delta}_{r,z} , A^{\ep,\delta,B}_{s,y}- \delta_0(B^y_{s-\cdot}-\cdot)  \rangle_{\mathcal{H}}\\
 &&+  \langle \psi^{\ep,\delta}_{s,y} , A^{\ep,\delta,\widetilde{B}}_{r,z}- \delta_0(\widetilde{B}^r_{z-\cdot}-\cdot)  \rangle_{\mathcal{H}}
 \langle \psi^{\ep,\delta}_{r,z} , A^{\ep,\delta,\widetilde{B}}_{r,z}- \delta_0(\widetilde{B}^r_{z-\cdot}-\cdot)  \rangle_{\mathcal{H}}
 \end{eqnarray*}

According to Proposition \ref{proplim}, we know that
\[
\lim_{\ep \downarrow 0} \lim_{\delta \downarrow 0} \sup_{s\in[0,T], y\in\R^d}
   \be \lc |\widetilde{u}^{\ep,\delta}_{s,y}|^2 \rc =0,
   \]
   and with the same arguments as in Proposition \ref{proplim} we can also show that
\[
\lim_{\ep \downarrow 0} \lim_{\delta \downarrow 0} \sup_{s,r\in[0,T], y,z\in\R^d}
   \be \lc |Z ^{\ep,\delta}_{s,yr,z}|^2 \rc =0.
   \]
   Therefore,  with formula \eqref{eq:dcp-G-ep-delta} in mind, the convergence to zero of $B^{\ep,\delta}$ will follow, provided we show the following quantities are uniformly bounded in $\ep \in (0,1)$ and $\delta \in (0,1)$
   \begin{equation} \label{j1}
\theta_1:=   \iot \iot \int_{\R^{2d}} p_{t-s}(x-y)p_{t-r}(x-z)
 \left|  \langle \psi^{\ep,\delta}_{s,y} , \psi^{\ep,\delta}_{r,z}  \rangle_{\mathcal{H}} \right|   dydz dsdr
   \end{equation}
   and
      \begin{equation} \label{j2}
 \theta_2:=  \iot \iot \int_{\R^{2d}} p_{t-s}(x-y)p_{t-r}(x-z)
 \left\|    \Gamma^{\ep,\delta}_{s,y,r,z}   \right\|_2  dydz dsdr  ,
   \end{equation}
    where $\| \Gamma^{\ep,\delta}_{s,y,r,z} \|_2 $ stands for the norm of $\Gamma^{\ep,\delta}_{s,y,r,z}$ in $L^{2}(\Omega)$. The remainder of the proof is thus just reduced to an estimation of \eqref{j1} and \eqref{j2}. 
   
   \smallskip
     In order to bound $\theta_{1}$, we apply  the estimate  (\ref{j3}) and  the semigroup property of the heat kernel, which yields
   \begin{align*}
    \langle \psi^{\ep,\delta}_{s,y} , \psi^{\ep,\delta}_{r,z}  \rangle_{\mathcal{H}}
   & =    \lp  \frac 1{\delta^2} \int^s_{(s- \delta)_+} \int^r _{(r- \delta)_+} \gamma(u-v)  dudv \rp 
    \int_{\R^{2d}} p_{\ep} (y-z_1) p_\ep(z-z_2) \laa(z_1-z_2) dz_1 dz_2\\
   &\le c_{T,\beta} |r-s|^{-\beta}   \int_{\R^{d}} p_{2\ep} (y-z-w)\laa(w) dw.
    \end{align*}
   Substituting  this estimate into (\ref{j1}), we obtain
   \[
   \theta_1 \le  c_{T,\beta} \iot \iot \int_{\R^{d}} p_{2t-s-r+2\ep }(w) \laa (w) |r-s|^{-\beta} dw
   \le   c'_{T,\beta} \int_0^{2t}  \int_{\R^{d}} p_{2t-s }(w) \laa (w)dw ds   <\infty,
   \]
where we get rid of $\ep$ in Fourier mode, similarly to the proof of \eqref{expint}.
   
   \smallskip
   
   We now turn to the control of  $\theta_2$: we first write,  using the estimate  (\ref{j3}) and  the semigroup property of the heat kernel,
       \begin{eqnarray*}
    \langle \psi^{\ep,\delta}_{s,y} , A^{\ep,\delta}_{r,z} \rangle{_\mathcal{H}} &=&
  \frac 1{\delta^2}  \int_{(s-\delta)_+}^s   \int_{( r-\sigma-\delta)_+} ^{r-\sigma}  \int_0^r   \int_{\R^{2d}} p_\ep(y-z_1)  p_{\ep} (B^z_{\sigma} -z_2)  \gamma(u-v) \\
  &&\hspace{3in}\times \laa(z_1-z_2) dz_1dz_2   d\sigma dv du \\
  &\le&c_{T,\beta}    \int_0^r \int_{\R^d} p_{2\ep} (y-B^z_{r-\sigma}-w) \laa(w)  |s-\sigma|^{-\beta}dw d\sigma
       \end{eqnarray*}
       Invoking again arguments of Fourier analysis, analogous to those in the proof of  (\ref{expint}), we can show that
       \[
       \be \lc   \left| \int_0^r \int_{\R^d} p_{2\ep} (y-B^z_{r-\sigma}-w) \laa(w) |s-\sigma|^{-\beta}dw d\sigma \right|^4 \rc
       \le  \be \lc   \left| \int_0^r   \laa (B_{r-\sigma}) |s-\sigma|^{-\beta}d\sigma \right|^4 \rc ,
       \]
    and
    \[
    \sup_{r,s\in [0,T]}  \be \lc   \left| \int_0^r   \laa (B_{r-\sigma}) |s-\sigma|^{-\beta}d\sigma \right|^4 \rc  <\infty.
    \]
    This implies that  $ \left\|    \Gamma^{\ep,\delta}_{s,y,r,z}   \right\|_2  $, and thus, $\theta_2$, are uniformly bounded.
    The proof of the theorem is complete.
    \end{proof}

\begin{remark}
Consider the case where the space dimension is $1$, $\laa(x)$ is the Dirac delta function $\delta_0(x)$  corresponding to the  white noise, which in our setting means that condition~\eqref{t2} is satisfied with $0 < \beta < \frac{1}{2}$. Then our theorems of Section \ref{sec:eq-strato} cover assumption \eqref{t1}, with  $0 < \beta < \frac{1}{2}$ too,  if we interpret the composition $\laa(B_r-B_s)$ as a generalized Wiener functional. 
  Notice that in the case of the fractional Brownian motion with Hurst parameter $H$ (that is $\gamma(x) =c_{H} |x|^{2H-2}$) the condition  $0 < \beta < \frac{1}{2}$ means that
$H >\frac 34$. In this case it is already known that the process defined by (\ref{FK}) is still a solution to equation (\ref{eqSt1}) (see \cite{Song Jian}).
\end{remark}

\subsubsection{Time independent case}
Let $W$ be the time independent noise introduced in Section~\ref{sec2.2}. We claim that as in the time independent case, the Feynman-Kac functional given by~(\ref{FK2}) is a 
  mild solution to equation (\ref{eqSt2}) in the Stratonovich  sense.
  
  \smallskip
  
The Stratonovich integral with respect to the noise $W$ is defined as the limit of the integrals with respect to  regularization of the noise.
\begin{definition}\label{def5}
Given a random field $v=\{v_{x}; x\in \R^d\}$ such that
$ \int_{\R^d}|v_{x}|dx < \infty$
 almost surely, the Stratonovich integral $
\int_{\R^d} v_{x}W(dx)$ is defined as the following limit in
probability, if it exists:
\begin{equation*}
\lim_{\varepsilon\downarrow 0} \int_{\R^d} v_{x}
\dot{W}^{\varepsilon}_{x}dx \,,
\end{equation*}
where $\dot{W}^{\varepsilon}_{x}=  \int_{\mathbb{R}^d}  p_\varepsilon(x-y) W(dy) $.
\end{definition}
 With this definition of integral, we have the following notion of
solution for equation (\ref{eqSt2}).

\begin{definition}\label{def6}
A random field $u=\{u_{t,x}; t\geq 0, x \in \R^d\}$ is a mild
solution of equation (\ref{eqSt2}) if  we
have
\begin{eqnarray*}
 u_{t,x}= p_t u_0(x)+\ \int_0^t  \left( \int_{\R^d} p_{t-s}(x-y) u_{s,y}W(dy) \right)ds
\end{eqnarray*}
almost surely for all $t \geq 0$, where the last term is a
Stratonovich stochastic integral in the sense of Definition \ref{def5}.
\end{definition}

The next result is the existence of a solution based on the Feynman-Kac representation.

\begin{theorem}\label{thmFK2n}
Suppose that $\mu$ satisfies (\ref{mu1}).  Then,
the process $u_{t,x}$ given by  (\ref{FK2})
is a mild solution of equation (\ref{eqSt2}).
\end{theorem}

The proof of this theorem is similar to that of Theorem \ref{thmFK1}, and it is omitted.

\subsection{Existence and uniqueness of a pathwise solution}
In this section we define and solve equations \eqref{eqSt1} and  \eqref{eqSt2}   in a pathwise manner in $\R^{d}$, when the noise $W$ satisfies  some additional hypotheses.  Contrarily to the Stratonovich technology invoked at Section \ref{sec:strato-solution}, the pathwise method yields uniqueness theorems, which will be used in order to identify Feynman-Kac and pathwise solutions. At a technical level, our  results will be achieved in the framework of weighted Besov spaces, that  we proceed to recall now.

\subsubsection{Besov spaces}\label{sec:besov-spaces}

The definition of Besov spaces is based on Littlewood-Paley theory, which relies on decompositions of functions into spectrally localized blocks. We thus first introduce the following basic definitions.
\begin{definition}
We call \textit{annulus} any set of the form $C= \{ x \in
\mathbb{R}^d : a \leqslant | x | \leqslant b \}$ for some $0 < a < b$. A
{\textit{ball}} is a set of the form $B= \{ x \in \mathbb{R}^d : | x
| \leqslant b \}$.
\end{definition}

\smallskip
The localizing functions for the Fourier domain alluded to above are defined as follows.

\begin{notation}\label{not:partition-unity} In the remaining part of this section,  we
shall use    $\chi , \vp  $ to denote two smooth nonnegative radial
functions with compact support such that:
\begin{enumerate}
  \item The support of $\chi$ is contained in a ball and the support of $\vp$
  is contained in an annulus $C$ with $a=3/4$ and $b=8/3$;

  \item We have $\chi ( \xi) + \sum_{j \geqslant 0} \vp ( 2^{- j} \xi) = 1$ for all
  $\xi \in \mathbb{R}^d$;

  \item It holds that $\text{supp} ( \chi) \cap \text{supp} ( \vp ( 2^{- i} \cdot)) =
  \varnothing$ for $i \geqslant 1$ and if $| i - j | > 1$, then $\text{supp} (
  \vp ( 2^{- i} \cdot)) \cap \text{supp} ( \vp ( 2^{- j} \cdot)) =
  \varnothing$.
\end{enumerate}
In the sequel, we set $\vp_{j}(\xi):=\vp(2^{- j} \xi)$.
\end{notation}
For the existence of $\chi$ and $\vp$ see {\cite[Proposition 2.10]{BCD-bk}}. With this notation in mind, the Littlewood-Paley blocks are now defined as follows.
\begin{definition}
Let $u\in\cs'(\R^{d})$. We set
\[ \Delta_{- 1} u =\mathcal{F}^{- 1} ( \chi \, \mathcal{F}u), \quad
   \text{and for } j \geqslant 0 \quad
   \Delta_j u=\mathcal{F}^{- 1} ( \vp_{j}  \, \mathcal{F}u)
   . \]
We also use the notation $S_{k}u=\sum_{j=-1}^{k-1}\Delta_j u$, valid for all $k\ge 0$.
\end{definition}
Observe that one can also write $\Delta_{- 1} u =  \widetilde {K} \ast u$ and $\Delta_j u =
K_j \ast u$ for $j \geqslant 0$ , where $ \widetilde {K} =\mathcal{F}^{- 1} \chi$ and  $K_j =
2^{jd}\mathcal{F}^{- 1} \vp (2^j \cdot)$. In particular the $\Delta_j u$ are
smooth functions for all $u\in\cs'(\R^{d})$.

\smallskip

 In order to handle equations whose space parameter lies in   an unbounded domain like $\R^{d}$,
we shall use  spaces of weighted H\"older type functions for
polynomial or exponential weights, where the weights satisfy some
smoothness conditions. In this way we define the following class of weights.

\begin{definition}
We denote by $\cw$ the class of weights $w\in \cac_{b}^{\infty}(\R^{d};\R_{+})$  consisting of:
\begin{itemize}
\item
The weights $\rho_{\ka}$ obtained as functions of the form $c \,
(1+|x|^{\ka})^{-1}$, with $\ka \ge 1$, smoothed at 0.
\item
The weights $e_{\lambda}$ obtained as functions of the form $c \, e^{-\lambda |x|}$, with $\lambda>0$, smoothed at 0.
\item
Products of these functions.
\end{itemize}
\end{definition}

Notice that more general classes of weights are introduced in \cite{Tr-bk}. We have also tried to stick to the notation given in \cite{HPP}, from which our developments are inspired.

\smallskip

Weighted Besov spaces are sets of functions characterized by their Littlewood-Paley block decomposition. Specifically, their definition is as follows.
\begin{definition}
Let $w\in\cw$ and $\ka\in\R$. We set
\begin{equation}\label{eq:def-Besov-gamma}
\cb_{w}^{\ka}(\R^{d}) = \lcl f\in\cs^{\prime}(\R^d) ; \,
\|f\|_{w,\ka}:=\sup_ {j\ge -1 } 2^{j\ka} \| w \, \Delta_{j}
f\|_{L^{\infty}} < \infty \rcl\,.
\end{equation}
We  call this space a weighted Besov-H\"older space. When $w=1$, we
just denote the space by $\cb^{\ka}(\R^{d})$, and it corresponds to
the usual Besov space $\cb^{\ka}_{\infty,\infty}(\R^{d})$.
\end{definition}

Notice that we follow here the terminology of \cite{Tr-bk}. The
weighted Besov-H\"older spaces are well understood objects, and let
us recall some basic facts about them.

\begin{proposition}\label{prop:basic-weighted-besov}
Let $w,w_{1},w_{2}\in\cw$, $\ka\in\R$ and $f\in\cb_{w}^{\ka}(\R^{d})$. Then the following holds true:

\smallskip

\noindent \emph{(i)} There exist some positive constants
$c_{\ka,w}^{1},c_{\ka,w}^{2}$ such that $c_{\ka,w}^{1} \|f
w\|_{\ka}\le \|f\|_{w,\ka} \le c_{\ka,w}^{2} \|f w\|_{\ka}$.

\smallskip

\noindent \emph{(ii)} For $\ka\in  (0,1)$, we have $f\in
\cb_{w}^{\ka}(\R^{d})$ iff $f w$ is a $\ka$-H\"older  function.

\smallskip

\noindent
\emph{(iii)}
If $w_{1}<w_{2}$ we have $\|f\|_{w_{1},\ka}\le \|f\|_{w_{2},\ka}$.
\end{proposition}

\begin{proof}
Item (i) is borrowed from \cite[Chapter 6]{Tr-bk}.
The fact that $\cb^{\ka}(\R^{d})$ coincides with the space of H\"older continuous functions $\cac^{\ka}(\R^{d})$ for $\ka\in [0,1]$ is shown in \cite[Theorem 2.36]{BCD-bk}, and it yields item (ii) thanks to (i). Finally, item (iii) is also taken from  \cite[Chapter 6]{Tr-bk}.
\end{proof}

Let us now state a result about products of distributions which
turns out to be useful for our existence and uniqueness result.

\begin{proposition}
Let $w_{1},w_{2}$ be two  weight functions in $\cw$, and
$\ka_{1},\ka_{2}\in\R$ such that $\ka_{2}<0<\ka_{1}$ and
$\ka_{1}>|\ka_{2}|$ and let $w=w_{1}w_{2}$.   Then
\begin{equation}\label{eq:cty-product-besov}
(f_{1},f_{2}) \in \cb_{w_{1}}^{\ka_{1}} \times \cb_{w_{2}}^{\ka_{2}}
\longmapsto f_{1} \, f_{2} \in \cb_{w}^{\ka_{2}} \quad\text{is
continuous.}
\end{equation}
Furthermore, the following bound holds true:
\begin{equation}\label{eq:bnd-product-besov}
\left\|  f_{1} \, f_{2} \right\|_{\cb_{w}^{\ka_{2}}} \le \left\|
f_{1} \right\|_{\cb_{w_{1}}^{\ka_{1}}} \, \left\|  f_{2}
\right\|_{\cb_{w_{2}}^{\ka_{2}}} .
\end{equation}
\end{proposition}

Finally we label the action of the heat semigroup on functions in weighted Besov spaces.

\begin{proposition}\label{prop:heat-semigroup-besov}
Let $w\in\cw$, $\ka\in\R$ and $f\in\cb_{w}^{\ka}(\R^{d})$. Then for
all $t\in[0,\tau]$, $ \gamma>0$ and $\hat{\ka}>\ka$ we have
\begin{equation*}
\|p_{t}f\|_{w,\hat{\ka}} \le c_{\tau,w,\kappa ,\hat{\kappa}} t^{-\frac{\hat{\ka}-\ka}{2}} \|f\|_{w,\ka},
\quad\text{and}\quad
\|[\id- p_{t}] f\|_{w,\ka-2\ga} \le c_{\tau,w,\gamma}  \, t^{\ga} \|f\|_{w,\ka}.
\end{equation*}
\end{proposition}

\medskip
\subsubsection{Notion of solution}
In order to give a pathwise definition of solution for equation
\eqref{eqSt1}, we will replace the noise $W$ by a  nonrandom H\"older continuous function in time with values in a
Besov space of distributions, denoted by ${\mathscr W}$.  We will show later (see Proposition \ref{prop:hyp-gauss-pathwise})  that under Hypothesis \ref{hyp:mu-holder}, almost surely  the mapping $t\rightarrow W(\mathbf{1}_{[0,t]} \varphi)$, $\varphi \in \mathcal{D}$,  is H\"older continuous with values in this Besov space.  We thus label a notation for this kind of
space.

\begin{notation}
Let $\theta\in(0,1)$, $\ka\in\R$ and $w\in\cw$. The space of
$\te$-H\"older continuous functions from $[0,T]$ to a weighted
Sobolev space $\cb_{ w}^{\ka} $ is denoted by
$\cac_{T,w}^{\te,\ka}$. Otherwise stated, we have
$\cac_{T,w}^{\te,\ka}=\cac^{\te}([0,T]; \cb_{ w}^{\ka} )$. In
order to alleviate notations, we shall write $\cac_{w}^{\te,\ka}$
only when the value of $T$ is non ambiguous.
\end{notation}

Now we introduce the pathwise type assumption  that we shall make on  the multiplicative input distribution ${\mathscr W}$.

\begin{hypothesis}\label{hyp:W-pathwise-space-time}
We assume that there exist 
 two constants $\te, \ka\in(0,1)$   satisfying 
$\frac{1+\ka}{2} < \te <1$, 
such that 
$\W \in\cac_{T,\rho_{\si}}^{\te,-\ka}$, for any $\si >0$ arbitrarily small.
\end{hypothesis}

We also label some more notation for further use:

\begin{notation}
 For a function $f:[0,T]\to\cb$, where $\cb$ stands for a generic Banach space, we set $\delta f_{st}= f_{t}-f_{s}$ for $0\le s \le t \le T$. Notice that $\delta$ has also been used for Skorohod integrals, but this should not lead to ambiguities since Skorohod integrals won't be used in this section.  
 \end{notation}

With these preliminaries in hand, we shall combine the following ingredients in order to solve equation \eqref{eqSt1}: 
\smallskip

\noindent $\bullet$ Like the input  $\W$, the solution $u$ will live
in a space of H\"older functions in time, with values in a weighted
Sobolev space of the form $\cb_{ e_{\lambda}}^{\ka_{u}}$. This allows
the use of estimates of Young integration type in order  to define
integrals involving increments of the form $u \, d\W$.

\smallskip

\noindent $\bullet$ We have to take into account of the fact that,
when one multiplies the function $u_{s}\in\cb_{ e_{\la}}^{\ka_{u}}$
by the distribution $\delta \W_{st}\in\cb_{\rho_{\si}}^{-\ka}$, the
resulting distribution $u_{s}\delta \W_{st}$ lies (provided
$\ka_{u}>\ka$) into the space $\cb_{e_{\lambda}\rho_{\si}}^{-\ka}$.
This will force us to assume in fact
$u_{s}\in\cb_{w_{s}}^{\ka_{u}}$, where the weight $w_{s}\in\cw$
decreases with $s$.

\smallskip

\noindent Let us turn now to the technical part of our task. We first  fix positive constants $\lambda, \sigma$ and define
a weight $w_{t}=e_{\lambda+\si t}$. We shall seek the solution to
equation \eqref{eqSt1} in  the following space:
\begin{multline}\label{eq:def-D-gamma-kappa}
\cd_{\lambda, \si} ^{\te_{u},\ka_{u}} = \Big\{ f\in
\cac([0,T]\times\R^{d}) ; \, \|f_{t}\|_{\cb_{ w_{t}}^{\ka_{u}} }
\le c_{T,f} \\
\text{ and } \|f_{t} -  f_{s}\|_{\cb_{
w_{t}}^{\ka_{u}} } \le c_{f} \, |t-s|^{\te_{u}} \ \ \forall 0\le
s<t\le  T \Big\}.
\end{multline}
We introduce the H\"older  norm in this space by
\begin{equation}  \label{norm1}
\| f\| _{\cd _{\lambda, \si}^{\te_{u},\ka_{u}}} 
=\sup_{0\le s <t \le T} \frac {\|f_{t} -  f_{s}\|_{\cb_{
w_{t}}^{\ka_{u}} }}  {  |t-s|^{\te_{u}}}.
\end{equation}

We now introduce a  pathwise mild formulation  for equation \eqref{eqSt1}   in the spaces $\cd_{\lambda, \si} ^{\te_{u},\ka_{u}}$.

\begin{definition}\label{def:heat-mild-space-time}
Suppose that $\W$ satisfies Hypothesis \ref{hyp:W-pathwise-space-time}.
Let $u\in\cd_{\la, \si} ^{\te_{u},\ka_{u}}$ for  fixed $\lambda, \si>0$, $\te_{u}+\te >1$ and $\ka_{u}\in (\ka  , 1)$.  Consider an initial condition $u_{0}\in\cb_{e_{\lambda}}^{\ka_{u}}$.  We say that $u$ is a mild solution to equation
\begin{equation}\label{eqBe1}
\frac{\partial u}{\partial t}=\frac{1}{2}\Delta u + u\frac{\partial \W}{
\partial t}  
\end{equation}
 with initial condition $u_0$ if it satisfies the following integral equation 
\begin{equation}\label{eq:mild-young-she-space-time}
u_{t} = p_{t}u_{0} + \int_{0}^{t} p_{t-s} \lp u_{s} \, \W(ds) \rp ,
\end{equation}
where the product $u \, \W$ is interpreted in the distributional sense  of \eqref{eq:cty-product-besov}  and the time integral is understood in the Young sense.
\end{definition}

\begin{remark}\label{rmk:def-J-u}
Let us specify what we mean by $J_{t}^{u} :=\int_{0}^{t} p_{t-s} \lp u_{s} \, \W(ds) \rp$ under the conditions of Definition \ref{def:heat-mild-space-time}. First, we should understand $J_{t}^{u}$ as
\begin{equation*}
J_{t}^{u} = \lim_{\ep\to 0} J_{t}^{u,\ep},
\quad\text{where}\quad
J_{t}^{u,\ep} = \int_{0}^{t-\ep} p_{t-s} \lp u_{s} \, \W(ds) \rp .
\end{equation*}
The integration on $[0,t-\ep]$ avoids any singularity of $p_{t-s}$
as an operator from $\cb^{-\ka}$ to $\cb^{\ka_{u}}$, so that
$J_{t}^{u,\ep}$ is  defined as a Young integral. This integral  is
in particular limit of Riemann sums along dyadic partitions of
$[0,t]$:
\begin{equation*}
J_{t}^{u,\ep} = \lim_{n\to\infty} \sum_{j=0}^{2^{n}-1} p_{t-t_{j}^{n}} \lp u_{t_{j}^{n}} \, \delta \W_{t_{j}^{n}t_{j+1}^{n}} \rp \1_{[0,t-\ep]}(t_{j+1}^{n}),
\quad\text{where}\quad
t_{j}^{n} = \frac{j t}{2^{n}}.
\end{equation*}
We then assume that one can combine the limiting procedures in $n$ and $\ep$ (the justification of this step is left to the patient reader), and finally we define:
\begin{equation}\label{eq:def-J-u}
J_{t}^{u} = \lim_{n\to  \infty} J_{t}^{u,n},
\quad\text{where}\quad
J_{t}^{u,n} = \sum_{j=0}^{2^{n}-1} p_{t-t_{j}^{n}} \lp u_{t_{j}^{n}} \, \delta \W_{t_{j}^{n}t_{j+1}^{n}} \rp .
\end{equation}
Here again, recall that the product $u_{t_{j}^{n}} \, \delta \W_{t_{j}^{n}t_{j+1}^{n}}$ is interpreted according to \eqref{eq:cty-product-besov}.
This will be our way to understand equation \eqref{eq:mild-young-she-space-time}.
\end{remark}

We can now turn to the resolution of the equation in this context.

\subsubsection{Resolution of the equation}
Our existence and uniqueness result takes the following form:

\begin{theorem}\label{thm:pathwise-solt-space-time}
Let $\W$ be a H\"older continuous distribution valued function satisfying Hypothesis \ref{hyp:W-pathwise-space-time}, and let $\lambda,\sigma$ be two strictly positive constants. Consider an initial condition $u_{0}\in\cb_{e_{\la}}^{\ka_{u}}$. Then:

\smallskip

\noindent
\emph{(a)} There exist $\te_{u},\ka_{u}$ satisfying $\te_{u}+\te >1$ and $\ka_{u}\in (\ka , 1)$, such that equation~\eqref{eq:mild-young-she-space-time} admits a unique solution in $\cd_{\lambda, \si}^{\te_{u},\ka_{u}}$.

\smallskip

\noindent
\emph{(b)} The application $(u_{0},\W)\mapsto u$ is continuous from $\cb_{e_{\la}}^{\ka_{u}} \times \cac_{T,\rho_{\si}}^{\theta,-\ka}$ to  $\cd_{\lambda, \si}^{\te_{u},\ka_{u}}$.
\end{theorem}

\begin{proof}
We divide this proof into several steps.

\smallskip

\noindent \textit{Step 1: Definition of a contracting map.} 
We fix a time interval $[0,\tau]$, where $\tau \le T$, and along the proof we    denote  by  $\cd_{\lambda, \si}^{\te_{u},\ka_{u}}$ and  $\| \cdot \|_{\cd_{\lambda, \si} ^{\te_{u},\ka_{u}}}$ the space and the H\"older norm defined in (\ref{eq:def-D-gamma-kappa}) and (\ref{norm1}), respectively, but restricted to the interval $[0,\tau]$.  

We
consider a map $\gga$ defined on $\cd_{\lambda, \si}^{\te_{u},\ka_{u}}$
by $\gga(u)=v$, where $v$ is the function defined by $ v:=p_{t}u_{0}+
J_{t}^{u}$ as in Remark \ref{rmk:def-J-u}. The proof of our result
relies on two  steps: (i) Show that $\gga$ defines a map from
$\cd_{\lambda, \si} ^{\te_{u},\ka_{u}}$ to
$\cd_{\lambda, \si}^{\te_{u},\ka_{u}}$, independently of the length of
the interval $[0,\tau]$. (ii) Check that $\gga$ is in fact a
contraction if $\tau$ is made small enough. The two steps hinge on
the same type of computations, so that we shall admit point (i) and
focus on point (ii) for sake of conciseness.

\smallskip

In order to prove that $\gga$ is a contraction, consider
$u^{1},u^{2}\in\cd_{\lambda, \si} ^{\te_{u},\ka_{u}}$, and for $j=1,2$ set
$v^{j}=\gga(u^{j})$. For notational sake, we also set
$u^{12}=u^{1}-u^{2}$ and $v^{12}=v^{1}-v^{2}$. Consistently with
equation \eqref{eq:mild-young-she-space-time}, $v^{12}$ satisfies
the relation 
\begin{equation*}
v^{12}_{t}  =
\int_{0}^{t} p_{t-r}\lp u^{12}_{r} \, \W(dr) \rp\,.
\end{equation*}
Notice that the function $v^{12}$ is in fact defined by relation
\eqref{eq:def-J-u}. We have admitted point~(i) above, which means in
particular that we assume that the Riemann sums in
\eqref{eq:def-J-u} are converging whenever
$u^{12}\in\cd_{\lambda, \si}^{\te_{u},\ka_{u}}$. We now wish to prove
that, provided $\tau$ is small enough, we have
$\|v^{12}\|_{\cd_{\lambda, \si} ^{\te_{u},\ka_{u}}} \le \frac12
\|u^{12}\|_{\cd_{\lambda, \si}^{\te_{u},\ka_{u}}}$.

\smallskip

\noindent
\textit{Step 2: Study of differences.}
Let $0\le s < t \le \tau$. We decompose $v^{12}_{t} - v^{12}_{s}$ as $L_{st}^{1}+L_{st}^{2}$, with
\begin{equation*}
L_{st}^{1} = \int_{0}^{s} \lc p_{t-s} -\id \rc p_{s-v}\lp u^{12}_{v} \, \W(dv) \rp ,
\quad\text{and}\quad
L_{st}^{2} = \int_{s}^{t} p_{t-v}\lp u^{12}_{v} \, \W(dv) \rp,
\end{equation*}
where the Young integrals with respect to $\W(dv)$ are understood as limit of Riemann sums as in \eqref{eq:def-J-u}. We now proceed to the analysis of $L_{st}^{1}$ and $L_{st}^{2}$.

\smallskip

As in relation \eqref{eq:def-J-u}, we write $L_{st}^{1}=\lim_{n\to\infty}L_{st}^{1,n}$, where we consider points  $s_{k}^{n}=2^{-n}k s$ in the dyadic partition of $[0,s]$ and where we set
\begin{equation}\label{eq:def-L-1n}
L_{st}^{1,n}
=
\sum_{j=0}^{2^{n}-1}
\lc p_{t-s} -\id \rc p_{s-s_{j}^{n}} \lp u^{12}_{s_{j}^{n}} \, \delta \W_{s_{j}^{n}s_{j+1}^{n}} \rp.
\end{equation}
In order to estimate $L_{st}^{1}$, let us thus first analyze the quantity $L^{1,n+1}_{st}-L^{1,n}_{st}$. Indeed, it is readily checked that $L^{1,n+1}_{st}-L^{1,n}_{st}
= \sum_{j=0}^{2^{n}-1}  L^{1,n,j}_{st}$, where $L^{1,n,j}_{st}$ is defined by:
\begin{equation*}
L^{1,n,j}_{st} =
\lc p_{t-s} -\id \rc p_{s-s_{2j+1}^{n+1}} \lp u^{12}_{s_{2j+1}^{n+1}} \, \delta \W_{s_{2j+1}^{n+1}s_{2j+2}^{n+1}} \rp
-\lc p_{t-s} -\id \rc p_{s-s_{2j}^{n+1}} \lp u^{12}_{s_{2j}^{n+1}} \, \delta \W_{s_{2j+1}^{n+1}s_{2j+2}^{n+1}} \rp.
\end{equation*}
We now drop the index $n+1$ in the next computations for sake of readability, and write $L^{1,n,j}_{st} = L^{11,n,j}_{st} -L^{12,n,j}_{st}$ with
\begin{eqnarray*}
L^{11,n,j}_{st}&=&
\lc p_{t-s} -\id \rc p_{s-s_{2j+1}} \lp \delta u^{12}_{s_{2j}s_{2j+1}} \, \delta \W_{s_{2j+1}s_{2j+2}} \rp
:=  \lc p_{t-s} -\id \rc \hat{L}^{11,n,j}_{st}  \\
L^{12,n,j}_{st}&=&
\lc p_{t-s} -\id \rc \lc p_{s_{2j+1}-s_{2j}} -\id  \rc 
 p_{s-s_{2j+1}} 
\lp u^{12}_{s_{2j}} \, \delta \W_{s_{2j+1}s_{2j+2}} \rp
:=   \lc p_{t-s} -\id \rc \hat{L}^{12,n,j}_{st}.
\end{eqnarray*}
We treat again the two terms $L^{11,n,j}_{st},L^{12,n,j}_{st}$ separately.

\smallskip

Owing to Proposition \ref{prop:heat-semigroup-besov}, we have
\begin{equation*}
\| L^{11,n,j}_{st} \|_{\cb_{ w_{t}}^{\ka_{u}} } \le c \,
(t-s)^{\te_{u}} \, \| \hat{L}^{11,n,j}_{st} \|_{\cb_{
w_{t}}^{\ka_{u} + 2 \te _{u}} } \le \frac{c \, (t-s)^{\te_{u}} \, \|
\delta u^{12}_{s_{2j}s_{2j+1}} \, \delta \W_{s_{2j+1}s_{2j+2}}
\|_{\cb_{ w_{t}}^{-\ka}} }
{(s-s_{2j+1})^{\te_{u}+\frac{\ka_{u}+\ka}{2}}}
\end{equation*}
Let us now recall the following elementary bound:
\begin{equation}\label{eq:bnd-polynomial-exponential}
\vp_{\al,\ka}(x) :=x^{\al} \, e^{-\ka x}
\quad\Longrightarrow\quad
0 \le \vp_{\al,\ka}(x) \le \frac{c_{\al}}{\ka^{\al}},
\quad\text{for}\quad
x,\al,\ka\in\R_{+}.
\end{equation}
This entails $w_{t} \le c_{\si} (t-t_{2j+1})^{-\si} w_{t_{2j+1}}
\rho_{\si}$, and according to \eqref{eq:bnd-product-besov} we obtain
\begin{eqnarray*}
\| L^{11,n,j}_{st} \|_{\cb_{ w_{t}}^{\ka_{u}} } &\le&  \frac{c_{\si}
\, (t-s)^{\te_{u}} \| \delta u^{12}_{s_{2j}s_{2j+1}} \, \delta
\W_{s_{2j+1}s_{2j+2}} \|_{\cb_{ w_{s_{2j+1} \rho_{\si} }}^{-\ka }}}
{(s-s_{2j+1})^{\te_{u}+\frac{\ka_{u}+\ka}{2}+\si}} \\
&\le&  \frac{c_{\si} \, (t-s)^{\te_{u}} \| \delta
u^{12}_{s_{2j}s_{2j+1}} \|_{\cb_{w_{s_{2j+1}}}^{\ka_{u}}} \,
\|\delta \W_{s_{2j+1}s_{2j+2}} \|_{\cb_{\rho_{\si}}^{-\ka}}}
{(s-s_{2j+1})^{\te_{u}+\frac{\ka_{u}+\ka}{2}+\si}} \\
&\le&  \frac{c_{\si} \, (t-s)^{\te_{u}}
\| u^{12} \|_{\cd_{\la,\si}^{\te_{u},\ka_{u}}} \, \|\W \|_{\cac_{\rho_{\si}}^{\te,-\ka}}}
{(s-s_{2j+1})^{\te_{u}+\frac{\ka_{u}+\ka}{2}+\si}} \, \lp \frac{s}{2^{n}} \rp^{\te_{u}+\te}.
\end{eqnarray*}
As far as $L^{12,n,j}_{st}$ is concerned, we have as above:
\begin{equation}\label{eq:bnd-L12-hat-L12}
\| L^{12,n,j}_{st} \|_{\cb_{ w_{t}}^{\ka_{u}} } \le c \,
(t-s)^{\te_{u}} \, \| \hat{L}^{12,n,j}_{st} \|_{ \cb_{
w_{t}}^{\ka_{u} + 2\te_{u}} } .
\end{equation}
We now take an arbitrarily small and strictly positive constant $\ep$ and write:
\begin{eqnarray*}
\| \hat{L}^{12,n,j}_{st} \|_{\cb_{ w_{t}}^{\ka_{u}+2\te_{u}}} &\le&
(s_{2j+1}-s_{2j})^{1-\te+\ep} \left\| p_{s-s_{2j+1}}\lp
u^{12}_{s_{2j}} \,
\delta \W_{s_{2j+1}s_{2j+2}} \rp\right\|_{\cb_{w_{t}}^{\ka_{u}+2\te_{u}+2(1-\te+\ep)} } \\
&\le&
\frac{(s_{2j+1}-s_{2j})^{1-\te+\ep}}{(s-s_{2j+1})^{1+\te_{u}-\te+\ep+\frac{\ka_{u}+\ka}{2}}}
\, \| u^{12}_{s_{2j}} \, \delta \W_{s_{2j+1}s_{2j+2}} \|_{\cb_{
w_{t}}^{-\ka}} \, ,
\end{eqnarray*}
and thus relation \eqref{eq:bnd-L12-hat-L12} entails:
\begin{equation*}
\| L^{12,n,j}_{st} \|_{\cb_{ w_{t}}^{\ka_{u}} } \le \frac{c_{\si} \,
(t-s)^{\te_{u}} \, \| u^{12} \|_{\cd_{\la,\si}^{\te_{u},\ka_{u}}} \,
\|\W \|_{\cac_{\rho_{\si}}^{\te,-\ka}}}
{(s-s_{2j+1})^{1+\te_{u}-\te+\ep+\frac{\ka_{u}+\ka}{2}+\si}}
\, \lp \frac{s}{2^{n}} \rp^{1+\ep}.
\end{equation*}
Putting together the last two estimates on $L^{11,n,j}_{st}$ and
$L^{12,n,j}_{st}$ and choosing $\te_{u}=1-\te+\ep$, we thus end
up with:
\begin{equation}\label{eq:bnd-J-fnj}
\| L^{1,n,j}_{st} \|_{\cb_{ w_{t}}^{\ka_{u}} } \le \frac{c_{\si} \,
(t-s)^{\te_{u}} \| u^{12} \|_{\cd_{\la,\si}^{\te_{u},\ka_{u}}} \,
\|\W \|_{\cac_{\rho_{\si}}^{\te,-\ka}}}
{(s-s_{2j+1})^{2-2\te+2\ep+\frac{\ka_{u}+\ka}{2}+\si}} \,
\lp \frac{s}{2^{n}} \rp^{1+\ep}.
\end{equation}

Let us now discuss exponent values: for the convergence of $L^{1,n}_{st}$ we need the condition
\begin{equation*}
2-2\te+2\ep+\frac{\ka_{u}+\ka}{2}+\si < 1
\end{equation*}
to be fulfilled. If we choose $\ka_{u}=\ka+2\ep$, we can recast this condition into $\te>\frac{1+\ka}{2} + \frac{3\ep+\si}{2}$. Since $\ep,\si$ are chosen to be arbitrarily small, we can satisfy this constraint as soon as $\te>\frac{1+\ka}{2}$, which was part of our Hypothesis \ref{hyp:W-pathwise-space-time}. For the remainder of the discussion, we thus assume that
\begin{equation*}
2-2\te+2\ep+\frac{\ka_{u}+\ka}{2}+\si = 1 - \eta,
\quad\text{with}\quad
\eta >0.
\end{equation*}

\smallskip

\noindent
\textit{Step 3: Bound on $L^{1}_{st}$.}
We express $\lim_{n\to\infty}L^{1,n}_{st}$ as $L^{1,0}_{st}+\sum_{n=0}^{\infty}(L^{1,n+1}_{st}-L^{1,n}_{st})$. Now
\begin{equation*}
\sum_{n=0}^{\infty} \| L^{1,n+1}_{st}-L^{1,n}_{st}\|_{\cb_{
w_{t}}^{\ka_{u}} } \le \sum_{n=0}^{\infty} \sum_{j=0}^{2^{n}-1} \|
L^{1,n,j}_{st} \|_{\cb_{ w_{t}}^{\ka_{u}} },
\end{equation*}
and plugging our estimate \eqref{eq:bnd-J-fnj}, we get that
$\sum_{n=0}^{\infty} \| L^{1,n+1}_{st}-L^{1,n}_{st}\|_{\cb_{
w_{t}}^{\ka_{u}} }$ is bounded by:
\begin{equation*}
c_{\si} \,
\| u^{12} \|_{\cd_{\la,\si}^{\te_{u},\ka_{u}}} \, \|\W \|_{\cac_{\rho_{\si}}^{\te,-\ka}}
\, (t-s)^{\te_{u}}
\sum_{n=0}^{\infty} \lp \frac{s}{2^{n}} \rp^{\ep}
\left(
\frac{s}{2^{n}} \,\sum_{j=0}^{2^{n}-1}
\frac{1}{(s-s_{2j+1})^{1-\eta}} \,   \right).
\end{equation*}
Furthermore, the following uniform bound holds true:
\begin{equation*}
\frac{s}{2^{n}} \, \sum_{j=0}^{2^{n}-1}
\frac{1}{(s-s_{2j+1})^{1-\eta}} \,
\le c\,
\int_{0}^{s} \frac{dr}{r^{1-\eta}}
= c\, s^{\eta},
\end{equation*}
and thus
\begin{equation*}
\sum_{n=0}^{\infty} \| L^{1,n+1}_{st}-L^{1,n}_{st}\|_{\cb_{
w_{t}}^{\ka_{u}} } \le c \, s^{\eta} \, (t-s)^{\te_{u}}
\sum_{n=0}^{\infty} \lp \frac{s}{2^{n}} \rp^{\ep} \le c\,
s^{\eta+\ep} \, (t-s)^{\te_{u}},
\end{equation*}
which ensures the convergence of $L^{1,n}_{st}$. Finally, invoking our definition \eqref{eq:def-L-1n} plus the fact that $u_{0}^{12}=0$, it is readily checked  that $L^{1,0}_{st}=0$. Thus the relation above transfers into:
\begin{eqnarray}\label{eq:bnd-L1}
\| L^{1}_{st}\|_{\cb_{ w_{t}}^{\ka_{u}} }  &\le& c \, s^{\eta+\ep} \,
\| u^{12} \|_{\cd_{\la,\si}^{\te_{u},\ka_{u}}}
\, \|\W \|_{\cac_{\rho_{\si}}^{\te, -\ka}}(t-s)^{\te_{u}}  \notag \\
&\le&  c \, \tau^{\eta+\ep}
\, \| u^{12} \|_{\cd_{\la,\si}^{\te_{u},\ka_{u}}}
\, \|\W \|_{\cac_{\rho_{\si}}^{\te,-\ka}} \,  (t-s)^{\te_{u}}.
\end{eqnarray}

\smallskip

\noindent
\textit{Step 4: Bound on $L^{2}_{st}$.}
The bound on $L^{2}_{st}$ follows along the same lines as for $L^{1,n}_{st}$, and is in fact slightly easier. Let us just mention that we approximate $L^{2}_{st}$ by a sequence $L^{2,n}_{st}$ based on the dyadic partition of $[s,t]$, namely $s_{j}^{n}=s+j 2^{-n} (t-s)$. Like in Step 2, we end up with some terms $L^{21,n,j}_{st}, L^{22,n,j}_{st}$, where 
\[
L^{21,n,j}_{st}=
 p_{s-s_{2j+1}} ( \delta u^{12}_{s_{2j}s_{2j+1}} \, \delta \W_{s_{2j+1}s_{2j+2}} )
 \]
  and 
\begin{equation*}
L^{22,n,j}_{st}=
 \lc p_{s_{2j+1}-s_{2j}} -\id  \rc  p_{t-s_{2j+1}}
 \lp u^{12}_{s_{2j}} \, \delta \W_{s_{2j+1}s_{2j+2}} \rp.
\end{equation*}
From this decomposition, we leave to the patient reader the task of checking that relation~\eqref{eq:bnd-L1} also holds true for $L^{2}_{st}$.

\smallskip

\noindent
\textit{Step 5: Conclusion.}
Putting together the last 2 steps, we have been able to prove that for all $0\le s<t\le\tau$ we have
\begin{equation*}
\| v^{12}_{t} - v^{12}_{s}\|_{\cb_{ w_{t}}^{\ka_{u}} } \le c \,
\tau^{\eta+\ep} \, \| u^{12} \|_{\cd_{\lambda, \si} ^{\te_{u},\ka_{u}}} \,
\| \W \|_{\cac_{\rho_{\si}}^{\te,-\ka}}  \, (t-s)^{\te_{u}} .
\end{equation*}
Thus, choosing $\tau = (c\,\|\W \|_{\cac_{\rho_{\si}}^{\te,-\ka}}/2)^{1/(\ep+\eta)}$, this yields
\begin{equation*}
\| v^{12}_{t} - v^{12}_{s}\|_{\cb_{ w_{t}}^{\ka_{u}} } \le \frac12
\, \| u^{12} \|_{\cd_{\lambda, \si} ^{\te_{u},\ka_{u}}}
  \, (t-s)^{\te_{u}},
\end{equation*}
namely the announced contraction property. We have thus obtained existence and uniqueness of the solution to equation \eqref{eq:mild-young-she-space-time} on $[0,\tau]$. In order to get a global solution on an arbitrary interval, it suffices to observe that all our bounds above do not depend on the initial condition of the solution. One can thus patch solutions on small intervals of constant length $\tau$. The continuity result \emph{(b)} is obtained thanks to the same kind of considerations, and we spare the details to the reader for sake of conciseness.

\end{proof}

\subsubsection{Identification of the Feynman-Kac solution}
This section is devoted to the identification of the solution to the stochastic heat equation given by the Feynman-Kac representation formula and the pathwise solution constructed in this section. Calling $u^{F}$ the Feynman-Kac solution, the global strategy for this identification procedure is the following:
\begin{enumerate}
\item
Relate the covariance structure \eqref{cov1} of the Gaussian noise $W$ to Hypothesis \ref{hyp:W-pathwise-space-time}. We shall see that our Hypothesis \ref{hyp:mu-holder}  implies that $W$ satisfies \ref{hyp:W-pathwise-space-time} almost surely for suitable values of the parameters $\te$ and $\kappa$.
\item
Prove that $u^{F}$ coincides with the pathwise solution to \eqref{eq:mild-young-she-space-time}, by means of approximations of the noise $W$.
\end{enumerate}
We now handle those three problems.

\smallskip

Let us start by establishing  the  pathwise property of   $W$ as a distribution valued function.

\begin{proposition}\label{prop:hyp-gauss-pathwise}
Let $W$ be a centered Gaussian noise defined by $\mu$ and $\ga$ as in  \eqref{cov1}, satisfying Hypothesis \ref{hyp:mu-holder} for some $0<\al<1-\beta$. 
Then  the mapping $(t,\varphi)\rightarrow W(\mathbf{1}_{[0,t]} \varphi)$ is almost surely  H\"older continuous of order $\theta$ in time with values in   $\cb^{   -\kappa}_{\rho_\sigma}$  for arbitrarily small $\si$ and for all $\te, \ka\in(0,1) $ such that $ \te <1-\frac \beta 2$ and
$ \ka >1-\al-\beta$.  That is, almost surely $W$ satisfies  Hypothesis   \ref{hyp:W-pathwise-space-time}.
Moreover,  $\|W\|_{  \mathcal {C}^{\te,-\ka}_{\rho_\si}}$  is a random variable which admits moments of all orders.
 
\end{proposition}

\begin{proof}[Proof of Proposition \ref{prop:hyp-gauss-pathwise}]
Fix $\ka >\ka' >1-\al-\beta$.
For $q\ge 1$, let us denote the Besov space
$\cb_{2q,2q,\rho_{\si}}^{-\ka' }$ by $\ca_{q}$, and recall that
the norm on $\ca_{q}$ is given by:
\begin{equation*}
\| f\|_{\ca_{q}}^{2q} = \sum_{j\ge -1} 2^{-2qj \ka'  } \|\Delta_{j} f\|_{L_{\rho_{\si}}^{2q}}^{2q}.
\end{equation*}
We will choose $q $ large enough so that
$\ca_{q}\hookrightarrow \cb_{\rho_{\si}}^{-\ka }$, a fact which is
ensured by Besov embedding theorems. We will show that almost
surely:
\begin{equation}\label{eq:as-bnd-increments-W}
\| \delta W_{st}\|_{\ca_{q}} \le Z \, (t-s)^{ \te},
\end{equation}
for any $\te\in (0,  1-\frac \beta 2  )$ and the random variable $Z$ admitting moments of all orders. This will
 complete the proof of the proposition.

\smallskip

 To this aim, recall from Section \ref{sec:besov-spaces}  that $\Delta_{j} f(x) = [K_{j}*f](x)$, where $K_{j}(z)=2^{jd} K(2^{j}z)$ and $K$ is the inverse Fourier transform of $\vp$. Otherwise stated, $K_{j}$ is the inverse Fourier transform of $\vp_{j}$.
With these preliminary considerations in mind, set $K_{j,x}(y):=K_{j}(x-y)$ and evaluate:
\begin{eqnarray}\label{eq:bnd-W-in-Aq}
\be\lc \|\delta W_{st} \|_{\ca_{q}}^{2q} \rc
&=&
\sum_{j\ge -1} 2^{-2qj\ka'}
\int_{\R^{d}} \be\lc \lln W\lp \1_{[s,t]} \otimes K_{j,x} \rp \rrn^{2q} \rc \rho_{\si}^{2q}(x) \, dx  \notag \\
&\le& c_{q}
\sum_{j\ge -1} 2^{-2qj\ka'}
\int_{\R^{d}} \be^{q}\lc \lln W\lp \1_{[s,t]} \otimes K_{j,x} \rp \rrn^{2} \rc \rho_{\si}^{2q}(x) \, dx
\end{eqnarray}
Moreover, we have  
\begin{eqnarray}\label{eq:bnd-var-W-Kj}
\be\lc \lln W\lp \1_{[s,t]} \otimes K_{j,x} \rp \rrn^{2} \rc
&=&
\int_{[s,t]^{2}} \lp \int_{\R^{d}} \lln  \cf K_{j,x} \rrn^{2} \mu(d\xi) \rp  \ga(u-v)  \, du dv \notag \\
&\le&
(t-s)^{2-\beta } \int_{\R^{d}} \lln  \vp\lp 2^{-j} \xi \rp \rrn^{2} \, \mu(d\xi).
\end{eqnarray}
Let us introduce the measure $\nu(d\xi)=\mu(d\xi)/(1+|\xi|^{2(1-\al-\beta)})$, which is a finite measure on $\R^{d}$ according to our standing assumption. Also recall from Notation \ref{not:partition-unity} that ${\rm Supp}(\vp)\subset \{ x \in \mathbb{R}^d : a \leqslant | x | \leqslant b \}$. Hence
\begin{equation*}
\int_{\R^{d}} \lln  \vp\lp 2^{-j} \xi \rp \rrn^{2} \, \mu(d\xi)
\le
\int_{\R^{d}} \1_{[0,2^{j}b]}(|\xi|) \lc 1+ |\xi|^{2(1-\al-\beta)} \rc \, \nu(d\xi)
\le
c_{\mu} \, 2^{2(1-\al-\beta) j}.
\end{equation*}
Plugging this identity into \eqref{eq:bnd-var-W-Kj} and then \eqref{eq:bnd-W-in-Aq} we end up with the relation
$\be[ \|\delta W_{st} \|_{\ca_{q}}^{2q} ] \le c_{q} (t-s)^{ (2-\beta)  q}$, valid for all $0\le s < t \le T$ and any $q\ge 1$. A standard application of Garsia's and Fernique's lemma then yields relation \eqref{eq:as-bnd-increments-W}, and thus Hypothesis \ref{hyp:W-pathwise-space-time} .
\end{proof}

\smallskip
\begin{remark}
 In particular, equation
\eqref{eqBe1} driven by $W$ admits a unique
pathwise solution in $\cd_{\lambda, \si}^{\te_{u},\ka_{u}}$, as in
Theorem \ref{thm:pathwise-solt-space-time}, for some  $\te_u >\frac \beta 2$ and $\ka_u >1-\al-\beta$. Notice here that one obtains (see Theorem \ref{thmFK1}) the existence of a solution to our equation in the Stratonovich sense under Hypothesis \ref{hyp:mu2} only. We call this assumption the critical case. In order to get existence and uniqueness of a pathwise solution we have to impose the more restrictive Hypothesis \ref{hyp:mu-holder} with an arbitrarily small constant $\al$, which can be seen as a supercritical situation. This is the price to pay in order to get uniqueness of the solution.
\end{remark}

 We now turn to the second point of our strategy, namely prove that the Feynman-Kac solution $u^{F}$  coincides with the  unique pathwise solution to equation  \ref{eqBe1} driven by $W$.

\begin{proposition}\label{prop:identif-FK-pathwise}
Let $u^F$ be the random field given by equation (\ref{FK}).  Assume that $W$ satisfies Hypothesis  \ref{hyp:mu-holder}. Then there exist $\theta_{u}>\frac{\beta}{2}$ and $\ka _u> 1-\al-\beta $ such that almost surely $u^F$ belongs to the space $\cd_{\lambda, \si}^{\te_{u},\ka_{u}}$. 
Moreover, $u^F$ is the pathwise solution to equation  (\ref{eqBe1}) driven by $W$.
\end{proposition}

 \begin{proof}
   To show that $u^F$ is the pathwise solution to equation  \ref{eqBe1}, we use the fact that $u_{t,x}^F$ is the limit  in $L^p(\Omega)$ of the approximating sequence $u^{ \ep,\del}_{t,x}$ introduced in (\ref{FKaprox}) (see (\ref{m9})) as $\ep$ and $\del$ tend to zero, for any $p\ge 1$.
 On the other hand,  it is clear that $u^{ \ep,\del}$ is the pathwise solution to equation (\ref{eqBe1}) driven by the trajectories
 of $W^{\ep,\del}$
 \[
u^{\ep,\del}_{t} = p_{t}u_{0} + \int_{0}^{t} p_{t-s} \lp u^{\ep,\delta}_{s} \, W^{\ep,\del}(ds) \rp.
 \]
 Then, it suffices to take the limit in the above equation   to show that $u^F$ is a  pathwise solution to equation (\ref{eqBe1}) driven by $W$.  In fact,  that for two particular sequences $\ep_n \downarrow 0$ and $\del_n\downarrow 0$
  $W^{\ep_n,\del_n}$  converges to $W$ almost surely in the space  $\mathcal{C}^{\theta, -\ka}_{T, \rho_\si}$. This implies (see Theorem \ref{thm:pathwise-solt-space-time} item (b)) that 
  $u^{\ep_n,\del_n}$ converges  almost surely to a process $u$ in $\mathcal{D}^{\theta_u, \ka_u}_{\la,\si}$, which is  the pathwise solution to equation \ref{eqBe1} driven by $W$. Therefore, $u=u^F$ and this concludes the proof. 
 \end{proof}

\subsubsection{Time independent case}

 The case of a time independent noise is obviously easier to handle than the time dependent one. Basically, the Young integration arguments invoked above can be skipped, and they are replaced by Gronwall type lemmas for Lebesgue integration. We won't detail the proofs here, and just mention the main steps for sake of conciseness.

\smallskip
First, the pathwise type assumption we make on the noise $W$, considered as a distribution on $\R^{d}$, is the following counterpart of Hypothesis \ref{hyp:W-pathwise-space-time}:

\begin{hypothesis}\label{hyp:W-pathwise-d-dim}
Suppose that $\W$ is a distribution on $\R^{d}$ such that $\W\in \cb_{\rho_{\si}}^{-\ka}$  with
$\ka \in( 0,1)$ and an arbitrarily small constant $\si>0$.
\end{hypothesis}

Another simplification of the time independent case is that one can solve the equation in a space of continuous functions in time (compared to the H\"older regularity we had to consider before), with values in weighted Besov spaces. We thus define the following sets of functions
\begin{equation*}
\cac_{\la,\si}^{\ka_{u}} = \lcl f\in
\cac([0,T]\times\R^{d}) ; \|f_{t}\|_{\cb_{w_{t}}^{\ka_{u}}} \le
c_{f} \rcl, \quad\text{where}\quad w_{t} := e_{\la+\si t}.
\end{equation*}

With these conventions in hand, we interpret equation \eqref{eqSt2} as a mild equation in the spaces $\cac_{\la,\si}^{\ka_{u}}$.

\begin{definition}\label{def:heat-mild-d-dim}
Let $u\in\cac_{\la,\si}^{\ka_{u}}$ for $\la,\si>0$ and
$\ka_{u}\in ( \ka, 1)$. Consider an initial condition $u_{0}\in\cb_{e_{\la}}^{\ka_{u}}$.
We say that $u$ is a mild
solution to equation
\begin{equation}
\frac {\partial u}{\partial t} = \frac 12 \Delta u+ u \W
\end{equation}
with initial condition $u_0$ if it satisfies the following
integral equation
\begin{equation}\label{eq:mild-young-she-d-dim}
u_{t} = p_{t}u_{0} + \int_{0}^{t} p_{t-s} \lp u_{s} \, \W \rp \, ds,
\end{equation}
where the product $u \,  \W$ is interpreted in the distributional
sense.
\end{definition}

We can now turn to the resolution of the equation in this context, and the main theorem in this direction is the following.

\begin{theorem}
Let   $\W$ be  a distribution satisfying Hypothesis \ref{hyp:W-pathwise-d-dim} and let $\la$ be a strictly positive constant.
 Then equation
\eqref{eq:mild-young-she-d-dim} admits a unique solution in
$\cac_{\la,\si}^{\ka_{u}}$, in the sense given by Definition
\ref{def:heat-mild-d-dim}, with $ \ka<\ka_{u}<1$.
\end{theorem}

\begin{proof}
As in the proof of Theorem \ref{thm:pathwise-solt-space-time}, we focus on the proof of uniqueness, and fix a small time interval $[0,\tau]$. Consider $u^{1},u^{2}$ two solutions in $\cac_{\la,\si}^{\ka_{u}}$ and we set $u^{12}=u^{1}-u^{2}$. Consistently with Definition \ref{def:heat-mild-d-dim}, the equation for $u^{12}$ is given by:
\begin{equation}\label{eq:def-u12-d-dim}
u^{12}_{t} = \int_{0}^{t} p_{t-s} \lp  u^{12}_{s}  \, \W \rp \,
ds,
\end{equation}
and we wish to prove that $u^{12}\equiv 0$.

\smallskip

Towards this aim, let us bound the Besov norm of $u$ starting from equation \eqref{eq:def-u12-d-dim}. Owing to Proposition \ref{prop:heat-semigroup-besov}, we get
\begin{equation*}
\| u^{12}_{t} \|_{\cb_{w_{t}}^{\ka_{u}}} \le \int_{0}^{t}  \| p_{t-s} \lp
 u^{12}_{s}  \, \W \rp \|_{\cb_{w_{t}}^{\ka_{u}}} \, ds \le
c_{\tau,\la,\si} \int_{0}^{t}
(t-s)^{-\frac{(\ka_{u}+\ka)}{2}} \, \|  u_{s}^{12} \, \W
\|_{\cb_{w_{t}}^{-\ka}} \, ds .
\end{equation*}
Along the same lines as in the proof of Theorem \ref{thm:pathwise-solt-space-time}, we now invoke the bound \eqref{eq:bnd-polynomial-exponential}, which
yields $w_{t}\le c _{\tau,\la,\si} \, (t-s)^{-\si} w_{s} \, \rho_{\si}$. Hence, according to Proposition
\ref{prop:basic-weighted-besov} item (iii), we have
\begin{equation*}
\| u^{12}_{t} \|_{\cb_{w_{t}}^{\ka_{u}}} \le c_{\la,\si}
\int_{0}^{t}  (t-s)^{-\frac{(\ka_{u}+\ka)}{2} -\si} \, \|
u_{s}^{12} \W  \|_{\cb_{w_{s} \rho_{\si}}^{-\ka}} \, ds .
\end{equation*}
Since $\ka_{u}>\ka$, we now apply relation
\eqref{eq:bnd-product-besov} with $w_{1}=w_{s}$, $\ka_{1}=\ka_{u}$,
$w_{2}=\rho_{\si}$ and $\ka_{2}=\ka$. We end up with
\begin{equation*}
\| u^{12}_{t} \|_{\cb_{w_{t}}^{\ka_{u}}} \le c_{\tau,\la,\si} \, 
\|\W\|_{\cb_{\rho_{\si}}^{-\ka}} 
\int_{0}^{t} \frac{\|u_{s}^{12}\|_{\cb_{w_{s}}^{\ka_{u}}}}{(t-s)^{\frac{(\ka_{u}+\ka)}{2}
+\si} } \, ds\,.
\end{equation*}
Taking into account   that $\ka_{u}+\ka<2$ and $\si$  can be
arbitrarily small, our conclusion $u^{12}\equiv 0$ follows easily from  a Gronwall
type argument.
\end{proof}

We now state a result which allows to identify the Feynman-Kac and the pathwise solution to our spatial equation. Its proof is omitted for sake of conciseness, since it is easier than in the time dependent case.

\begin{proposition}
Let $W$ be a spatial Gaussian noise defined by the covariance structure~\eqref{cov2} and \eqref{innprod2}. Assume that the measure $\mu$ satisfies the condition 
\begin{equation}\label{eq:hyp-mu-pathwise-spatial-noise}
\int_{\R^{d}} \frac{\mu(d\xi)}{1+|\xi|^{2(1-\al)}}   <\infty,
\end{equation}
for a constant $\al\in (0,1)$. Then:

\smallskip

\noindent
\emph{(i)} There exists $\ka\in( 0,1)$ such that   for any arbitrarily $\si>0$, $W$ has a version in $\cb^{-\ka}_{\rho_\si}$ and the random variable $\|W\| _{\cb^{-\ka}_{\rho_\si}}$ has moments of all orders, that is the trajectories of $W$ satisfy Hypothesis   \ref{hyp:W-pathwise-d-dim}.
As a consequence, equation \eqref{eq:mild-young-she-d-dim} driven by  the trajectories of $W$ admits a unique pathwise solution in $\cac  _{\la,\si} ^{\ka_{u}}$.

\smallskip

\noindent
\emph{(ii)} Let $u^{F}$ be the Feynman-Kac solution to the heat equation given by \eqref{FK2}. Then almost surely the process $u^{F}$ lies into $\cac _{\la,\si}^{\ka_{u}}$, and it coincides with the unique pathwise solution to equation \eqref{eq:mild-young-she-d-dim}.
\end{proposition}

\begin{remark}
Here again, we see that the Feynman-Kac solution $u^{F}$ exists under the critical condition $\int_{\R^{d}}(1+|\xi|^{2})^{-1}\mu(d\xi)<\infty$, while the pathwise solution requires the more stringent condition \eqref{eq:hyp-mu-pathwise-spatial-noise}.
\end{remark}

\section{Moment estimates}\label{sec:moment-estimates}

 As mentioned in the introduction, intermittency properties for $u$ are characterized by the family of Lyapounov type coefficients $\ell(k)$ defined by \eqref{eq:def-lyapunov-exponent} or by the limiting behavior \eqref{eq:fraction-moments-u-k}. In any case, the intermittency phenomenon stems from an asymptotic study of the moments of $u$, for large values of $k$ and $t$. We propose to lead this study in the context of the general Gaussian noises considered in the current paper.

\smallskip

Notice that delicate results such as limiting behaviors for moments will rely on more specific conditions on the noise $W$.
We are thus going to make use of the following conditions.

\begin{hypothesis} \label{h2}
 There exist constants $c_0,C_0$ and $0<\beta<1$, such that
\[
c_0 |x|^{-\beta} \le \gamma(x) \le C_0 |x| ^{-\beta}.
\]
\end{hypothesis}

\begin{hypothesis} \label{h3}
 There exist constants $c_1,C_1$ and $0<\eta<2$, such that
\[
c_1 |x|^{-\eta} \le \laa(x) \le C_1 |x| ^{-\eta}.
\]
\end{hypothesis}

\begin{hypothesis} \label{h4} There exist  constants
$c_2,C_2$ and $0<\eta_i<1$, with $\sum_{i=1}^d \eta_i <2$, such that
\[
  c_2   \prod_{i=1}^d |x_{i}| ^{-\eta_i}\le \laa (x) \le
 C_2  \prod_{i=1}^d |x_{i}| ^{-\eta_i}.
\]
\end{hypothesis}

Clearly, Hypothesis  \ref{h2} and Hypothesis \ref{h3} generalize the case of Riesz kernels and Hypothesis  \ref{h4} generalizes the case of fractional noises. Notice that under Hypotheses \ref{h3} or~\ref{h4} the spectral measure $\mu$ satisfies the integrability condition (\ref{mu1}).

\begin{theorem}\label{thmBd1}
Suppose that  $\gamma$ satisfies Hypothesis \ref{h2} and $\laa$
satisfies  Hypothesis \ref{h3}  or Hypothesis \ref{h4}.  Denote
\[
a=
\begin{cases} \eta & \quad \hbox{  if Hypothesis \ref{h3} holds}\\
 \sum_{i=1}^d \eta_i& \quad \hbox{ if Hypothesis \ref{h4}  holds}.  \\
 \end{cases}
\]
 Consider the following two cases:
\begin{itemize}
\item[(i)]   $u$   is the solution to the Skorohod  equation \eref{eqSk1} driven by a time dependent noise with time covariance $\gamma$ and space covariance $\laa$.
\item[(ii)]     $u$ is the solution to  the Stratonovich equation \eref{eqSt1} driven by a time dependent noise with time covariance $\gamma$ and space covariance $\laa$, and we assume    that $a<2-2\beta$. 
\end{itemize}
 Then in both of these two cases   we have
\begin{equation}   \label{es0}
\exp\left(C t^{\frac{4-2\beta-a}{2-a}}k^{\frac{4-a}{2-a}}\right)\leq
\be \lc u_{t,x}^k\rc\leq  \exp\left(C^{\prime}
t^{\frac{4-2\beta-a}{2-a}}k^{\frac{4-a}{2-a}}\right)
\end{equation}
for all $t\ge 0\,, x\in \R^d\,, k\ge 2$,   where $C, C^{\prime}$ are
constants independent of $t$ and $k$.
\end{theorem}

\begin{proof}
Let us first discuss the upper bound.  For the Skorohod equation,
using the chaos expansion and the hypercontractivity property we can
derive the upper bound as it has been done in  \cite{BalanConus}.
 For the Stratonovich equation,  notice first that Hypothesis  \ref{hyp:mu-holder} holds
because  $a<2-2\beta$.
 Using the Feynmann-Kac formula
\eref{FK} for the solution to equation \eref{eqSt1},  and applying
Cauchy-Schwartz inequality yields
\begin{eqnarray*}
\be \lc u_{t,x}^k\rc &=&\be_B \lc  \exp \left(\sum_{1\leq i,j\leq
k}\int_0^t \int_0^t \gamma(r-s)\laa(B_r^i-B_s^j)dr ds\right)
\right]\,\\
&\leq & \left[\be_B  \left[\exp \left(2 \sum_{1\leq i < j \leq
k}\int_0^t \int_0^t \gamma(r-s)\laa(B_r^i-B_s^j)dr ds\right) \right]
\right]^{\frac{1}{2}}  \\
&&\times\left[\be_B\left[\exp \left(2
\sum_{i=1}^k\int_0^t \int_0^t \gamma(r-s)\laa(B_r^i-B_s^i)dr ds
\right) \right]\right]^{\frac{1}{2}}\,.
\end{eqnarray*}
In the above expression, the first term is just the square root of
the Feynman-Kac formula~\eqref{momSk1} for the moment of  order $k$ of the solution
of a Skorohod equation with multiplicative noise, with covariances
$2\gamma$ and $2\laa$. For this term we know that we can derive the
upper bound~(\ref{es0})   using the chaos expansion and the
hypercontractivity property as it has been done in
\cite{BalanConus}.
 For the second factor,  using the asymptotic result proved in Proposition~2.1 in
 \cite{ChenHuSongXing}, we  derive the
 estimate
 \[
\be ^{\frac k2} \left[ \exp \left(2 \int_0^t \int_0^t
\gamma(r-s)\laa(B_r^1-B_s^1)dr ds \right) \right] 
\leq C^k \exp \left( C  t^{\frac{4-2\beta -a}{2-a}}k\right).
\]
Therefore, in this way we can obtain the desired upper bound of $\be \lc u_{t,x}^k\rc$.

\smallskip

Let us now discuss the lower bound. Taking into account again the
Feynman-Kac formula~\eqref{momSk1} for the moments of $u$, it suffices to consider the
case of the Skorohod equation  (it is readily checked from ~\eqref{momSk1} that the moments of $u$ for the Stratonovich equation are greater than those of the Skorohod equation). The argument of the proof is then based in
the small ball probability estimates for Brownian motion.
We
consider only the case when $\laa$ satisfies the lower bound given
in  hypothesis Hypothesis \ref{h3} (Riesz kernel case), since the case
Hypothesis \ref{h4} (fractional noise) is analogous.
 In this case, owing to formula~\eref{momSk1}  and the scaling property  of the Brownian motion,
it is easy to see that
\[
\be \lc u_{t,x}^k\rc \geq  \be  \left[
\exp\left(c_0c_1 t^{2-\beta-\frac{\eta}{2}}\sum_{1\leq i < j \leq k}\int_0^1
\int_0^1 |s-r|^{-\beta} |B_s^i-B_r^j|^{-\eta} ds
dr\right)\right]\,.
\]
Denote $B_s^{i,l}, l=1,2, \cdots, d$ the $l$-th component of the
$d$-dimensional Brownian motion $B_s^i$.  Consider the set
\[
A_\varepsilon=\left\{\sup_{1\leq i<j\leq k}\sup_{1\leq l \leq
d}\sup_{0\leq s, r \leq 1}|B_s^{i,l}-B_r^{j,l}|\leq
\varepsilon\right\}.
\]
Restricting the above expectation to this event and recalling that the value of a generic constant $c$ might change from line to line, we obtain:
\begin{eqnarray}\label{eq:low-bnd-mom-u-1}
\be \lc u_{t,x}\rc^k
&\geq&\be \left[\exp \left(c \, t^{2-\beta-\frac{\eta}{2}}\sum_{1\leq i< j\leq
k}\int_0^1 \int_0^1 |s-r|^{-\beta}|  B_s^i-B_r^j|^{-\eta}ds
dr\right)   {\bf 1}_{A_{\varepsilon}}\right]\\
&\geq&\exp\left(   \frac {c \, k(k-1)}{(2-\beta)(1-\beta)}
t^{2-\beta-\frac{\eta}{2}} \varepsilon^{-\eta}\right) \, \bp\big( A_\varepsilon\big)
\geq 
\exp\big(c \, t^{2-\beta-\frac{a}{2}} k^2 \varepsilon^{-\eta}\big) \, \bp\big( A_\varepsilon\big). \notag
\end{eqnarray}
Moreover, notice that
\begin{equation*}
\cap_{i=1}^{k} \cap_{l=1}^{d} F_{i,l} \subset  A_{\ep} ,
\quad\text{with}\quad
F_{i,l} = \lp \sup_{0\leq s \leq 1}|B_s^{i,l}|\leq \frac{\varepsilon}{2} \rp.
\end{equation*}
The events $F_{i,l}$ being i.id, we get:
\begin{equation*}
\bp\big( A_\varepsilon\big) \ge \bp^{kd}\lp F_{\ep} \rp,
\quad\text{with}\quad
F_{\ep} = \lp \sup_{0\leq s \leq 1}|b_s|\leq \frac{\varepsilon}{2} \rp,
\end{equation*}
where $b$ stands for a one dimensional standard Brownian motion. In addition, it is a well known fact (see e.g (1.3) in \cite{LS}) that  $\lim_{\varepsilon \to 0}\bp(F_{\ep})/\exp(-\frac{\pi^2}{2 \varepsilon^2})=1$. Hence, there exists an $\varepsilon_0>0$ such
that for $\varepsilon\le \varepsilon_0$, we have $\bp(F_{\ep})\geq\exp(- C  \varepsilon^{-2})$, for some constant $C>0$. Under the condition $\varepsilon\leq \varepsilon_0$, this entails:
\begin{eqnarray*}
\be\lc u_{t,x}\rc^k\geq  \exp \left( c \,
t^{2-\beta-\frac{\eta}{2}}k^2\varepsilon^{-\eta}-\frac{C dk}{\varepsilon^2}\right)\, .
\end{eqnarray*}
In order to optimize this expression, we try to equate the two terms inside the exponential above. To this aim,  we set
\begin{equation*}
\varepsilon=\frac{t^{\frac{2-\beta-\frac{\eta}{2}}{\eta-2}}(c \, k)^{\frac{1}{\eta-2}}}{(2dC)^{\frac{1}{\eta-2}}}\,,
\end{equation*}
and notice that for  $k\geq 2$ and $t$ sufficiently large, the condition
$\varepsilon \leq  \varepsilon_0$ is fulfilled. Therefore, we
conclude that for $t$ and $k$ large enough
\begin{equation}
\be \lc u_{t,x}^k\rc \geq \exp
\left(\frac{   c^{\frac{\eta}{2-\eta}} \, t^{\frac{4-2\beta-\eta}{2-\eta}}k^{\frac{4-\eta}{2-\eta}}}{8
(2dC)^{\frac{a}{2-a}}}\right),
\end{equation}
which finishes the proof of  (\ref{es0}).
\end{proof}

 We now give two extensions of the theorem above. The first one concerns 
the moment estimates in the time independent case. Its proof is very similar to the proof of Theorem~\ref{thmBd1}, and is thus omitted for sake of conciseness.

\begin{theorem}
Suppose that   $\laa$ satisfies  Hypothesis \ref{h3}  or Hypothesis
\ref{h4}.   Set $a=\eta$  if  Hypothesis \ref{h3}  holds, and $a=\sum_{i=1}^d
\eta_i$ if  Hypothesis  \ref{h4} holds.  Suppose that $u$   is the
solution to the Skorohod  equation \eref{eqSk2} or the Stratonovich
equation \eref{eqSt2}  driven by a multiplicative  time independent
noise  with covariance $\laa$.
 Then,   for any $x \in \R^d$,  $k
\geq 2$, we have
\begin{equation}
\exp\left (C t^{\frac{4-a}{2-a}}k^{\frac{4-a}{2-a}}\right)\leq \be
\lc  u_{t,x} ^k\rc\leq  \exp\left(C^{\prime}
t^{\frac{4-a}{2-a}}k^{\frac{4-a}{2-a}}\right)\,,
\end{equation}
 where $C, C^{\prime} > 0$  are constants independent of  $t$ and $k$.
\end{theorem}

Finally, when $d=1$ we can also obtain moment estimates in the case where the space covariance is a Dirac delta function, that is, the noise is white in space.

\begin{theorem}
Suppose that $\gamma$ satisfies condition Hypothesis \ref{h2} and
the spatial  dimension is $1$. Consider two cases:
\begin{itemize}
\item[(i)]
Suppose that $u$ satisfies  either the Skorohod equation
\eref{eqSk1}  or the Stratonovich equation (\ref{eqSt1}) driven by a
multiplicative noise with time   covariance  $\gamma$ and spatial
covariance $\laa(x)=\delta_0(x)$.
 Then,   for any $x \in \R^d$,  $k
\geq 2$ and  $t>0$, we have
\begin{equation}   \label{es1}
\exp\left(C
t^{3-2\beta}k^{3}\right)\leq \be \lc   u_{t,x}^k\rc \leq  \exp\left(C^{\prime}
t^{3 -2\beta}k^{3}\right)\,,
\end{equation}
 where $C, C^{\prime} > 0$  are constants independent of  $t$ and $k$.
 \item[(ii)]
 Suppose that $u$ satisfies  either the Skorohod equation \eref{eqSk2}  or the Stratonovich equation (\ref{eqSt2}) driven by a time independent multiplicative noise with spatial  covariance $\laa(x)=\delta_0(x)$.
 Then,   for any $x \in \R^d$,  $k
\geq 2$ and  $t>0$, we have
\begin{equation}   \label{es2}
\exp\left(C t^{3}k^{3}\right)\leq \be \left[ u_{t,x}^k\right]\leq
\exp\left(C^{\prime} t^{3 }k^{3}\right)\,,
\end{equation}
 where $C, C^{\prime} > 0$  are constants independent of  $t$ and $k$.
 \end{itemize}
\end{theorem}

\begin{proof}
In the Skorohod case with time dependent noise, the moments of $u_{t,x}$ are given by equation  \eref{momwhite}.
We will only discuss the lower bound because the upper bound can be obtained by using chaos expansions as in  \cite{BalanConus}.
We consider the approximation of the Dirac delta function by the heat kernel $p_{\ep}$, and define
\begin{equation}
I_{t,k,\varepsilon}=\be _B \left[\exp \left(\sum_{1\leq i< j \leq
k}\int_0^t \int_0^t
\gamma(s-r)p_{\varepsilon}(B_s^i-B_r^j)dsdr\right)\right]\,.
\end{equation}
Expanding the exponential and using Fourier analysis as in
\cite{HN}, one can show that $\be \lc
u_{t,x}^k\rc\geq I_{t,k,\varepsilon}$, for any $ \varepsilon>0$.  For any positive
$\varepsilon$, denote
\begin{equation*}
A_{k,\varepsilon,t}=
\left\{\max_{1\leq i \leq k}\sup_{0\leq s \leq t}|B^i_s|\leq \sqrt{\varepsilon}\right\}\,.
\end{equation*}
On the event $A_{k,\varepsilon,t}$ we have
$p_{\varepsilon}(B_s^i-B_r^j)\geq \frac{C}{\sqrt{\varepsilon}}$ for
some positive constant $C$.  Therefore, using the lower bound in
Hypothesis \ref{h2}, we can write similarly to \eqref{eq:low-bnd-mom-u-1}:
\begin{eqnarray*}
I_{t,k,\varepsilon}&\geq& \exp \left(c\, k^{2} \int_0^t \int_0^t |s-r|^{-\beta}\frac{C}{\sqrt{\varepsilon}}ds dr\right)\bp \left(A_{k,\varepsilon,t}\right)\,.
\end{eqnarray*}
Furthermore, by the scaling property of Brownian motion, $\bp \left(A_{k,\varepsilon,t}\right)$ can be written as:
\begin{equation*}
\bp \lp A_{k,\varepsilon,t}\rp= \bp  \left(\max_{1\leq i \leq
k}\sup_{0\leq s \leq 1}|B_s^i|\leq \sqrt{\varepsilon/
t}\right)=\left(\bp \left(\max_{0 \leq s \leq 1}|b_s|\leq
\sqrt{\varepsilon/t}\right) \right)^k\, ,
\end{equation*}
where $b$ stands  for a one-dimensional standard Brownian motion. We now invoke again~(1.3) in \cite{LS},  which yields $\lim_{\varepsilon \to
0}\bp (\sup_{0\leq s \leq 1}|B_s|\leq
\sqrt{\frac{\varepsilon}{t}})/\exp(-\frac{\pi^2}{8}\frac{t}{\varepsilon})=1$. Thus, when $\varepsilon$ is sufficiently small,
\begin{equation*}
\bp \left(\sup_{0\leq s \leq 1} |B_s|\leq \sqrt{\frac{\varepsilon}{t}} \right)\geq 
\exp \left(-C\frac{t}{\varepsilon}\right) ,
\end{equation*}
for some positive constant $C$ which does not depend on $t$.  Hence, we end up with the following lower bound:
\begin{eqnarray*}
I_{t,k,\varepsilon}\geq \exp \left(C_1 k^2 t^{2-\beta}\frac{1}{\sqrt{\varepsilon}}-C_2 \frac{t}{\varepsilon}\right)\,.
\end{eqnarray*}
As in the proof of Theorem \ref{thmBd1}, we optimize this expression by choosing $\varepsilon = \frac{4C_2^2}{C_1^2 k^3 t^{2-2\beta}}$,
and we obtain that
\begin{equation}
I_{t,k,\varepsilon}\geq \exp (C_3 t^{3-2\beta}k^3)
\end{equation}
when $t$ is sufficiently large, where the positive constant $C_3$ does not depend on $t$ or $k$.

\smallskip

For the Stratonovich case, the lower bound is obvious and for the
upper bound we use  the Cauchy-Schwartz inequality and  Lemma 2.2 in
\cite{ChenHuSongXing}.  The estimate \eref{es2} is proved similarly, which completes the proof.
\end{proof}

\begin{remark}
As a consequence of Theorems 6.4, 6.5 and 6.6, the solution $u$ of both the Skorohod and Stratonivich equations is intermittent in the sense of condition (\ref{eq:fraction-moments-u-k}). 
\end{remark}

\end{document}